\documentclass[12pt,a4paper,reqno]{amsart}
\usepackage{amssymb}
\usepackage{amscd}
\usepackage{enumerate}
\usepackage{graphicx}
\usepackage{siunitx}
\usepackage{tikz-cd}
\usepackage{color}
\usetikzlibrary{arrows}
\numberwithin{equation}{section}

\usepackage{mathtools}

\theoremstyle{plain}

\newtheorem{theorem}{Theorem}[section]
\newtheorem{proposition}[theorem]{Proposition}
\newtheorem{lemma}[theorem]{Lemma}
\newtheorem{corollary}[theorem]{Corollary}

\theoremstyle{definition}

\newtheorem{remark}[theorem]{Remark}

\newcommand\E{\mathbf{E}}
\newcommand\Var{\mathbf{Var}}
\newcommand\Cov{\mathbf{Cov}}
\newcommand\R{\mathbb{R}}
\newcommand\Z{\mathbb{Z}}
\newcommand\N{\mathbb{N}}

\newcommand\C{\mathbb{C}}

\newcommand\eps{\varepsilon}

\newcommand\energy{\operatorname{Energy}}

\renewcommand\det{\operatorname{det}}
\newcommand\adj{\operatorname{adj}}

\newcommand\TV{\operatorname{TV}}
\newcommand\Expect{\operatorname{Integral}}
\newcommand\Disjoint{{\biguplus}}

\begin{document}

\title[Multilinear Kakeya, restriction, and oscillatory integrals]{Sharp bounds for multilinear curved Kakeya, restriction and oscillatory integral estimates away from the endpoint}

\author{Terence Tao}
\address{UCLA Department of Mathematics, Los Angeles, CA 90095-1555.}
\email{tao@math.ucla.edu}



\begin{abstract}  We revisit the multilinear Kakeya, curved Kakeya, restriction, and oscillatory integral estimates that were obtained in a paper of Bennett, Carbery, and the author using a heat flow monotonicity method applied to a fractional Cartesian product, together with induction on scales arguments.  Many of these estimates contained losses of the form $R^\eps$ (or $\log^{O(1)} R$) for some scale factor $R$.  By further developing the heat flow method, and applying it directly for the first time to the multilinear curved Kakeya and restriction settings, we are able to eliminate these losses, as long as the exponent $p$ stays away from the endpoint.  In particular, we establish global multilinear restriction estimates away from the endpoint, without any curvature hypotheses on the hypersurfaces.
\end{abstract}

\maketitle


\section{Introduction}

Throughout this paper, we fix a natural number $d \geq 2$.  
We use the asymptotic notation $X \lesssim Y$, $Y \gtrsim X$, or $X = O(Y)$ to denote the estimate $|X| \leq CY$ for some constant $C$ depending only on $d$; we also abbreviate $X \lesssim Y \lesssim X$ as $X \sim Y$.  If we need the implied constant $C$ to depend on additional parameters, we indicate this by subscripts, thus for instance $X \lesssim_A Y$ denotes the estimate $|X| \leq C_A Y$ for some $C_A$ depending on $d$ and $A$.  

\subsection{Multilinear Kakeya estimates}

The multilinear Kakeya estimate from \cite{bct}, \cite{guth} can be stated as follows.  For any natural number $n \in \N \coloneqq \{0,1,2,\dots\}$, we let $[n]$ denote the index set\footnote{We will also use brackets to denote dependence of an element of a function space on a time parameter $t$ or a spatial parameter $x$, for instance $w[t,x] \in L^\infty(\Omega \to \R)$ might be a weight $\omega \mapsto w[t,x](\omega)$ on a domain $\Omega$ that varies with the parameters $t,x$.  We hope that this conflict of notation will not cause undue confusion, as the symbols $t,x$ will not be used in this paper to denote natural numbers.}
$$ [n] \coloneqq \{ j \in \N: 1 \leq j \leq n \} = \{1,\dots,n\}.$$

\begin{theorem}[Multilinear Kakeya estimate]\label{mlk}  Let $\delta > 0$ be a radius.  For each $j \in [d]$, let $\mathbb{T}_j$ denote a finite family of infinite tubes $T_j$ in $\R^d$ of radius $\delta$.  Assume the following\footnote{For this theorem we only have the single axiom of transversality, but in subsequent theorems we will have multiple axioms, one of which will be an analogue of this transversality axiom.} axiom:
\begin{itemize}
\item[(i)] (Transversality) whenever $T_j \in \mathbb{T}_j$ is oriented in the direction of a unit vector $n_j$ for $j \in [d]$, we have
\begin{equation}\label{kak-trans}
 \left|\bigwedge_{j \in [d]} n_j\right| \geq A^{-1}
\end{equation}
for some $A>0$, where we use the usual Euclidean norm on the wedge product $\bigwedge^d \R^d$.  
\end{itemize}
Then, for any $p \geq \frac{1}{d-1}$, one has
\begin{equation}\label{gauss}
 \left\| \prod_{j \in [d]} \sum_{T_j \in \mathbb{T}_j} 1_{T_j} \right\|_{L^p(\R^d)} \lesssim_{A,p} \delta^{\frac{d}{p}} \prod_{j \in [d]} \# \mathbb{T}_j.
\end{equation}
where $L^p(\R^d)$ are the usual Lebesgue norms with respect to Lebesgue measure, $1_{T_j}$ denotes the indicator function of $T_j$, and $\# \mathbb{T}_j$ denotes the cardinality of $\mathbb{T}_j$.
\end{theorem}

\begin{remark}
The exponent $\frac{d}{p}$ in $\delta^{\frac{d}{p}}$ is optimal, as can be seen by considering the case when each $\mathbb{T}_j$ consists of a single tube passing through the origin; one can also derive from scaling considerations that this is the only possible exponent that makes \eqref{gauss} valid.  Estimates of this form have a number of applications; for instance, they were used in the first solution of the Vinogradov Main Conjecture in \cite{bdg}.
\end{remark}
 
The endpoint case $p=\frac{1}{d-1}$ of this inequality is the most difficult, and was only established for general dimension in \cite{guth} using techniques from algebraic topology related to the polynomial ham sandwich theorem; see also \cite{bg} for an extension to the case of tubular neighbourhoods of algebraic curves, and \cite{cv} for an alternate proof based on the Borsuk-Ulam theorem.  The non-endpoint cases $p>\frac{1}{d-1}$ of Theorem \ref{mlk} were established previously in \cite{bct} using a heat flow monotonicity method, which we now briefly review here (using slightly different notation and normalisations from \cite{bct}).  We rename the family $\mathbb{T}_j$ of tubes as $\Omega_j$, which we now endow with counting measure $\mu_j$.  The elements of $\Omega_j$ we now rename as $\omega_j$ instead of $T_j$, and the parameter $\delta$ we now rename as a time parameter $t$.  Each tube $\omega_j$ in $T_j$ can now be written in the form $\phi_j[\cdot](\omega_j)^{-1}( B^{d-1}(0,t) )$, where $B^{d-1}(0,t)$ is the ball of radius $t$ centred at the origin in $\R^{d-1}$, and $\phi_j[\cdot](\omega_j) \colon \R^d \to \R^{d-1}$ is an affine-linear map of the form
\begin{equation}\label{phui-form}
\phi_j[x](\omega_j) = x B_j(\omega_j) - v_j(\omega_j)
\end{equation}
for some vector $v_j(\omega_j) \in \R^{d-1}$ and some orthogonal matrix $B_j(\omega_j) \in \R^{d \times d-1}$.  Here and in the sequel we use $\R^{D_1 \times D_2}$ to denote the space of real matrices with $D_1$ rows and $D_2$ columns, and identify elements of $\R^D$ with row vectors, thus $\R^D \equiv \R^{1 \times D}$.  By hypothesis, the transversality condition \eqref{kak-trans} then holds whenever each $n_j \in \R^d$ is a unit vector in the left null space of $B_j(\omega_j)$ (i.e., $n_j B_j(\omega_j) = 0$) for some $\omega_j \in \Omega_j$.

The estimate \eqref{gauss} can then be rewritten as
$$ t^{-d} \int_{\R^d} \prod_{j \in [d]} \left(\int_{\Omega_j} 1_{B^{d-1}(0,t)}(\phi_j[x])\ d\mu_j\right)^p\ dx \lesssim_{A,p} \prod_{j \in [d]} \mu_j(\Omega_j)^p.$$
We now introduce the gaussian weights
$$ w_j[t,x] \coloneqq \gamma_t( \phi_j[x] )$$
where 
\begin{equation}\label{gammat}
\gamma_t(y) \coloneqq \exp( -yy^T / t^2 ) = \exp( - |y|^2/t^2)
\end{equation}
and $y^T$ denotes the transpose of $y$.  If we let $\mu_w$ denote the weighting of a measure $\mu$ on a space $\Omega$ by a non-negative weight $w \in L^1(\Omega,\mu)$, thus
$$ \int_{\Omega} f\ d\mu_w \coloneqq \int_\Omega f w\ d\mu$$
for any bounded measurable $f \colon \Omega \to \R$ (or equivalently $d\mu_w = w\ d\mu$), then we have
$$ \int_{\Omega_j} 1_{B^{d-1}(0,t)}(\phi_j[x])\ d\mu_j \lesssim (\mu_j)_{w_j[t,x]}(\Omega_j)$$
and it now suffices to establish the claim
\begin{equation}\label{tosh}
 t^{-d} \int_{\R^d} \prod_{j \in [d]} (\mu_j)_{w_j[t,x]}(\Omega_j)^p\ dx \lesssim_{A,p} \prod_{j \in [d]} \mu_j(\Omega_j)^p.
\end{equation}
It is not difficult to show that this claim is asymptotically true in the limit $t \to +\infty$ (this basically corresponds, after rescaling, to the case where all the tubes $T_j$ go through the origin).  The strategy in \cite{bct} is then to show that the left-hand side of \eqref{tosh} (or more precisely, a modified quantity comparable to this left-hand side) is monotone non-decreasing in $t$.

To motivate the argument, let us first restrict to the case when $p$ is a natural number.  At first glance this appears to be a severe restriction (in particular, one now cannot get close to the endpoint $p = \frac{1}{d-1}$ when $d \geq 3$); however, as observed in \cite{bct}, many algebraic manipulations that are valid for natural number $p$ can be usefully ``extrapolated'' to the case of fractional $p$.  In this case we can rewrite the product $\prod_{j \in [d]} (\mu_j)_{w_j[t,x]}^p\ dx$ as an integral over a certain product space.  Namely, we introduce the $d$-tuple of spaces
$$ \vec \Omega \coloneqq (\Omega_1,\dots,\Omega_d),$$
the $d$-tuple of measures
$$ \vec \mu \coloneqq (\mu_1,\dots,\mu_d)$$
and the $d$-tuple of (natural number) exponents
$$ \vec p \coloneqq (p_1,\dots,p_d)$$
where in our case we have $p_j = p$ for all $j \in [d]$, and then define the disjoint union $\Disjoint \vec \Omega$ of the spaces $\Omega_j$ as
\begin{align*}
\Disjoint \vec \Omega &\coloneqq \bigcup_{j \in [d]} \{j\} \times \Omega_j \\
&= \{ (j,\omega_j): j \in [d], \omega_j \in \Omega_j \}.
\end{align*}
We can view $\vec \mu$ (by abuse of notation) as a measure on this disjoint union by the formula
$$ \int_{\Disjoint \vec \Omega} f\ d\vec \mu = \sum_{j \in [d]} \int_{\Omega_j} f_j\ d\mu_j$$
for all bounded measurable $f \colon \Disjoint \vec \Omega \to \R$, where each $f_j \colon \Omega_j \to \R$ is a component of $f$, defined by
$$ f_j(\omega_j) \coloneqq f(j,\omega_j)$$
for $\omega_j \in \Omega_j$.  In particular, we can view all the weights $w_j[t,x]$ as the components of a single concatenated weight $w[t,x]: \Disjoint \vec \Omega \to \R$ defined by
$$ w[t,x](j,\omega_j) \coloneqq w_j[t,x](\omega_j),$$
and we have
$$ (\vec \mu_{w[t,x]})_j = (\mu_j)_{w_j[t,x]}.$$
Next, we define the product space
$$ \vec \Omega^{\vec p} \coloneqq \prod_{j \in [d]} \Omega_j^{p_j}$$
and product measure
$$ \vec \mu^{\vec p} \coloneqq \prod_{j \in [d]} \mu_j^{p_j}$$
with the corresponding weighted measure
$$ \vec \mu_{w[t,x]}^{\vec p} \coloneqq \prod_{j \in [d]} (\mu_j)_{w_j[t,x]}^{p_j}.$$
If we make the multi-index notation
\begin{equation}\label{multi}
\vec \mu(\vec \Omega)^{\vec p} \coloneqq \prod_{j \in [d]} \mu_j(\Omega_j)^{p_j}
\end{equation}
and observe from Fubini's theorem that
$$ \vec \mu(\vec \Omega)^{\vec p} = \vec \mu^{\vec p}( \vec \Omega^{\vec p})$$
then we may rewrite \eqref{tosh} as
$$
 t^{-d} \int_{\R^d} \int_{\vec \Omega^{\vec p}}\ d\vec \mu_{w[t,x]}^{\vec p} dx \lesssim_{A,p} \vec \mu(\vec \Omega)^{\vec p}.
$$
In \cite{bct}, this estimate was (essentially) established by first replacing it with an equivalent estimate
$$
 t^{-d} \int_{\R^d} \int_{\vec \Omega^{\vec p}} \det(M) \ d\vec \mu_{w[t,x]}^{\vec p} dx \lesssim_{A,p} \vec \mu(\vec \Omega)^{\vec p}.
$$
where $M: \vec \Omega^{\vec p} \to \R^{d \times d}$ is a certain matrix-valued function\footnote{This matrix was denoted $\mathbf{A}_*$ in \cite{bct}; its analogue in the current paper is introduced in \eqref{M-def} (for the curved Kakeya problem) or \eqref{mdef} (for the restriction problem).} that is close to a constant multiple $p(d-1) I_d$ of the identity matrix $I_d$ (after performing some initial reductions to make the normal vectors $n_j$ in \eqref{kak-trans} sufficiently close to the standard basis vectors $e_j$, where ``sufficiently close'' depends on how close $p$ is to the endpoint $\frac{1}{d-1}$).  The main reason for this weight is because the inverse matrix $M^{-1}$ naturally will appear in subsequent calculations, and the fact that this inverse $M^{-1}$ is not a polynomial function of the coefficients of $M$ will cause severe difficulties (particularly when one attempts to extrapolate to the case of non-integer $p$).  However, from the standard identity
\begin{equation}\label{Adj}
\det(M) I_d = \adj(M) M = M \adj(M)
\end{equation}
where $\adj(M)$ is the adjugate matrix of $M$, the presence of the weight $\det(M)$ will convert this inverse into an expression which is polynomial in the coefficients of $M$, which will allow one to extrapolate the calculations to non-integer values of $p$.

After some integration by parts and linear algebra manipulations, it was then shown in \cite{bct} that the derivative
$$ \partial_t \left( t^{-d} \int_{\R^d} \int_{\vec \Omega^{\vec p}} \det(M) \ d\vec \mu_{w[t,x]}^{\vec p} dx \right)$$
was non-negative, with the proof being of such a form that this monotonicity could also be extrapolated to the case of non-integer $p$ in the non-endpoint range $p > \frac{1}{d-1}$.  This concluded the proof of the non-endpoint case of Theorem \ref{mlk} in \cite{bct}.

In this paper, we revisit the arguments in \cite{bct}, now treating the case of fractional $p$ directly rather than by working primarily with natural number $p$ and appealing to extrapolation results at the end of the argument to extend to fractional $p$.  With this slightly different perspective, we now view $(\vec \Omega^{\vec p}, \vec \mu_{w[t,x]}^{\vec p})$ for fractional $p$ as a ``virtual measure space'', and expressions such as $\det(M)$ as ``virtual functions''.  As it turns out, even though these are not classical functions on a classical measure space, one can still define virtual functions in an abstract algebraic fashion, and construct a well-defined ``virtual integral'' of such functions on the virtual measure space $(\vec \Omega^{\vec p}, \vec \mu_{w[t,x]}^{\vec p})$.  Several of the laws of calculus (e.g., differentiation under the integral sign, or the change of variables formula) can then be rigorously established for this virtual integral, and the manipulations in \cite{bct} can now be done in this formalism directly for non-integer $p$ without appealing to any extrapolation theory.

One advantage of this perspective is that it extends readily to the variable-coefficient setting, allowing one for the first time to recover estimates in this setting with optimal bounds (except for a polynomial blowup as the exponent $p$ approaches the endpoint $\frac{1}{d-1}$).  To state these estimates, we introduce some further notation.  We use $\nabla_x$ to denote the gradient operator in $\R^d$,
$$ \nabla_x \coloneqq ( \partial_{x_1}, \dots, \partial_{x_d} )$$
where $\partial_{x_j}$ denotes partial differentiation in the direction of the $j^{\mathrm{th}}$ standard coordinate $x_j$ of $\R^d$.  Note that we view $\nabla$ as a row vector of differential operators; we shall also frequently use the transpose
$$\nabla_x^T = ( \partial_{x_1}, \dots, \partial_{x_d} )^T.$$
Thus for instance if $\phi: \R^d \to \R^{d-1}$ is a differentiable map, then for any $x \in \R^d$, $\nabla_x^T \phi(x) \in \R^{d \times d-1}$ is a $d \times d-1$ matrix whose rows are $\partial_{x_1} \phi(x), \dots, \partial_{x_d} \phi(x)$.  We also will sometimes need higher order derivatives $\nabla_x^{\otimes k} \otimes \phi$ of a map such as $\phi: \R^d \to \R^{d-1}$, which will be a $d \times \dots \times d \times d-1$ tensor (with $k$ copies of $d$) whose coefficients are of the form $\partial_{x_{i_1}} \dots \partial_{x_{i_k}} \phi_j(x)$ for $i_1,\dots,i_k \in [d]$ and $j \in [d-1]$, with $\phi_j$ denoting the coefficients of $\phi$.

Our main result (which answers in the affirmative a question\footnote{In \cite{bct}, the functions $\phi_j$ are taken to be of the specific form $\phi_j[x]\colon (\xi_0, \omega) \mapsto \nabla_\xi \Phi_j( x, \xi_0 ) - \omega$ for some phase function $\Phi_j \colon V \times \R^{d-1} \to \R$, with $\Omega_j = \R^{d-1} \times\R^{d-1}$ being parameterised as $(\xi_0,\omega)$, as this is the case of interest in applications to oscillatory integrals; however, the analysis in that paper extends without difficulty to the more general setting considered in Theorem \ref{curv-kak}.} in \cite[Remark 6.6]{bct}) is as follows.

\begin{theorem}[Curved multilinear Kakeya estimate]\label{curv-kak}  Let $\frac{1}{d-1} < p \leq \infty$ be an exponent, and let $A \geq 2$ be a parameter.  Let $V \subset B^d(0,A)$ be an open subset of the ball $B^d(0,A)$ of radius $A$ centred at the origin in $\R^d$.  Let $(\vec \Omega, \vec \mu)$ be a $d$-tuple of finite measure spaces $(\Omega_j,\mu_j)$, and for each $j \in [d]$, let $\phi_j \colon V \times \Omega_j \to \R^{d-1}$ be a measurable map (denoted $(x,\omega_j) \mapsto \phi_j[x](\omega_j)$) obeying the following axioms:
\begin{itemize}
\item[(i)] (Regularity) For each $j \in [d]$ and $\omega_j \in \Omega_j$, the map $\phi_j[\cdot](\omega_j) \colon x \mapsto \phi_j[x](\omega_j)$ is a $C^2$ map with norm bounds
\begin{equation}\label{c2}
 |\nabla_x^{\otimes a} \otimes \phi_j[x](\omega_j)| \leq A
\end{equation}
for all $x \in V$ and $a=1,2$, where we use the usual Euclidean norm on tensor product spaces such as $\R^{d \times d \times d-1}$.
\item[(ii)]  (Submersion)  For each $j \in [d]$, $\omega_j \in \Omega_j$, and $x \in V$, the derivative matrix $\nabla_x^T \phi_j[x](\omega_j) \in \R^{d \times d-1}$ is of full rank, with the $d-1$ non-trivial singular values lying between $A^{-1}$ and $A$. In particular, $\phi_j[\cdot](\omega_j) \colon V \to \R^{d-1}$ is a $C^2$ submersion.
\item[(iii)]  (Transversality) For any $x \in U$ and $\omega_j \in \Omega_j$ for $j \in [d]$, we have the lower bound
$$ \left|\bigwedge_{j \in [d]} n_j\right| \geq A^{-1}$$
whenever $n_j \in \R^d$ is a unit vector in the left null space of $\nabla_x^T \phi_j[x](\omega_j) \in \R^{d \times d-1}$.
\end{itemize}
Let $V_{1/A}$ be the subset of $V$ defined by
\begin{equation}\label{uua}
V_{1/A} \coloneqq \{ x \in V:  B^d(x,1/A) \subset V \}.
\end{equation}
Then for any $0 < t \leq A$, one has
\begin{equation}\label{dop}
\begin{split}
& \left\| \prod_{j \in [d]} \int_{\Omega_j} 1_{B^{d-1}(0,t)}( \phi_j[x](\omega_j) )\ d\mu_j(\omega_j) \right\|_{L^p(V_{1/A})} \\
&\quad \leq A^{O(1)} \left(d-1-\frac{1}{p}\right)^{-O(1)} t^{\frac{d}{p}} \prod_{j \in [d]} \mu_j(\Omega_j).
\end{split}
\end{equation}
\end{theorem}

We prove this theorem in Section \ref{curv-sec}, after setting up the formalism of virtual integration in Section \ref{fractional}.  The proof largely follows the proof of the multilinear Kakeya estimate in \cite{bct} as sketched above; the variable coefficient nature of $\phi_j$ introduces some additional lower order error terms, mostly arising from the fact that the virtual matrix $M = M[x]$ now depends (smoothly) on the spatial parameter $x$, but these error terms can be easily handled by appealing to an induction hypothesis in which the exponent $\frac{d}{p}$ in the $t^{\frac{d}{p}}$ factor is replaced by a smaller quantity.  We remark that the algebraic topology arguments of \cite{guth}, \cite{bg}, \cite{cv} already establish Theorem \ref{curv-kak} in the case when the $\phi_j$ are polynomials of bounded degree, but do not seem to extend readily to the case of non-algebraic maps $\phi_j$.

\begin{remark}
By restricting attention to maps $\phi_j$ of the affine-linear form \eqref{phui-form}, we see that Theorem \ref{curv-kak} implies a local version of Theorem \ref{mlk} in which the $L^p$ norm is restricted to (say) $B^d(0,1)$, and $\delta$ is also restricted to be at most $1$; these restrictions can then be easily lifted by a scaling argument, so that the full strength of Theorem \ref{mlk} becomes a corollary of Theorem \ref{curv-kak}.  This implication also shows that the estimate \eqref{dop} is sharp except for the factors $A^{O(1)} (d-1-\frac{1}{p})^{-O(1)}$.
\end{remark}

\begin{remark}
The methods of \cite[\S 6]{bct} allow one to establish a weaker form of \eqref{dop} in which an additional logarithmic factor $\log^{O(1)}(2 + \frac{1}{t})$ is inserted.  Roughly speaking (and following the notation from \cite[\S 5]{aspects}), if one lets $C_{\mathrm{CurvyKak}}(t)$ denote the best constant replacing the quantity $A^{O(1)} (d-1-\frac{1}{p})^{-O(1)}$ in \eqref{dop}, it is not difficult to show a recursive inequality of the form
$$ \mathcal{C}_{\mathrm{CurvyKak}}(t) \lesssim \mathcal{C}_{\mathrm{Kak}}(\sqrt{t}) \mathcal{C}_{\mathrm{CurvyKak}}(\sqrt{t}) $$
(cf. \cite[Proposition 5.3]{aspects}), where $\mathcal{C}_{\mathrm{Kak}}$ is an analogue of $C_{\mathrm{CurvyKak}}$ in the case when the maps $\phi_j$ are affine-linear.  Theorem \ref{mlk} implies that $\mathcal{C}_{\mathrm{Kak}}(\sqrt{t})$ is bounded, and the claim now follows by a standard iteration argument.  We leave the details to the interested reader. We remark that this argument can also be combined with a similar argument in \cite{guth-easy} to derive a similar logarithmically lossy bound using the Loomis-Whitney inequality \cite{loomis} in place of the multilinear Kakeya inequality (if one carefully optimises the quantitative bounds arising from the induction on scales argument used in \cite{guth-easy},following the analysis in \cite[\S 4.3]{aspects}).
\end{remark}

\begin{remark}
The endpoint $p = \frac{1}{d-1}$ of Theorem \ref{curv-kak} (deleting the $(d-1-\frac{1}{p})^{-O(1)}$ factor) is not directly addressed by our methods; however, by setting $p = \frac{1}{d-1} + \frac{1}{\log(2 + \frac{1}{t})}$ and applying H\"older's inequality one can recover the endpoint result with a logarithmic loss, thus matching the previously known results in this case.  When $d=2$, the endpoint $p=1$ of Theorem \ref{curv-kak} can be established by direct calculation; in view of this, as well as the constant coefficient result in \cite{guth}, it is natural to conjecture that in higher dimensions Theorem \ref{curv-kak} holds at the endpoint (again deleting the $(d-1-\frac{1}{p})^{-O(1)}$ factor, of course, and perhaps imposing some additional regularity hypotheses on the $\phi_j$).  However this does not seem to be achievable purely from the techniques used in the current paper, though there is still the possibility that some other, even more precise induction-on-scales type argument may still be effective at the endpoint.
\end{remark}

\subsection{Multilinear restriction estimates}

We now adapt these arguments to the oscillatory integral setting, beginning with the ``constant coefficient'' case of multilinear restriction estimates.  Let $S_1,\dots,S_d \subset \R^d$ be $d$ smooth hypersurfaces in $\R^d$, which for simplicity we will take to be graphs
$$ S_j = \{ (\xi, h_j(\xi)): \xi \in U_j \}$$
for some open bounded subsets $U_j$ of $\R^{d-1}$, where each $h_j: U_j \to \R$ is a smooth compactly supported function.  We assume the transversality condition
\begin{equation}\label{transverse}
 \left| \bigwedge_{j \in [d]} n_j(\xi_j)\right| \geq A^{-1} 
\end{equation}
for some constant $A>0$ and all $\xi_j \in U_j$, $j \in [d]$, where 
\begin{equation}\label{njxi}
 n_j(\xi_j) \coloneqq (-\nabla_\xi h_j(\xi_j),1)
\end{equation}
is the (non-unit) normal to $S_j$ at $(\xi_j,h_j(\xi_j))$, and $\nabla_\xi = (\partial_{\xi_1},\dots,\partial_{\xi_{d-1}})$ is the gradient in the $\xi$ variable; geometrically, this means that if $v_j$ is a unit normal to a point in $S_j$ for each $j \in [d]$, then the $v_1,\dots,v_d$ never lie close to a hyperplane through the origin.  In applications, four examples of $S_j$ are of particular interest:

\begin{itemize}
\item (Loomis-Whitney case) $U_j \subset \R^{d-1}$, and each $h_j(\xi) = \xi \cdot v_j + \tau_j$ is an affine-linear function of $\xi$ on $U_j$.
\item (Paraboloid case) $U_j \subset \R^{d-1}$, and $h_j(\xi) = |\xi|^2$ for all $\xi \in U_j$.
\item (Cone case) $U_j \subset \{ \xi \in \R^{d-1}: 1 \leq |\xi| \leq 2 \}$, and $h_j(\xi) = |\xi|$ for all $\xi \in U_j$.
\item (Sphere case) $U_j \subset \{ \xi \in \R^{d-1}: |\xi| \leq 1-\delta \}$ for some $\delta>0$, and $h_j(\xi) = \pm (1 - |\xi|^2)$ for all $\xi \in U_j$.
\end{itemize}
In each case one of course would need to impose additional separation conditions on the $U_i$ (or on the velocities $v_j$ in the Loomis-Whitney case) in order to obtain the transversality condition \eqref{transverse}.

For each $j$ and each $f_j \in L^1(U_j \to \C)$, we define the extension operator ${\mathcal E}_j f_j \in L^\infty(\R^d \to \C)$ by the formula
\begin{equation}\label{ext-def}
 {\mathcal E}_j f_j( x', x_d ) \coloneqq \int_{U_j} e^{2\pi i (x' \xi^T + x_d h_j(\xi))} f_j(\xi)\ d\xi
\end{equation}
for all $x' \in \R^{d-1}$ and $x_d \in \R$.  Using the $d-1$-dimensional Fourier transform
$$ \hat F(\xi) \coloneqq \int_{\R^{d-1}} F(x) e^{-2\pi i x' \xi^T}\ dx'$$
and its inverse
$$ \check f(x') \coloneqq \int_{\R^{d-1}} f(\xi) e^{2 \pi i x' \xi^T}\ d\xi$$
one can view $x' \mapsto {\mathcal E}_j f_j( x', x_d )$ as the inverse Fourier transform of $\xi \mapsto e^{2\pi i x_d h_j(\xi)} f_j(\xi)$. In Section \ref{rest-sec}, we will establish the following claim.

\begin{theorem}[Multilinear restriction theorem]\label{lmrc}
Let $\frac{1}{d-1} < p \leq \infty$ be an exponent, and let $A \geq 2$ be a parameter.  Let $M$ be a sufficiently large natural number, depending only on $d$.  For $j \in [d]$, let $U_j$ be an open subset of $B^{d-1}(0,A)$, and let $h_j \colon U_j \to \R$ be a smooth function obeying the following axioms:
\begin{itemize}
\item[(i)] (Regularity)  For each $j \in [d]$ and $\xi \in U_j$, one has
\begin{equation}\label{regular}
 |\nabla_\xi^{\otimes m} \otimes  h_j(\xi)| \leq A
\end{equation}
for all $1 \leq m \leq M$.
\item[(ii)]  (Transversality)  One has \eqref{transverse} whenever $\xi_j \in U_j$ for $j \in [d]$.
\end{itemize}
Let $U_{j,1/A} \subset U_j$ be the sets
\begin{equation}\label{ujdef}
U_{j,1/A} \coloneqq \{ \xi \in U_j: B^{d-1}(\xi,1/A) \subset U_j \}.
\end{equation}
Then one has
$$
\left\| \prod_{j \in [d]} {\mathcal E}_j f_j \right\|_{L^{2p}(\R^d)} 
\leq A^{O(1)} \left(d-1-\frac{1}{p}\right)^{-O(1)} \prod_{j \in [d]} \|f_j \|_{L^2(U_{j,1/A})}$$
for any $f_j \in L^2(U_{j,1/A} \to \C)$, $j \in [d]$, extended by zero outside of $U_{j,1/A}$.
\end{theorem}

\begin{remark}
In this paper we do not specify the precise amount of regularity $M$ required in the hypotheses.  In view of previous results in \cite{bct}, it is likely that $M$ can be taken to be $2$, but we do not pursue this question here.  The exponent $p$ here corresponds to what would be written as $p/2$ in other literature, but we have chosen this normalisation to make the multilinear restriction theory align more closely with the multilinear Kakeya theory.
\end{remark}

\begin{remark}
In \cite{bct} (using the optimised quantitative analysis from \cite[\S 4.3]{aspects}), the multilinear Kakeya estimate was used to give\footnote{As it turns out, the original presentation of this argument in \cite{bct} had a slight gap in the case $p<\frac{1}{2}$, but later treatments have addressed the issue; see Appendix \ref{erratum}.} a local version of the above theorem in which the $L^p(\R^d)$ norm was replaced by an $L^p(B(0,R))$ norm for any $R \geq 2$, and an additional loss of $\log^{O(1)} R$ appeared on the right-hand side.   An alternate proof of such a local estimate was later given\footnote{Strictly speaking, the results in \cite{bejenaru} only claim the estimate with a loss of $O_\eps(R^\eps)$ for any $\eps>0$, but it is likely that by optimising the argument as in \cite[\S 4.3]{aspects} one can tighten this loss to a logarithmic loss.} in \cite{bejenaru}, using (a discrete form of) the Loomis-Whitney inequality in place of the multilinear Kakeya estimate, but now requiring a high degree of regularity $M$. 
\end{remark}

\begin{remark}
In the case where the surfaces $S_j$ have curvature uniformly bounded away from zero (so that the Fourier transform of the surfaces exhibits some decay at infinity), Theorem \ref{lmrc} (with a worse dependence of constants on $p$) was essentially obtained in \cite[Lemma A3]{bg}, adapting the ``epsilon removal'' argument from \cite{tao-br}.  It should be possible to also recover Theorem \ref{lmrc} in this case with the polynomial dependence of constants on $p$ from the results of \cite{bct} combined with the epsilon removal argument of Bourgain \cite{borg} (as detailed in \cite[\S 6]{TV}), and keeping careful track of the bounds, but we do not attempt to do so here.  In contrast, our proof of Theorem \ref{lmrc} is insensitive to curvature hypotheses, and in particular gives uniform bounds as the curvature degenerates to zero, which will be useful in our next application to variable-coefficient oscillatory integral estimates.  
\end{remark}

\begin{remark}
The endpoint $p = \frac{1}{d-1}$ of the above theorem (with the $(d-1-\frac{1}{p})^{-O(1)}$ factor deleted) remains an open problem, except in the two-dimensional setting $d=2$, $p=1$ in which case the claim can be verified easily from Plancherel's theorem.  The usual extrapolation arguments applied to this theorem only give an estimate in which the endpoint space $L^{\frac{2d}{d-1}}$ is replaced by an Orlicz norm $L^{\frac{2d}{d-1}} \log^{O(1)} L$.
\end{remark}

Our proof of Theorem \ref{lmrc} also gives a more general ``locally averaged'' estimate, which does not appear to have been explicitly observed previously even in the local setting.  For any function $f \in L^2_{\mathrm{loc}}(\R^d \to \C)$ and any radius $r>0$, define the local energy $\energy_{r}[f] \colon \R^d \to \R^+$ by the formula
\begin{equation}\label{Avg-def}
 \energy_{r}[f](x',x_d) \coloneqq \int_{\R^{d-1}} \rho_{x',r}(y') |f(y',x_d)|^2\ dy',
\end{equation}
where 
\begin{equation}\label{rhoar}
 \rho_{x',r}(y') \coloneqq \left\langle \frac{y'-x'}{r} \right\rangle^{-10d^2} 
\end{equation}
and we use the Japanese bracket
$$ \langle x \rangle \coloneqq (1 + |x|^2)^{1/2}.$$
The exponent $10d^2$ here is somewhat arbitrary and could be replaced by any other large quantity; we fix it as $10d^2$ here for sake of concreteness.  Informally, $\energy_{r}[f](x',x_d)$ measures the amount of energy of $f$ passing through (or near) the disk $B^{d-1}(x',r) \times \{x_d\}$.

\begin{theorem}[Locally averaged multilinear restriction theorem]\label{lamrc}  With the hypotheses of Theorem \ref{lmrc}, one has
\begin{equation}\label{dp}
\left\| \prod_{j \in [d]} \energy_{r}[{\mathcal E}_j f_j] \right\|_{L^{p}(\R^d)} \leq A^{O(1)} \left(d-1 - \frac{1}{p}\right)^{O(1)} r^{\frac{d}{p}} \prod_{j \in [d]} \| f_j \|_{L^2(U_{j,1/A})}^2
\end{equation}
whenever $f_j \in L^2(U_{j,1/A} \to \C)$ for $j \in [d]$, and $r \geq 1$.
\end{theorem}

\begin{remark}
The exponent $\frac{d}{p}$ on the right-hand side of \eqref{dp} is best possible, as can be seen by testing in the case when all the $f_j$ are non-zero bump functions.  Standard calculations then show that
$$ \energy_{r}[{\mathcal E}_j f_j](t,x) \gtrsim A^{-O(1)}$$
whenever $|x|, |t| \leq A^{-C} r$ for a sufficiently large constant $C$, which shows that \eqref{dp} cannot hold with the exponent $\frac{d}{p}$ replaced by any larger exponent.  We also remark that quantities similar to that appearing in the left-hand side of \eqref{dp} have appeared in the literature on decoupling theorems such as \cite{bd}, \cite{bd-study}, \cite{bdg}; indeed, this entire paper was inspired by the induction on scales arguments appearing in that literature.
\end{remark}

\begin{remark}
The $r=1$ case of Theorem \ref{lamrc} easily implies Theorem \ref{lmrc}.  Indeed, for any fixed $x_d \in \R$ and $j \in [d]$, the Fourier transform $f_j(\xi) e^{2\pi i x_d h_j(\xi)}$ of the function $x' \mapsto {\mathcal E}_j f_j(x',x_d)$ is supported in $B^{d-1}(0,A)$.  Factoring out a function $\varphi(\cdot/A)$, where $\varphi$ is a non-negative real even Schwartz function with Fourier transform supported on $B^{d-1}(0,1)$, and has Fourier transform everywhere non-negative, we obtain a reproducing formula
$$ {\mathcal E}_j f_j(x',x_d) = \int_{\R^d} {\mathcal E}_j \tilde f_j(x'-A^{-1} y',x_d) \check \varphi(y')\ dy'$$
where $\tilde f_j(\xi) \coloneqq f_j(\xi) / \varphi(\xi/A)$.  By Cauchy-Schwarz, this implies the pointwise bound
$$ {\mathcal E}_j f_j(x',x_d) \lesssim A^{O(1)} \energy_1[{\mathcal E}_j \tilde f_j](x',x_d)^{1/2}$$
and we now see that Theorem \ref{lmrc} follows from the $r=1$ case of Theorem \ref{lamrc}.
\end{remark}

\begin{remark}
The $p=\infty$ case of Theorem \ref{lamrc} follows easily from Plancherel's theorem; the difficulty is thus with small values of $p$.
\end{remark}

The induction on scales arguments in \cite{bct}, \cite{bejenaru} can be adapted without too much difficulty to obtain a local version 
\begin{equation}\label{from}
\left\| \prod_{j \in [d]} \energy_{r}[{\mathcal E}_j f_j] \right\|_{L^{p}(B^d(x_0,R))} \lesssim_A \log^{O_A(1)} \langle R \rangle r^{\frac{d}{p}} \prod_{j \in [d]} \| f_j \|_{L^2(U_{j,1/A})}^2
\end{equation}
of Theorem \ref{lmrc} and any $R \geq 1$.  We sketch the argument as follows, glossing over some minor technical details.  In the argument that follows we allow all implied constants to depend on $A$.  In the spirit of the arguments used to prove decoupling theorems (see e.g., \cite{bd}, \cite{bd-study}, \cite{bdg}) we will induct by ``inflating'' the inner scale $r$ rather than increasing the outer scale $R$.  By interpolation with the easy $p=\infty$ bound it suffices to verify the claim at the endpoint $p = \frac{1}{d-1}$.  For $1 \leq r \leq R$, let $C(R,r)$ denote the best constant in the inequality
$$
r^{-d} \| \prod_{j \in [d]} \energy_{r}[{\mathcal E}_j f_j] \|_{L^{\frac{1}{d-1}}(B^d(x_0,R))}^{\frac{1}{d-1}} \leq C(R,r) \prod_{j \in [d]} \| f_j \|_{L^2(U_{j,1/A})}^{\frac{2}{d-1}}
$$
From Plancherel's theorem it is not difficult to establish the bound
\begin{equation}\label{aprr}
 C(R,R) \lesssim 1.
\end{equation}
We will sketch a proof of the recursive inequality
\begin{equation}\label{pr1}
 C(R,r_1) \lesssim C(R,r_2)
\end{equation}
for $1 \leq r_1 \leq r_2 \leq R$ with $r_2 \lesssim r_1^2$; iterating this bound starting from \eqref{aprr} will give the required bound $C(R,r) \lesssim \log^{O(1)} \langle R \rangle$ for any $1 \leq r \leq R$.

To prove \eqref{pr1}, it will suffice to establish the quasi-monotonicity (or ``ball inflation estimate'')
\begin{align*}
&r_1^{-d} \left\| \prod_{j \in [d]} \energy_{r_1}[{\mathcal E}_j f_j] \right\|_{L^{\frac{1}{d-1}}(B^d(x_0,R))}^{\frac{1}{d-1}} \\
&\quad \lesssim
r_2^{-d} \left\| \prod_{j \in [d]} \energy_{r_2}[{\mathcal E}_j f_j]  \right\|_{L^{\frac{1}{d-1}}(B^d(x_0,R))}^{\frac{1}{d-1}},
\end{align*}
ignoring a technical issue in which the sharp cutoff to $B^d(x_0,R)$ in fact has to be replaced with a weight that decays rapidly outside of this ball.  Covering $B^d(x_0,R)$ by balls of radius $r_2$, and using translation invariance, it (morally) suffices to show the localised estimate
\begin{align*}
&r_1^{-d} \| \prod_{j \in [d]} \energy_{r_1}[{\mathcal E}_j f_j] \|_{L^{\frac{1}{d-1}}(B^d(0,r_2))}^{\frac{1}{d-1}} \\
&\quad \lesssim
 r_2^{-d} \| \prod_{j \in [d]} \energy_{r_2}[{\mathcal E}_j f_j]  \|_{L^{\frac{1}{d-1}}(B^d(0,r_2))}^{\frac{1}{d-1}}.
\end{align*}
By the uncertainty principle, we should be able to (morally) restrict to the case when $f_j$ has Fourier transform supported on $B^{d-1}(0,r_2)$.  By Plancherel's theorem,  the right-hand side is then morally comparable to $\prod_{j \in [d]} \| f_j \|_{L^2(U_j)}$, thus we are now morally reduced to establishing the bound
$$
r_1^{-d} \| \prod_{j \in [d]} \energy_{r_1}[{\mathcal E}_j f_j] \|_{L^{\frac{1}{d-1}}(B^d(0,r_2))}^{\frac{1}{d-1}} \lesssim \prod_{j \in [d]} \| f_j \|_{L^2(U_j)}.$$
We now cover each $U_j$ by boundedly overlapping balls $B^{d-1}( \xi_j, r_1/r_2 )$ of radius $r_1/r_2$, and let $f_j = \sum_{\xi_j} f_{j,\xi_j}$ be a corresponding decomposition of $f_j$ resulting from some associated partition of unity in frequency space.  By hypothesis, $r_1/r_2 \lesssim r_1^{-1}$.  The uncertainty principle (or ``$L^2$ decoupling'') then indicates that the functions ${\mathcal E}_j f_{j,\xi_j}$ are ``almost orthogonal'' as $\xi_j$ varies on balls $B(x,r_1)$, which morally suggests a pointwise inequality of the form
$$\energy_{r_1}[{\mathcal E}_j f_j] \lesssim \sum_{\xi_j} \energy_{r_1}[{\mathcal E}_j f_{j,\xi_j}]$$
(once again, this is not quite true due to certain technicalities that we are ignoring).  Furthermore, stationary phase heuristics suggest that the function ${\mathcal E}_j f_{j,\xi_j}$  propagates in the direction $n_j(\xi_j) + O( r_1/r_2)$ (with $n_j(\xi)$ defined in \eqref{njxi}), which then suggests the ``energy estimate''
$$ \energy_{r_1}[{\mathcal E}_j f_{j,\xi_j}](x)  \lesssim \energy_{r_1}[{\mathcal E}_j f_{j,\xi_j}](x - x_d n_j(\xi_j))$$
for $x_d = O( r_2 )$ (again ignoring technicalities).  As a consequence, we expect a ``wave packet decomposition'' of the form
$$ \energy_{r_1}[{\mathcal E}_j f_{j,\xi_j}](x) \lesssim \sum_{T_j \in {\mathbf T}_{j,\xi_j}} c_{j,T_j} 1_{T_j}(x)$$
for $|x_d| \leq r_2$, where $T_j$ ranges over a finite collection ${\mathbf T}_{j,\xi_j}$ of $r_1 \times r_2$ tubes of the form
$$ T_j = \{ (x',x_d): |x' - x_{T_j} - x_d v_{T_j}| \leq r_1; \quad |t| \leq r_2 \}$$ 
for some initial position $x_{T_j} \in \R^{d-1}$ and some velocity $v_{T_j} = - \nabla_\xi h_j(\xi_j) + O( r_1/r_2)$, and the coefficients $c_{j,T_j}$ are non-negative quantities obeying the Bessel inequality
$$ \sum_{T_j} c_{j,T_j} \lesssim \|f_j \|_{L^2(U_j)}^2.$$
Putting all this together, we are thus reduced to establishing the bound
$$
r_1^{-d} \| \prod_{j \in [d]} \sum_{\xi_j} \sum_{T_j \in {\mathbf T}_{j,\xi_j}} c_{j,T_j} 1_{T_j} \|_{L^{\frac{1}{d-1}}(B^d(0,r_2))}^{\frac{1}{d-1}} \lesssim \prod_{j \in [d]} \sum_{\xi_j} \sum_{T_j \in {\mathbf T}_{j,\xi_j}} c_{j,T_j}.$$
But this follows easily from the endpoint multilinear Kakeya inequality \cite{guth} (or Theorem \ref{mlk}) and the transversality hypothesis \eqref{transverse}.

In order to avoid the $R^\eps$ or $\log^{O(1)} \langle R \rangle$ type losses in the above arguments, we shall use the heat flow monotonicity method on fractional Cartesian products as used in \cite{bct}; roughly speaking, this method allows one to sharpen \eqref{pr1} to something more like
$$ C(R,r_1) \leq \left(1 + O\left( A^{O(1)} \left(r_1^{-c} + (r_2/R)^c\right) \right)\right) C(R,r_2)$$
for some $c>0$ depending only on $d$, and this can be iterated without incurring logarithmic losses.  Strictly speaking, and in analogy with the situation in the multilinear Kakeya setting, one has to first replace the quantity $C(R,r)$ by a more technical variant $\tilde C(R,r)$ comparable to $C(R,r)$ (and involving a certain virtual function $\det M$) before being able to establish this approximate monotonicity formula.  While the original extension operator functions ${\mathcal E}_j f_j$ are oscillatory and thus not directly amenable to monotonicity arguments, for the purposes of computing local energies such as $\energy_r[{\mathcal E}_j f_j](x)$, one can replace these functions by a proxy, namely a Gabor-type transform
$$
 G_{j,r}((x',x_d), \xi_j) \coloneqq r^{-0.9(d-1)} \left| \int_{\R^{d-1}} {\mathcal E}_j f_j( y', x_d ) e^{-2\pi i y' \xi^T_j} \varphi_{x',r^{0.9}}(y')\ dy' \right|^2.
$$
of $f$, where 
\begin{equation}\label{varphis}
\varphi_{x',r^{-0.9}}(y') \coloneqq \varphi\left( \frac{y'-x'}{r^{0.9}} \right)
\end{equation}
and $\varphi$ is a suitably normalised Schwartz function; the scale $r^{0.9}$ at which one takes the Gabor-type transform is chosen to be slightly less than the scale $r$ at which one is computing the local energy in order to recover a $r^{-c}$ factor in certain error term estimates.  Plancherel's theorem allows us to approximately express local energies $\energy_r[{\mathcal E}_j f_j](x)$ as Gaussian integrals against a measure $\mu_{j,r}$ on a certain space $\Omega_j$ weighted by this non-negative weight $G_{j,r}$.  By exploiting the dispersion relation of the extension operator ${\mathcal E}_j$, one can show that for fixed $\xi_j$ this weight $G_{j,r}$ is approximately constant along the direction $n_j(\xi)$ defined in \eqref{njxi}.  Combined with the transversality hypothesis \eqref{transverse}, this places one in a situation similar enough to the multilinear Kakeya situation that one can now establish approximate monotonicity in $r$ (though one also needs to rely heavily on Plancherel's theorem to study how the weights $G_{j,r}$ vary in $r$).  


\subsection{Multilinear oscillatory integral estimates}

Just as the multilinear Kakeya estimate in Theorem \ref{mlk} can be extended to a variable coefficient version in Theorem \ref{curv-kak} (at least away from the endpoint $p = \frac{1}{d-1}$), the multilinear restriction estimate in Theorem \ref{lmrc} can be extended to a variable coefficient version involving oscillatory integral operators\footnote{We have swapped the roles of $x$ and $\xi$ (and also $p$ and $q$) compared to the notation in \cite{bct}, in order to be more consistent with the notation elsewhere in this paper.}
\begin{equation}\label{slam}
S_\lambda^{(j)} f(x) \coloneqq \int_{\R^{d-1}} e^{2\pi i \lambda\Phi_j(x,\xi)} \psi_j(x,\xi) f(\xi) d\xi
\end{equation}
for some phase functions $\Phi_j: \R^d \times \R^{d-1} \to \R$, amplitude functions $\psi_j: \R^d \times \R^{d-1} \to \C$, and scaling parameter $\lambda \geq 1$.  More specifically, we show the following result, conjectured in \cite[Remark 6.3]{bct} (though as with our other results, we do not obtain the endpoint case of this conjecture):

\begin{theorem}[Multilinear oscillatory integral estimate]\label{lmrc-osc}
Let $\frac{1}{d-1} < p \leq \infty$ and $1 \leq q < \infty$ be exponents with 
$$
\frac{1}{q} + \frac{1}{2(d-1)p} < 1.$$
Let $A \geq 2$ be a parameter.  Let $M$ be a sufficiently large natural number, depending only on $d$, and let $V$ be an open subset of $B^d(0,A)$.  For $j \in [d]$, let $U_j$ be an open subset of $B^{d-1}(0,A)$, and $\Phi_j \colon V \times U_j \to \R$ and $\psi_j \colon V \times U_j \to \C$ be smooth functions obeying the following axioms:
\begin{itemize}
\item[(i)] (Regularity)  For each $j \in [d]$, $x \in V$, $\xi \in U_j$, one has
\begin{equation}\label{regular-osc}
 |\nabla^{\otimes m}_x \otimes \nabla^{\otimes m'}_\xi \Phi_j(x,\xi)|,
 |\nabla^{\otimes m}_x \otimes \nabla^{\otimes m'}_\xi \psi_j(x,\xi)| \leq A
\end{equation}
for all $0 \leq m,m' \leq M$.  
\item[(ii)]  (Submersion)  For each $j \in [d]$, $x \in V$, $\xi \in U_j$, the mixed Hessian $\nabla_x^T \nabla_\xi \Phi_j(x,\xi) \in \R^{d \times d-1}$ is of full rank, with the $d-1$ non-trivial singular values between $A^{-1}$ and $A$.
\item[(iii)]  (Transversality) For any $x \in V$, we have the lower bound
$$ |\bigwedge_{j \in [d]} n_j| \geq A^{-1}$$
whenever $n_j \in \R^d$ is a unit vector in the left null space of $\nabla_x^T \nabla_\xi \Phi_j(x,\xi)$ for some $\xi \in U_j$.
\end{itemize}
Then one has
\begin{align*}
&\| \prod_{j \in [d]} S_\lambda^{(j)} f_j \|_{L^{2p}(V_{1/A})} \\
&\quad \leq A^{O(1)} \left(1 - \frac{1}{\min(q,2)} - \frac{1}{2(d-1)p}\right)^{-O(1)} \lambda^{-d/2p} \prod_{j \in [d]} \|f_j \|_{L^q(U_{j,1/A})}
\end{align*}
for any $\lambda \geq 1$ and $f_j \in L^2(U_{j,1/A})$, $j \in [d]$, extended by zero outside of $U_{j,1/A}$, where the operators $S_\lambda^{(j)}$ are defined by \eqref{slam}, the sets $U_{j,1/A}$ are defined by \eqref{ujdef}, and $V_{1/A} \subset V$ is similarly defined by \eqref{uua}.
\end{theorem}

\begin{remark}
As noted in \cite{bct}, Theorem \ref{lmrc} can be deduced by a rescaling argument from the special case of Theorem \ref{lmrc-osc} when the phases $\Phi_j(x,\xi)$ take the form $\Phi_j((x',x_d),\xi) = x' \xi^T + x_d h_j(\xi)$; we leave the details to the interested reader.  The most interesting cases of this theorem occur when $q=2$, as the results for other values of $q$ can be recovered by interpolation and H\"older's inequality.  In \cite[Theorem 6.2]{bct}, a version of this result was established in which the $A^{O(1)} (1 - \frac{1}{\min(q,2)} - \frac{1}{2(d-1)p})^{-O(1)}$ factor was replaced by $O_{A,\eps}(\lambda^\eps)$ for any $\eps>0$ (and with the regularity $M$ required also dependent on $\eps$); as with previous results, we can now sharpen this to $A^{O(1)} \log^{O(1)} \langle \lambda \rangle$, even at the endpoint case $(p,q) = (\frac{1}{d-1},2)$.  The exponent $\lambda^{-d/2p}$ is optimal, as can be seen by setting $f_j \coloneqq \overline{\psi_j}(0,\xi)$ and evaluating $\prod_{j \in [d]} S_\lambda^{(j)} f_j(x)$ for $|x| \leq C A^{-C}/\lambda$ for a sufficiently large $C$.
\end{remark}

Following the wave packet decomposition arguments in \cite[Proposition 6.9]{bct}, one would expect Theorem \ref{lmrc-osc} to follow quickly from Theorem \ref{curv-kak} and Theorem \ref{lmrc}.  Unfortunately, the arguments in \cite[Proposition 6.9]{bct} contain an additional epsilon loss in exponents, and so we need to obtain a more efficient version of this argument that avoids epsilon losses.  This is done in Section \ref{osc-sec}.

The author was supported by a Simons Investigator grant, the James and Carol Collins Chair, the Mathematical Analysis \&
Application Research Fund Endowment, and by NSF grant DMS-1764034. We thank Laura Cladek for helpful discussions and encouragement, and Jon Bennett and Tony Carbery for useful references and discussion.  The author also thanks the anonymous referee for many useful comments and corrections.

\section{Notation}\label{notation-sec}

If $p \in \R$ is a real number and $k \in \N$ is a natural number, we define the binomial coefficient
$$\binom{p}{k} \coloneqq \frac{p(p-1) \dots (p-k+1)}{k!}.$$
If $(\Omega,\mu)$ is a measure space\footnote{All measure spaces will be assumed to be $\sigma$-finite, in order to define product measures without difficulty.},  we let $L^p(\Omega \to \R, \mu)$ be the usual Lebesgue spaces of real-valued functions with respect to such measure spaces; in this paper we will not need to identify functions that agree almost everywhere.  We also will work with complex functions $f \in L^p(\Omega \to \C, \mu)$, vector valued functions $u \in L^p(\Omega \to \R^d, \mu)$ or matrix-valued functions $M \in L^p(\Omega \to \R^{d_1 \times d_2}, d\mu)$.  If $0 < \mu(\Omega) < \infty$ and $f \in L^\infty(\Omega \to \R)$, we can define the expectation
$$ \E_{\mu} f \coloneqq \frac{1}{\mu(\Omega)} \int_\Omega f\ d\mu$$
and the variance
$$ \Var_{\mu} f \coloneqq \E_{\mu} |f - \E_{\mu} f|^2 = \E_{\mu}(f^2) - (\E_{\mu} f)^2;$$
similarly, for $f,g \in L^\infty(\Omega \to \R)$, we can define the covariance
$$ \Cov_{\mu}(f,g) \coloneqq \E_{\mu} (f - \E_{\mu} f)(g - \E_{\mu} g) = \E_{\mu}(fg) - (\E_{\mu} f)(\E_{\mu} g).$$
If $w \in L^1(\Omega \to \R)$ is strictly positive, then $\mu$ and $\mu_w$ are mutually absolutely continuous, and the spaces $L^\infty(\Omega \to \R, \mu)$ and $L^\infty(\Omega \to \R, \mu_w)$ are identical.  We shall often abbreviate such spaces as $L^\infty(\Omega \to \R)$; similarly for vector-valued or matrix-valued analogues of these spaces.

We record the standard quasi-triangle inequality
\begin{equation}\label{quasi}
 \| \sum_k f_k \|_{L^p} \leq \left(\sum_k \|f_k \|_{L^p}^{\min(p,1)}\right)^{1/\min(p,1)},
\end{equation}
valid whenever $0 < p \leq \infty$ and the right-hand side is finite.

Given two finite measures $\mu, \nu$ on a common space $\Omega$, define the \emph{total variation distance} $\| \mu - \nu \|_{\TV}$ to be the quantity
$$ \| \mu - \nu \|_{\TV} \coloneqq \sup_{F\colon \Omega \to [-1,1]} \left|\int_\Omega F\ d\mu - \int_\Omega F\ d\nu\right|$$
where $F$ ranges over measurable functions taking values in the interval $[-1,1]$. Thus for instance
$$ \| \mu \|_{\TV} = \mu(\Omega)$$
and we of course have the triangle inequality
$$ \left|\| \mu \|_{\TV} - \| \nu \|_{\TV}\right| \leq \| \mu - \nu \|_{\TV}.$$

If $E$ is a set, we use $1_E$ to denote its indicator function, thus $1_E(x)=1$ when $x \in E$ and $1_E(x)=0$ when $x \not \in E$. Similarly, if $S$ is a statement, we define the indicator $1_S$ to equal $1$ when $S$ is true and $0$ when $S$ is false, thus for instance $1_E(x) = 1_{x \in E}$.

We endow finite-dimensional spaces such as $\R^d$ or $\R^{d_1 \times d_2}$ with the Euclidean metric, thus for instance if $M \in \R^{d_1 \times d_2}$ then $|M|$ is its Frobenius norm. In any Euclidean space $\R^D$, we use $B^D(x,r) \coloneqq \{ y \in \R^d: |x-y| \leq r\}$ to denote the ball of radius $r>0$ centred at some point $x \in \R^D$ (it will not be of importance to us whether the balls are open or closed).

In this paper we will utilise the following cutoff functions:
\begin{itemize}
\item The gaussian function $\gamma: \R^{d-1} \to \R$ defined by
$$ \gamma(x) \coloneqq \exp( - x x^T ) = \exp(-|x|^2).$$
\item The one-dimensional gaussian function $\gamma^{(1)}: \R \to \R$ defined by
$$ \gamma^{(1)}(x) \coloneqq \exp( - \pi x^2 ).$$
\item The polynomially decaying weight $\rho: \R^{d-1} \to \R$ defined by
$$ \rho(x) \coloneqq \langle x \rangle^{-10d^2}.$$
\item A non-negative bump function $\eta: \R^{d} \to \R$ supported on $B^{d}(0,2)$ that equals 1 on $B^{d}(0,1)$.
\item A non-negative bump function $\eta': \R^{d-1} \to \R$ supported on $B^{d-1}(0,2)$ that equals 1 on $B^{d-1}(0,1)$.
\item A real even Schwartz strictly positive function $\varphi: \R^{d-1} \to \R$ whose Fourier transform $\hat \varphi$ is supported in $B^{d-1}(0,1)$, and which obeys the normalisation
\begin{equation}\label{norm} 
\int_{\R^{d-1}} |\varphi(x')|^2\ dx' = \int_{\R^{d-1}} |\hat \varphi(\xi)|^2\ d\xi = 1.
\end{equation}
\end{itemize}

The function $\varphi$ can be constructed as follows: first start with a non-trivial real even bump function supported on $B^{d-1}(0,1/2)$, take its inverse Fourier transform, and square it to obtain a real even non-negative function whose Fourier transform is supported on $B^{d-1}(0,1)$ and is real analytic.  Convolving this function with itself will then remove all the zeroes, then dividing by its $L^2$ norm will ensure the normalisation \eqref{norm}.

If $f$ denotes any of the above functions, we define the rescaled versions $f_r(x) \coloneqq f(x/r)$ and $f_{x_0,r}(x) \coloneqq f\left(\frac{x-x_0}{r}\right)$ for $r>0$ and $x_0$ in $\R$, $\R^{d-1}$, or $\R^d$ as appropriate, thus for instance we recover the conventions \eqref{gammat}, \eqref{rhoar}, \eqref{varphis} from the introduction.

\section{The calculus of virtual integration on fractional product spaces}\label{fractional}

If $\vec n = (n_1,\dots,n_d) \in \N^d$ is a $d$-tuple of natural numbers, we define the multi-index discrete interval
\begin{align*}
[\vec n] &\coloneqq \Disjoint([n_1], \dots,[n_d]) \\
&= \{ (j,k): j \in [d], k \in [n_j] \}
\end{align*}
thus for instance
$$ [(2,3)] = \{ (1,1), (1,2), (2,1), (2,2), (2,3) \}.$$

Let $\vec p = (p_1,\dots,p_d) \in \N^d$ be a $d$-tuple of natural numbers, and $(\vec \Omega, \vec \mu)$ a $d$-tuple of measure spaces $(\Omega_j,\mu_j)$.  We then define the product space $(\vec \Omega^{\vec p}, \vec \mu^{\vec p})$ to be the measure space given by
$$ \vec \Omega^{\vec p} \coloneqq \prod_{(j,k) \in [\vec p]} \Omega_j \equiv \prod_{j \in [d]} \Omega_j^{p_j}$$
and
$$ \vec \mu^{\vec p} \coloneqq \prod_{(j,k) \in [\vec p]} \mu_j \equiv \prod_{j \in [d]} \mu_j^{p_j}.$$
Given $f \in L^\infty(\Disjoint \vec \Omega \to \R)$, we can define the summed function $\Sigma_{\vec p}(f) \in L^\infty(\vec \Omega^{\vec p} \to \R)$ by the formula
\begin{equation}\label{Sigma-def}
 \Sigma_{\vec p}(f)(\omega) \coloneqq \sum_{(j,k) \in [\vec p]} f(j, \omega_{j,k}),\
\end{equation}
for any $\omega = (\omega_{j,k})_{j \in [d], k \in [p_j]}$ in $\vec \Omega^{\vec p}$.
For instance, if $\vec p = (2,3)$, then
$$ \Sigma_{(2,3)}(f)(\omega) = f(1,\omega_{1,1}) + f(1,\omega_{1,2}) + f(2,\omega_{2,1}) + f(2,\omega_{2,2}) + f(2,\omega_{2,3})$$
The summation operator $\Sigma_{\vec p}\colon L^\infty(\Disjoint \vec \Omega \to \R) \to L^\infty(\vec \Omega^{\vec p} \to \R)$ is linear, and thus can be easily extended to vector-valued or matrix-valued functions by working component-by-component.  

From Fubini's theorem we can easily compute various integrals involving these expressions:

\begin{lemma}[Basic integration formulae]\label{basic}  Let $(\vec \Omega,\vec \mu)$ be a $d$-tuple of measure spaces obeying the finiteness and nondegeneracy condition
\begin{equation}\label{finite}
0 < \mu_j(\Omega_j) < \infty
\end{equation}
for all $j \in [d]$.  Let $\vec p = (p_1,\dots,p_d)$ be a $d$-tuple of natural numbers. Then we have
\begin{equation}\label{id}
\int_{\vec \Omega^{\vec p}}\ d\vec \mu^{\vec p} = \vec \mu(\vec \Omega)^{\vec p} 
\end{equation}
where we recall from the introduction that
$$ \vec \mu(\vec \Omega)^{\vec p} \coloneqq \prod_{j \in [d]} \mu_j(\Omega_j)^{p_j} = \prod_{j \in [d]} \|\mu_j\|_{\TV}^{p_j} .$$
More generally, for any $f,g \in L^\infty(\Disjoint \vec \Omega \to \R)$, one has
\begin{equation}\label{id-2}
 \int_{\vec \Omega^{\vec p}}\Sigma_{\vec p}(f)\ d\vec \mu^{\vec p} = (\sum_{j \in [d]} p_j \E_{\mu_j} f_j) \vec \mu(\vec \Omega)^{\vec p} 
\end{equation}
and
\begin{equation}\label{paj}
\begin{split}
 \int_{\vec \Omega^{\vec p}} \Sigma_{\vec p}(f) \Sigma_{\vec p}(g)\ d\vec \mu^{\vec p} &= 
\left[ (\sum_{j \in [d]} p_j \E_{\mu_j} f_j) (\sum_{j \in [d]} p_j \E_{\mu_j} g_j) + \sum_{j \in [d]} p_j \Cov_{\mu_j}(f_j,g_j) \right] \\
&\quad \times \vec \mu(\vec \Omega)^{\vec p} .
\end{split}
\end{equation}
where $f_j\colon \omega_j \mapsto f(j,\omega_j)$ is the $j^{\operatorname{th}}$ component of $f$, and similarly for $g_j$.
\end{lemma}

\begin{proof} The identity \eqref{id} is immediate from the Fubini-Tonelli theorem.  By \eqref{Sigma-def}, the left-hand side of \eqref{id-2} is equal to
$$\sum_{(j,k) \in [\vec p]} \int_{\vec \Omega^{\vec p}} f_j(\omega_{j,k})\ d\vec \mu^{\vec p},$$
and the claim again follows from the Fubini-Tonelli theorem.

From \eqref{id-2} we already obtain \eqref{paj} in the case that all of the $g_j$ are constant, so by linearity we may assume without loss of generality that the $g_j$ have mean zero: $\E_{\mu_j} g_j = 0$.  Similarly we may assume that $\E_{\mu_j} f_j = 0$.  By \eqref{Sigma-def}, the left-hand side of \eqref{paj} is then equal to
$$\sum_{(j,k), (j',k') \in [\vec p]} \int_{\vec \Omega^{\vec p}} f_j(\omega_{j,k}) g_{j'}(\omega_{j',k'})\ d\vec \mu^{\vec p}.$$
The mean zero hypotheses ensure that the integrals here vanish unless $(j,k) = (j',k')$, and the claim again follows from the Fubini-Tonelli theorem.
\end{proof}

In fact we have a more general formula. If $\vec p = (p_1,\dots,p_d) \in \R^d$ is a $d$-tuple of real numbers and $\vec k = (k_1,\dots,k_d) \in \N^d$ is a $d$-tuple of natural numbers, we define the multi-index binomial coefficient
$$ \binom{\vec p}{\vec k} \coloneqq \prod_{j \in [d]} \binom{p_j}{k_j}.$$

\begin{lemma}[Product integration formula]  Let $(\vec \Omega,\vec \mu)$ be a $d$-tuple of measure spaces obeying \eqref{finite}, and let $\vec p \in \N^d$ be a $d$-tuple of natural numbers.  Then for any $f_1,\dots,f_n \in L^\infty(\Disjoint \vec \Omega \to \R)$, the integral
\begin{equation}\label{into}
\int_{\vec \Omega^{\vec p}} \prod_{i \in [n]} \Sigma_{\vec p}(f_i)\ d\vec \mu^{\vec p}
\end{equation}
can be written as
\begin{equation}\label{formula}
\begin{split}
& \sum_{\vec a \in \N^d} \binom{\vec p}{\vec a} \sum_{(j,k)\colon [n] \twoheadrightarrow [\vec a]} \int_{\vec \Omega^{\vec a}} \left(\prod_{i \in [n]} f_i(j_i, \omega_{j_i,k_i})\right) \ d\vec \mu^{\vec a} \vec \mu(\vec \Omega)^{\vec p - \vec a} 
\end{split}
\end{equation}
where, for each $\vec a \in \N^d$, $(j,k)$ ranges over surjective maps $i \mapsto (j_i,k_i)$ from $[n]$ to $[\vec a]$.  (In particular, there are only finitely many values of $\vec a, (j,k)$ that give a non-zero contribution to this formula.)
\end{lemma}

\begin{proof}   By \eqref{Sigma-def}, one can expand \eqref{into} as
$$ \sum_{(j,k)\colon [n] \to [\vec p]} 
\int_{\vec \Omega^{\vec p}} \prod_{i \in [n]} f_i(j_i, \omega_{j_i,k_i})\ d\vec \mu^{\vec p}$$
where $(j,k)$ ranges over all maps $i \mapsto (j_i,k_i)$ from $[n]$ to $[\vec p]$.  For each such map, there exists a unique tuple $\vec a \in \N^d$, a surjective map $i \mapsto (\tilde j_i, \tilde k_i)$ from $[n]$ to $[\vec a]$, and monotone increasing maps $\iota_j\colon [a_j] \to [p_j]$ for $j \in [n]$, such that
$$ (j_i,k_i) = (\tilde j_i, \iota_{j_i}( \tilde k_i) )$$
for all $i \in [n]$.  Conversely, each choice of $\vec a$, $i \mapsto (\tilde j_i, \tilde k_i)$, $\iota_j\colon [a_j] \to [p_j]$  generates a map $i \mapsto (j_i,k_i)$, whose contribution to \eqref{into} is equal to
$$ \sum_{(j,k)\colon [n] \twoheadrightarrow [\vec a]} \int_{\vec \Omega^{\vec a}} \left(\prod_{i \in [n]} f_i(j_i, \omega_{j_i,k_i})\right) \ d\vec \mu^{\vec a} \vec \mu(\vec \Omega)^{\vec p - \vec a} $$
thanks to the Fubini-Tonelli theorem.  For each choice of $\vec i$ and $i \mapsto (\tilde j_i, \tilde k_i)$, there are $\binom{\vec p}{\vec a}$ possible choices for the maps $\iota_j$, and the claim follows.
\end{proof}

For instance, applying this lemma with $n=2$ and $f,g \in L^\infty(\Disjoint \vec \Omega) \to \R)$, we see that
\begin{equation}\label{fgs}
 \int_{\vec \Omega^{\vec p}} \Sigma_{\vec p}(f) \Sigma_{\vec p}(g) \ d\vec \mu^{\vec p}
\end{equation}
is equal to the sum of 
\begin{align*}
& \sum_{1 \leq j_1 < j_2 \leq d} p_{j_1} p_{j_2} \int_{\Omega_{j_1} \times \Omega_{j_2}} (f(j_1,\omega_{j_1}) g(j_2,\omega_{j_2}) + f(j_2,\omega_{j_2}) g(j_1,\omega_{j_1}))\\
&\quad\quad\quad d\mu_{j_1}(\omega_{j_1}) d\mu_{j_2}(\omega_{j_2}) \\
&\quad \times \prod_{j \in [d]} \mu_j(\Omega_j)^{p_j - 1_{j=j_1} - 1_{j=j_2}}
\end{align*}
and
\begin{align*}
& \sum_{j_1 \in [d]} \binom{p_{j_1}}{2} \int_{\Omega_{j_1}^2} (f(j_1,\omega_{j_1}) g(j_1,\omega'_{j_1}) + f(j_1,\omega'_{j_1}) g(j_1,\omega_{j_1}))\ d\mu_{j_1}(\omega_{j_1}) d\mu_{j_1}(\omega'_{j_1}) \\
&\quad \times \prod_{j \in [d]} \mu_j(\Omega_j)^{p_j - 2 \times 1_{j=j_1}}
\end{align*}
and
\begin{align*}
\sum_{j_1 \in [d]} p_{j_1} \int_{\Omega_{j_1}} f(j_1,\omega_{j_1}) g(j_1,\omega_{j_1}) \ d\mu_{j_1}(\omega_{j_1}) \prod_{j \in [d]} \mu_j(\Omega_j)^{p_j - 1_{j=j_1}}
\end{align*}
and after some calculation one can check that this agrees with \eqref{paj}.

We now make the key observation that the expression \eqref{formula} continues to be well defined when $\vec p \in \R^d$ is a $d$-tuple of real numbers, rather than natural numbers.  (Here the lower bound in \eqref{finite} becomes important, because the exponents $p_j-a_j$ can be negative without the binomial coefficient $\binom{p_j}{a_j}$ vanishing.)  As such, this allows us to \emph{formally} make sense of integrals such as \eqref{into} when the exponents $p_j$ are no longer assumed to be natural numbers, even if one can no longer interpret $(\vec \Omega^{\vec p}, \vec\mu^{\vec p})$ as a classical measure space in this case.

More precisely, if $\vec p \in \R^d$ is a real $d$-tuple, we define ${\mathcal L}^\infty( \vec \Omega^{\vec p} \to \R )$ to be the real commutative unital algebra generated by the formal symbols $\Sigma_{\vec p}( f )$, $f \in L^\infty(\Disjoint \vec \Omega \to \R)$, thus elements of ${\mathcal L}^\infty( \vec \Omega^{\vec p} \to \R )$ (which we shall call \emph{virtual functions}) consists of formal real linear combinations of formal commutative products
$$ \Sigma_{\vec p}( f_1 ) \dots \Sigma_{\vec p}(f_n)$$
with $n \in \N$ and $f_1,\dots,f_n \in L^\infty(\Disjoint \vec \Omega \to \R)$.  We interpret $\Sigma_{\vec p}$ as linear operators, for instance we identify $\Sigma_{\vec p}(f+g)$ with $\Sigma_{\vec p}(f) + \Sigma_{\vec p}(g)$.  With these identifications one can view ${\mathcal L}^\infty( \vec \Omega^{\vec p} \to \R )$ as a Fock-type space
$$ {\mathcal L}^\infty( \vec \Omega^{\vec p} \to \R ) \equiv \bigoplus_{n=0}^\infty \mathrm{Sym}^n\left( L^\infty\left(\Disjoint \vec \Omega \to \R\right) \right)$$
where $\mathrm{Sym}^n$ denotes the $n^{\operatorname{th}}$ symmetric tensor power.

In this setting, the space $\vec \Omega^{\vec p}$ no longer exists as a classical set, and $\Sigma_{\vec p}( f )$ can no longer be interpreted as a classical function.  Nevertheless, given an element $F$ of this algebra ${\mathcal L}^\infty( \vec \Omega^{\vec p} \to \R )$, one can define the ``virtual integral''
$$ \int_{\vec \Omega^{\vec p}} F\ d\vec \mu^{\vec p}$$
by decomposing this expression as a finite linear combination of virtual integrals
$$ \int_{\vec \Omega^{\vec p}} \prod_{i \in [n]} \Sigma_{\vec p}( f_i )\ d\vec \mu^{\vec p}$$
and then evaluating each such virtual integral using the formula \eqref{formula} (where we extend \eqref{multi} to the case of non-integer $p_j$ in the obvious fashion).  This is well-defined because the expression \eqref{formula} is easily seen to be invariant with respect to permutations of the $f_1,\dots,f_n$, and is also multilinear in these inputs.  One can also extend this virtual integration concept to vector-valued virtual functions
$$ \phi \in {\mathcal L}^\infty( \vec \Omega^{\vec p} \to \R^D ) \equiv {\mathcal L}^\infty( \vec \Omega^{\vec p} \to \R ) \otimes \R^D$$
or matrix-valued virtual functions
$$ M \in {\mathcal L}^\infty( \vec \Omega^{\vec p} \to \R^{D_1 \times D_2} ) \equiv {\mathcal L}^\infty( \vec \Omega^{\vec p} \to \R ) \otimes \R^{D_1 \times D_2}$$
in the usual fashion (integrating each component separately).

With these conventions, we see that Lemma \ref{basic} now extends to the case where the $d$-tuple $\vec p$ consists of real numbers rather than natural numbers.  

From the H\"older and triangle inequalities one easily obtains the following upper bounds:

\begin{lemma}[H\"older bound]\label{holder} Let $(\vec \Omega,\vec \mu)$ be a $d$-tuple of measure spaces obeying \eqref{finite}, and let $\vec p \in \R^d$ be a $d$-tuple of real numbers bounded in magnitude by $C$.  Then for any $f_1,\dots,f_n \in L^\infty(\Disjoint \vec \Omega \to \R)$ and any exponents $1 \leq q_1,\dots,q_n \leq \infty$ with
$$ \frac{1}{q_1} + \dots + \frac{1}{q_n} = 1,$$
one has
$$
\int_{\vec \Omega^{\vec p}} \left(\prod_{i \in [n]} \Sigma_{\vec p}(f_i)\right) \ d\vec \mu^{\vec p} \lesssim_{C,n}
\left(\prod_{i \in [n]} \sum_{j \in [d]} (\E_{\mu_j} |(f_i)_j|^{q_i})^{1/q_i}\right) \vec \mu(\vec \Omega)^{\vec p} 
$$
where we adopt the usual convention that $\E_{\mu_j} |(f_i)_j|^{q_i})^{1/q_i}$ is the essential supremum of $|(f_i)_j|$ if $q_i=\infty$.
\end{lemma}

When $\vec p$ consists of natural numbers, we clearly have the non-negativity relation
$$ \int_{\vec \Omega^{\vec p}} F^2 \ d\vec \mu^{\vec p} \geq 0$$
whenever $F \in L^\infty(\vec \Omega^{\vec p} \to \R)$.  However, we caution that such non-negativity relations can break down for virtual integrals in the fractional exponent setting, even if we continue to insist that all the exponents are positive.  For instance, if $d=1$, $\mu(\Omega)=1$, and $f \in L^\infty(\Omega)$ has expectation zero, $\E_{\mu} f = 0$, one can check that
$$ \int_{\Omega^p} \Sigma_p(f)^4\ d\mu^p = p \Var_{\mu}(f^2) + p (3p-2) (\E_{\mu}(f^2))^2 $$
and the right-hand side can be negative when $p < 2/3$.

On the other hand, we do have the following non-negativity property, implicit in \cite{bct}, which will power our heat flow monotonicity results:

\begin{theorem}[Non-negativity]\label{non-neg}  Let $A \geq 2$, and suppose that $0 < \eps \leq A^{-1}$.  Let $(\vec \Omega,\vec \mu)$ be a $d$-tuple of measure spaces obeying \eqref{finite}, and let $\vec p \in \R^d$ be a $d$-tuple of real numbers $p_j$ with $A^{-1} \leq p_j \leq A$.  For each $j \in [d]$, let $B_j^0 \in \R^{d \times d-1}$ be a matrix of norm at most $A$, and let $B \in L^\infty(\Disjoint \vec \Omega \to \R^{d \times d-1})$ be a matrix-valued function such that
\begin{equation}\label{beo}
 B(j,\omega_j) = B_j^0 + O(\eps)
\end{equation}
for all $(j,\omega_j) \in \Disjoint \vec \Omega$.  Writing
\begin{equation}\label{mo}
 M^0 = \sum_{j \in [d]} p_j B_j^0 (B_j^0)^T \in \R^{d \times d}
\end{equation}
and let $M \in {\mathcal L}^\infty(\vec \Omega^{\vec p}) \otimes \R^{d \times d}$ be the virtual matrix-valued function
$$ M \coloneqq \Sigma_{\vec p}( B B^T ).$$
\begin{itemize}
\item[(i)]  One has
$$ \int_{\vec \Omega^{\vec p}} \det(M)\ d\vec \mu^{\vec p} = (\det(M_0) + O( A^{O(1)} \eps )) \vec \mu(\vec \Omega)^{\vec p}.$$
\item[(ii)]  Suppose that $M_0$ is positive definite with all singular values between $A^{-1}$ and $A$, and suppose that 
$$ (1 - C A^C \eps) M_0 - B_j^0 $$
is positive semi-definite for each $j \in [d]$, where $C$ is sufficiently large depending on $d$.  Then one has
$$
 \int_{\vec \Omega^{\vec p}} \{ \phi, \phi \} \ d\vec \mu^{\vec p} \geq 0
$$
for all virtual vector-valued functions $\phi \in {\mathcal L}^\infty(\vec \Omega^{\vec p} \to \R^{d \times 1})$, where $\{,\}$ is the bilinear form
$$ \{ u, v \} \coloneqq \det(M) \Sigma_{\vec p}(uv^T) - \Sigma_{\vec p}(u B^T) \adj(M) \Sigma_{\vec p}(B v^T).$$
\end{itemize}
\end{theorem}


\begin{proof}  We begin with (i).  From \eqref{beo} we can write
$$ B B^T = B^0 (B^0)^T + O(A\eps) $$
on $\Disjoint \vec \Omega$, where $B^0 \coloneqq \Disjoint \vec \Omega \to \R^{d \times d-1}$ is the function that maps $(j,\omega_j)$ to $B_j^0$ for all $(j,\omega_j) \in \Disjoint \vec \Omega$.  Thus
$$ M = M_0 + \eps \Sigma_{\vec p}( N )$$
for some matrix function $N \in L^\infty(\Disjoint \vec \Omega \to \R^{d \times d})$ of norm $O(A)$.  Taking determinants and adjugates (noting that $A\eps \leq 1$ and $M_0 = O(A^2)$) we conclude that
\begin{equation}\label{dete}
 \det(M) = \det(M_0) + \eps  E 
\end{equation}
and also
\begin{equation}\label{adje}
 \adj(M) = \adj(M_0) + \eps  E' 
\end{equation}
where the virtual function $E$ and all components of the virtual matrix function $E'$ are the sum of $O(1)$ terms of the form $\prod_{i \in [n]} \Sigma_{\vec p}(f_i)$ with $n=O(1)$ and each $f_i$ of norm $O(A^{O(1)})$ in $L^\infty(\Disjoint \vec \Omega \to \R)$.  In particular, from Lemma \ref{holder} we have
$$ \int_{\vec \Omega^{\vec p}} \det(M)\ d\vec \mu^{\vec p} = \det(M_0) \vec \mu(\vec \Omega)^{\vec p} + O( A^{O(1)} \eps \vec \mu(\vec \Omega)^{\vec p}  )$$
which gives (i).

Now we turn to (ii).  Observe that the bilinear form $\{,\}$ is symmetric, and that $h B$ is in the null space for this form for any row vector $h \in \R^d$, since by \eqref{Adj} one has
\begin{align*}
 \{ h B, v \}  &= \det(M) \Sigma_{\vec p}(h B v^T) - \Sigma_{\vec p}(h B B^T) \adj(M) \Sigma_{\vec p}(B v^T) \\
&= \det(M) h \Sigma_{\vec p}(B v^T) - h M  \adj(M) \Sigma_{\vec p}(B v^T) \\
&= 0.
\end{align*}
Hence we have
\begin{equation}\label{phib}
 \{ \phi, \phi \} = \{ \phi - h B, \phi - hB \}
\end{equation}
for any $h \in \R^d$.

Next, we claim that we can find $h \in \R^d$ for which we have 
\begin{equation}\label{som}
 \Sigma_{\vec p}\left(\E_{\vec \mu} (\phi - hB) (B^0)^T\right) = 0.
\end{equation}
Indeed, the left-hand side expands as the constant function
$$ \sum_{j \in [d]} p_j \E_{\mu_j} \phi_j (B^0_j)^T - h \sum_{j \in [d]} p_j \E_{\mu_j} B_j (B^0_j)^T$$
where $B_j \colon \omega_j \mapsto B(j,\omega_j)$ are the components of $B$.  From \eqref{beo}, \eqref{mo} we have
$$ \sum_{j \in [d]} p_j \E_{\mu_j} B_j (B^0_j)^T = M_0 + O(A^{O(1)} \eps ) $$
and in particular this matrix is invertible since the singular values of $M_0$ are at least $C A^C \eps$.  Thus we can find $h \in \R^d$ obeying \eqref{som}; by \eqref{phib} we may thus assume without loss of generality that we have the normalisation
\begin{equation}\label{lip}
  \Sigma_{\vec p}(\E_{\vec \mu} \phi (B^0)^T)) = 0.
\end{equation}

From \eqref{beo}, \eqref{dete}, \eqref{adje} one can write
$$ \{ \phi, \phi \} = \det(M^0) \Sigma_{\vec p}(\phi \phi^T) - \Sigma_{\vec p}(\phi (B^0)^T) \adj(M^0) \Sigma_{\vec p}(B^0 \phi^T) + \eps E'' $$
where the virtual function $E'' \in {\mathcal L}^\infty(\Disjoint \vec \Omega \to \R)$ is the sum of $O(1)$ terms of the form $\left(\prod_{i \in [n]} \Sigma_{\vec p}(f_i)\right) \Sigma_{\vec p}(\phi^i) \Sigma_{\vec p}(\phi^{i'})$ with $n=O(1)$ and $f_1,\dots,f_n$ of norm $O(A^{O(1)})$ in $L^\infty(\Disjoint \vec \Omega)$, and $\phi^i, \phi^{i'}$ two components of $\phi$.  From Lemma \ref{holder} we thus have
\begin{align*}
 \int_{\vec \Omega^{\vec p}} \{ \phi, \phi \}\ d\vec \mu^{\vec p} &= 
\int_{\vec \Omega^{\vec p}} \det(M^0) \Sigma_{\vec p}(\phi \phi^T) - \Sigma_{\vec p}(\phi (B^0)^T) \adj(M^0) \Sigma_{\vec p}(B^0 \phi^T) \ d\vec \mu^{\vec p} \\
&\quad + O\left( A^{O(1)} \eps \sum_{j \in [d]} (\E_{\mu_j} \phi_j \phi_j^T) \vec \mu(\vec \Omega)^{\vec p} \right).
\end{align*}
From Lemma \ref{basic} (extended to fractional exponents) and the fact that all singular values of $M^0$ are $A^{O(1)}$, one has
$$ (\E_{\mu_j} \phi_j \phi_j^T) \vec \mu(\vec \Omega)^{\vec p} \lesssim A^{O(1)}
\int_{\vec \Omega^{\vec p}} \Sigma_{\vec p}(\phi \phi^T)\ d\vec \mu^{\vec p},$$
and hence  we have (using \eqref{Adj})
\begin{align*}
& \int_{\vec \Omega^{\vec p}} \{ \phi, \phi \}\ d\vec \mu^{\vec p} \geq 
\det(M^0) \times \\
&\quad \int_{\vec \Omega^{\vec p}} (1-CA^C \eps) \Sigma_{\vec p}(\phi \phi^T) - \Sigma_{\vec p}(\phi (B^0)^T) (M^0)^{-1} \Sigma_{\vec p}(B^0 \phi^T) \ d\vec \mu^{\vec p}.
\end{align*}
We split $\phi = \E_{\vec \mu} \phi + \psi$, where
$$ \psi \coloneqq \phi - \E_{\vec \mu} \phi$$
is the mean zero component of $\phi$.  From \eqref{lip} we have
$$ \Sigma_{\vec p}(\phi (B^0)^T) = \Sigma_{\vec p}(\psi (B^0)^T) $$
and from Lemma \ref{basic} (extended to fractional $p$) we have
$$ \int_{\vec \Omega^{\vec p}} \Sigma_{\vec p}(\phi\phi^T) \ d{\vec\mu}^{\vec p} \geq \vec \mu(\vec \Omega)^{\vec p} \sum_{j \in [d]}
p_j \E_{\mu_j} \psi_j \psi_j^T $$
and
$$ \int_{\vec \Omega^{\vec p}} \Sigma_{\vec p}(\psi' (B^0)^T) (M^0)^{-1} \Sigma_{\vec p}(B^0 \psi^T))\ d\vec\mu^{\vec p} = \vec \mu(\vec \Omega)^{\vec p}  \sum_{j \in [d]} p_j \E_{\mu_j} \psi_j (B^0_j)^T (M^0)^{-1} B^0_j \psi_j^T.$$
From hypothesis, the matrix
$$ (1 - CA^C \eps) - (M^0)^{-1/2} (B^0_j)^T B^0_j (M^0)^{-1/2}$$
is positive semi-definite, hence the singular values of $B^0_j (M^0)^{-1/2}$ do not exceed $(1 - CA^C \eps)^{1/2}$.  This implies the pointwise bound
$$ (1 - CA^C \eps) \psi_j \psi_j^T \geq \psi_j (B^0_j)^T (M^0)^{-1} B^0_j \psi_j^T$$
and the claim follows.
\end{proof}

If $\mu,\nu$ are two finite measures on $\Omega$ and $p$ is a natural number, then $\mu^p, \nu^p$ are finite measures on $\Omega^p$, and we have the telescoping bound
\begin{equation}\label{tv-power}
\begin{split}
 \| \mu^p - \nu^p\|_{\TV} &\leq \sum_{j \in [p]} \| \mu^{j-1} \times (\mu-\nu) \times \nu^{p-j}  \|_{\TV} \\
&= \sum_{j \in [p]} \|\mu\|_{\TV}^{j-1} \| \nu \|_{\TV}^{p-j} \| \mu - \nu \|_{\TV} \\
&\leq p (\| \mu\|_{\TV}^{p-1} + \| \nu \|_{\TV}^{p-1} ) \| \mu - \nu \|_{\TV} \\
&\lesssim p (\| \mu\|_{\TV}^{p} + \| \nu \|_{\TV}^{p} ) \frac{\| \mu - \nu \|_{\TV}}{\| \mu\|_{\TV} + \| \nu \|_{\TV}}.
\end{split}
\end{equation} 

We have the following variant of the above bound for fractional products of measures:

\begin{lemma}[Total variation bound]\label{hold}  Let $(\vec \Omega, \vec \mu)$ and $(\vec \Omega, \vec \nu)$ be $d$-tuples of measure spaces on a common domain $\vec \Omega$ obeying \eqref{finite}, and let $\vec p$ be a $d$-tuple of positive reals bounded by $C$.  Let $F_1,\dots,F_n \in L^\infty(\Disjoint \vec \Omega \to \R)$ be bounded in norm by $C$, and let $G$ be the virtual function $G \coloneqq \prod_{i \in [n]} \Sigma_{\vec p}(F_i)$.  Then
$$ \left|\int_{\vec \Omega^p} G\ d\vec \mu^{\vec p} - \int_{\vec \Omega^p} G\ d\vec \nu^{\vec p}\right| \lesssim_{C,n} (\vec \mu(\vec \Omega)^{\vec p} + 
\vec \nu(\vec \Omega)^{\vec p}) \sum_{j \in [d]} \frac{\|\mu_j-\nu_j\|_{\TV}}{\|\mu_j\|_{\TV} + \| \nu_j \|_{\TV}}.$$
\end{lemma}

\begin{proof}  By rescaling $\mu_j,\nu_j$ by $1/\|\mu_j \|_{\TV}$, we may normalise $\|\mu_j \|_{\TV}=1$ for all $j$.  Let $\eps>0$ be a sufficiently small quantity depending on $d,C,n$.  If we have
$$ \|\mu_j-\nu_j\|_{\TV} \geq \eps ( \|\mu_j \|_{\TV} + \|\nu_j \|_{\TV} )$$
for some $j$, then the claim follows from the triangle inequality and Lemma \ref{holder}, so we may assume that
$$ \|\mu_j-\nu_j\|_{\TV} < \eps ( \|\mu_j \|_{\TV} + \|\nu_j \|_{\TV} ).$$
By the normalisation $\|\mu_j \|_{\TV}=1$ and the triangle inequality, this implies (for $\eps$ small enough) that
$$ \|\mu_j-\nu_j\|_{\TV} \lesssim \eps.$$
This and \eqref{tv-power} (together with the normalisation $\|\mu_j \|_{\TV}=1$) implies that
$$ \|\vec \mu^{\vec a}-\vec \nu^{\vec a}\|_{\TV} \lesssim_{\vec a} \eps$$
for any $d$-tuple $\vec a$ of natural numbers.  Applying \eqref{formula} we obtain
$$ \int_{\vec \Omega^p} G\ d\vec \mu^{\vec p} - \int_{\vec \Omega^p} G\ d\vec \nu^{\vec p} \lesssim_{C,n} \eps$$
giving the claim.
\end{proof}

From Lemma \ref{holder} and Lemma \ref{hold} we see that the virtual integration operation is continuous with respect to the sup norm topology on the integrands, and the total variation topology on the measure:

\begin{corollary}[Continuity of virtual integration]\label{continuity} Let $(\vec \Omega, \vec \mu[x])$ be a $d$-tuple of measure spaces $(\Omega_j,\mu_j[x])$ on a common domain $\vec \Omega$ obeying \eqref{finite} parameterised by elements $x$ of a topological space $X$, and let $\vec p$ be a $d$-tuple of positive reals.  Let $n$ be a natural number, and for each $i \in [n]$ and $x \in X$, let $F_i[x]$ be an element of $L^\infty(\Disjoint \vec \Omega \to \R)$.  Suppose that the map $x \mapsto F_i[x]$ is a continuous map (using the $L^\infty$ metric on $F_i[x]$, and suppose also that the maps $x \mapsto \mu_j[x]$ are continuous (using the total variation metric on $\mu_j[x]$).  Then the map
$$ x \mapsto \int_{\vec \Omega^p} \prod_{i \in [n]} \Sigma_{\vec p}(F_i[x])\ d\vec \mu[x]^{\vec p}$$
is continuous from $X$ to $\R$.
\end{corollary}

\begin{proof}  Let $x_\alpha$ be a net in $X$ converging to a limit $x$.  By hypothesis, the total variation norms of $\mu_j[x_\alpha]$ as well as the $L^\infty$ norms of $F_i[x_\alpha]$ are bounded for sufficiently large $\alpha$, with the former also being bounded away from zero.  By continuity and Lemma \ref{holder}, the difference
$$\int_{\vec \Omega^p} \prod_{i \in [n]} \Sigma_{\vec p}(F_i[x_\alpha]) \ d\vec \mu[x_\alpha]^{\vec p} - \int_{\vec \Omega^p} \prod_{i \in [n]} \Sigma_{\vec p}(F_i[x])\ d\vec \mu[x_\alpha]^{\vec p}$$
then goes to zero as $\alpha \to \infty$, and similarly from Lemma \ref{hold} the difference
$$\int_{\vec \Omega^p} \prod_{i \in [n]} \Sigma_{\vec p}(F_i[x])\ d\vec \mu[x_\alpha]^{\vec p} - \int_{\vec \Omega^p} \prod_{i \in [n]} \Sigma_{\vec p}(F_i[x]) \ d\vec \mu[x]^{\vec p}$$
also goes to zero as $\alpha \to \infty$.  The claim now follows from the triangle inequality.
\end{proof}

We will need a change of variables formula for virtual integrals.  Given a measure space $(\Omega,\mu)$ and a measurable map $\pi\colon \Omega \to \Omega'$ into another measurable space $\Omega'$, we can define the pushforward measure $\pi_* \mu$ on $\Omega'$ by the formula
$$ \pi_* \mu(E') \coloneqq \mu( \pi^{-1}(E') ) $$
for any measurable $E' \subset \Omega'$; one then has the change of variables formula
\begin{equation}\label{change}
 \int_\Omega (\pi^* f')\ d\mu = \int_{\Omega'} f'\ d\pi_* \mu
\end{equation}
for any $f' \in L^\infty(\Omega' \to \R)$, where the pullback $\pi^* f' \in L^\infty(\Omega \to \R)$ is defined by $\pi^* f'(\omega) \coloneqq f'(\pi(\omega))$.  More generally, given a $d$-tuple $(\vec \Omega,\vec \mu)$ of measure spaces $(\Omega_j,\mu_j)$ and a $d$-tuple $\vec \pi = (\pi_1,\dots,\pi_n)$ of measurable maps $\pi_j\colon \Omega_j \to \Omega'_j$, we obtain a $d$-tuple $(\vec \Omega', \vec \pi_* \vec \mu)$ of pushforward measure spaces $(\Omega_j, (\pi_j)_* \mu_j)$.   We can interpret $\vec \pi_* \vec \mu$ as the measure on $\Disjoint \vec \Omega'$ formed by pushing forward the measure $\vec \mu$ on $\Disjoint \vec \Omega$ by the map $\vec \pi\colon \Disjoint \vec \Omega \to \Disjoint \vec \Omega'$ defined (abusing notation slightly) as
$$ \vec \pi( j, \omega_j ) \coloneqq (j, \pi_j(\omega_j) ).$$
For any $d$-tuple $\vec p$ of natural numbers $p_j$, it is then not difficult (by checking first the case when $F'$ is a tensor product of one-variable functions) to verify the change of variables formula
\begin{equation}\label{change-multi}
 \int_{\vec \Omega^{\vec p}} (\pi^* F')\ d\vec \mu^{\vec p} = \int_{(\vec \Omega')^{\vec p}} F'\ d(\pi_* \vec \mu)^{\vec p}
\end{equation}
for any $F' \in L^\infty((\vec \Omega')^{\vec p} \to \R)$, where the pullback $\pi^* F' \in L^\infty(\vec \Omega^{\vec p} \to \R)$ is now defined as
$$ \pi^* F'( (\omega_{j,k})_{(j,k) \in [\vec p]} ) \coloneqq F'( (\pi_j(\omega_{j,k}))_{(j,k) \in [\vec p]} ).$$
In particular we see that the pullback homomorphism commutes with the summation operator $\Sigma_{\vec p}$, thus
\begin{equation}\label{sigmapi}
 \Sigma_{\vec p}( \pi^* f' ) = \pi^* \Sigma_{\vec p}(f')
\end{equation}
for any $f' \in L^\infty(\Disjoint \vec \Omega' \to \R)$.  Inspired by this, we now also define a pullback operation
$$ \pi^*\colon {\mathcal L}^\infty((\vec \Omega')^{\vec p} \to \R) \to {\mathcal L}^\infty(\vec \Omega^{\vec p} \to \R) $$
on virtual functions, with $\vec p$ now a $d$-tuple of real numbers, to be the unique algebra homomorphism obeying \eqref{sigmapi}, thus in particular
$$ \pi^*\left(\prod_{i \in [n]} \Sigma_{\vec p}(f'_i)\right) = \prod_{i \in [n]} \Sigma_{\vec p}(\pi^* f_i).$$
From \eqref{formula} and \eqref{change}, we see that the change of variables formula \eqref{change-multi} is now valid
for virtual functions $F' \in L^\infty(\Disjoint \vec \Omega' \to \R)$ and arbitrary real $d$-tuples $\vec p$.

\subsection{Differentiation under the virtual integral sign}\label{diff-sec}

The power of the virtual integration formalism truly emerges when one lets the measure $\vec \mu$ vary with additional parameters such as time $t \in \R$ or space $x \in \R^d$, and differentiates with respect to these parameters.  We just discuss here the case of a time parameter $t$, as the theory of a spatial parameter is completely analogous.

We begin with some formal calculations.  Suppose that $w = w[t] \in L^\infty(\Disjoint \vec \Omega \to \R)$ is a non-negative weight that varies with a time parameter $t$ and obeys the ordinary differential equation
\begin{equation}\label{wat}
 \partial_t w[t] = c[t] w[t]
\end{equation}
pointwise for some $c = c[t] \in L^\infty(\Disjoint \vec \Omega \to \R)$ that also depends on $t$, thus $c$ is the log-derivative of $w$.  Then from the chain rule and product rule, and assuming initially that $\vec p$ is a $d$-tuple of natural numbers, one formally has the identity
$$ \partial_t d\vec \mu_{w[t]}^{\vec p} = \Sigma_{\vec p}(c[t]) d\vec \mu_{w[t]}^{\vec p}$$
and hence if one formally differentiates under the integral sign, one expects to have the derivative formula
$$ \partial_t \int_{\vec \Omega^{\vec p}} F[t] \ d\vec \mu_{w[t]}^{\vec p} =
\int_{\vec \Omega^{\vec p}} (\partial_t F[t] + \Sigma_{\vec p}(c[t]) F[t])\ d\vec \mu_{w[t]}^{\vec p}.$$

We now extend this formula to fractional exponents.  We begin with the simpler case when the integrand $F$ is independent of the parameter.

\begin{proposition}[Differentiation under the integral sign, I]\label{diff-i}  
Let $(\vec \Omega, \vec \mu)$ be a $d$-tuple of measure spaces $(\Omega_j,\mu_j)$, and let $\vec p$ be a $d$-tuple of positive reals.  For all $t$ in an open interval $I$, let $w[t] \in L^1(\Disjoint \vec \Omega \to \R)$ be a non-negative weight that is not identically zero, and let $c[t] \in L^\infty(\Disjoint \vec \Omega \to \R)$ be an additional function obeying the equation \eqref{wat} pointwise.  We also assume that the map $t \mapsto c[t]$ is continuous (using the $L^\infty$ metric on $c[t]$).  Then for any $F \in {\mathcal L}^\infty(\vec \Omega \to \R)$, one has
\begin{equation}\label{cater}
 \partial_t \int_{\vec \Omega^{\vec p}} F\ d\vec \mu_{w[t]}^{\vec p} = \int_{\vec \Omega^{\vec p}} F \Sigma_{\vec p}(c[t])\ d\vec \mu_{w[t]}^{\vec p}
\end{equation}
for all $t \in I$; in particular, the expression $\int_{\vec \Omega^{\vec p}} F\ d\vec \mu_{w[t]}^{\vec p}$ is everywhere differentiable in $t$.
\end{proposition}

In practice, we shall be able to upgrade differentiability to continuous differentiability by invoking Corollary \ref{continuity}.  Similarly for Proposition \ref{diff-ii} below.

\begin{proof}  By linearity, we may assume that $F$ takes the form $F = \prod_{i \in [n]} \Sigma_{\vec p}(f_i)$ for some $f_1,\dots,f_n \in L^\infty(\Disjoint \vec \Omega \to \R)$.  

Fix $t \in I$.  We may normalise $\| (\mu_j)_{w_j[t]} \|_{\TV} = 1$ for all $j \in [d]$.
From the method of integrating factors we see that
$$ w[t+h] = \exp\left( \int_t^{t+h} c[t']\ dt' \right) w[t] $$
for sufficiently small real numbers $h$, so by the continuity of $c$ we have
\begin{equation}\label{tv}
 \| \vec \mu_{w[t+h]} - (1+hc[t]) \vec \mu_{w[t]} \|_{\TV} = o(|h|) 
\end{equation}
where $o(|h|)$ denotes a quantity that goes to zero as $h \to 0$ after dividing by $|h|$.  From Lemma \ref{hold}, we then have
$$ \int_{\vec \Omega^{\vec p}} F\ d\vec \mu_{w[t+h]}^{\vec p} 
= \int_{\vec \Omega^{\vec p}} F\ d((1+hc[t])\vec \mu_{w[t]})^{\vec p} + o(|h|).$$
To abbreviate notation we now write $c[t], w[t]$ as $c,w$ respectively.  From \eqref{formula}, we may write this expression as
\begin{align*}
& \sum_{\vec a \in \N^d} \binom{\vec p}{\vec a} \sum_{(j,k)\colon [n] \twoheadrightarrow [\vec a]} \int_{\vec \Omega^{\vec a}} \left(\prod_{i \in [n]} f_i(j_i, \omega_{j_i,k_i})\right) \ d((1 + hc)\vec \mu)^{\vec a} \\
&\quad \prod_{j \in [d]} (1 + h \E_{(\mu_j)_{w_j}} c)^{p_j-a_j} + o(|h|).
\end{align*}
Note that
$$((1 + hc)\vec \mu)^{\vec a} = (1 + h \Sigma_{\vec a}(c) + o(|h|)) \vec \mu^{\vec a}.$$
Thus we can differentiate at $h=0$ and express the left-hand side of \eqref{cater} as $X+Y$, where
$$ 
X \coloneqq \sum_{\vec a \in \N^d} \binom{\vec p}{\vec a} \sum_{(j,k)\colon [n] \twoheadrightarrow [\vec a]} \int_{\vec \Omega^{\vec a}} \left(\prod_{i \in [n]} f_i(j_i, \omega_{j_i,k_i})\right) \Sigma_{\vec a}(c)\ d\vec \mu_w^{\vec a} $$
and
$$ 
Y \coloneqq \sum_{\vec a \in \N^d} \binom{\vec p}{\vec a} \sum_{(j,k)\colon [n] \twoheadrightarrow [\vec a]} \int_{\vec \Omega^{\vec a}} \left(\prod_{i \in [n]} f_i(j_i, \omega_{j_i,k_i})\right)\ d\vec \mu_w^{\vec a} \sum_{j' \in [d]} (p_{j'}-a_{j'}) \E_{(\mu_j)_{w_j}} c.$$
If we now rename $c$ as $f_{n+1}$, we see that $X$ is equal to
\begin{equation}\label{terms}
\sum_{\vec a \in \N^d} \binom{\vec p}{\vec a} \sum_{(j,k) \in {\mathcal X}_{\vec a}} \int_{\vec \Omega^{\vec a}} \left(\prod_{i \in [n+1]} f_i(j_i, \omega_{j_i,k_i})\right)\ d\vec \mu_w^{\vec a},
\end{equation}
where ${\mathcal X}_{\vec a}$ is the collection of surjections $(j,k) \colon [n+1] \twoheadrightarrow [\vec a]$ for which $(j_{n+1},k_{n+1}) = (j_i,k_i)$ for at least one $i \in [n]$.  As for $Y$, we can rewrite it (after symmetrising) as
$$ \sum_{j' \in [d]} \sum_{\vec a \in \N^d} \binom{\vec p}{\vec a} \frac{p_{j'}-a_{j'}}{a_{j'}+1} \sum_{(j,k) \in {\mathcal Y}_{\vec a, j'}} \int_{\vec \Omega^{\vec a + e_{j'}}} \left(\prod_{i \in [n+1]} f_i(j_i, \omega_{j_i,k_i})\right)\ d\vec \mu_w^{\vec a + e_{j'}}$$
where ${\mathcal X}_{\vec a,j'}$ is the collection of surjections $(j,k) \colon [n+1] \to [\vec a + e_{j'}]$ for which $j_{n+1} = j'$ and $(j_{n+1},k_{n+1}) \neq (j_i,k_i)$ for all $i \in [n]$.  Note that
$$\binom{\vec p}{\vec a} \frac{p_{j'}-a_{j'}}{a_{j'}+1}  = \binom{\vec p}{\vec a + e_{j'}}.$$  
It is then clear that $X+Y$ is equal to
$$
\sum_{\vec a' \in \N^d} \binom{\vec p}{\vec a'} \sum_{(j,k): [n+1] \twoheadrightarrow [\vec a']} \int_{\vec \Omega^{\vec a'}} \left(\prod_{i \in [n+1]} f_i(j_i, \omega_{j_i,k_i})\right)\ d\vec \mu_w^{\vec a'},$$
and the claim \eqref{cater} now follows from \eqref{formula}.
\end{proof}

Now we allow $F$ to depend on $t$.  Here we run into a technical issue in that we have not imposed a topology on the space ${\mathcal L}^\infty(\Disjoint \Omega \to \R)$, so we cannot immediately define standard notions of derivative such as the Fr\'echet derivative.  While one could place a somewhat artificial topology on this space, we shall avoid this problem by working with a quite restricted notion of derivative.  We say that a map $t \mapsto F[t]$ from a time interval $I$ to ${\mathcal L}^\infty(\vec \Omega^{\vec p} \to \R)$ is \emph{continuously differentiable with derivative} $t \mapsto F'[t]$ if $F$ has a representation of the form
\begin{equation}\label{F-decomp}
 F[t] = \sum_{\alpha \in A} g_\alpha(t) \prod_{i \in [n_\alpha]} \Sigma_{\vec p}( f_{\alpha,i}[t] )
\end{equation}
with $A$ a finite set of labels, and for each $\alpha \in A$, $g_\alpha: I \to \R$ is a continuously differentiable function, $n_\alpha$ is a natural number, and $f_{\alpha,i}: I \to L^\infty(\Disjoint \vec \Omega \to \R)$ is a continuously differentiable function (in the Fr\'echet sense), and $F'$ is given by the Leibniz rule
\begin{equation}\label{leib}
\begin{split}
F'[t] &= \sum_{\alpha \in A} g'_\alpha(t) \prod_{i \in [n_\alpha]} \Sigma_{\vec p}( f_{\alpha,i}[t] )\\
&\quad + \sum_{\alpha \in A} g_\alpha(t) \sum_{i_0 \in [n_\alpha]} \Sigma_{\vec p}( f'_{\alpha,i_0}[t] ) \prod_{i \in [n_\alpha] \backslash \{i_0\}} \Sigma_{\vec p}( f_{\alpha,i}[t] )
\end{split}
\end{equation}
with $g'_\alpha, f'_{\alpha,i_0}$ denoting the derivatives of $g_\alpha, f_{\alpha,i_0}$.  Strictly speaking, we have not ruled out the possibility that the derivative $F'$ is non-unique, due to the fact that $F$ may have multiple representations of the form \eqref{F-decomp}; it is likely that uniqueness does hold, but we do not attempt to establish it here, as it is not needed for our arguments.  We have the usual Leibniz rule: if $F[t], G[t]$ are continuously differentiable in the above sense with derivatives $F'[t], G'[t]$ respectively, then $F[t] G[t]$ is also continuously differentiable with derivative $F'[t] G[t] + F[t] G'[t]$.  It is also clear that the derivative $F'[t]$ depends linearly on $F[t]$.

\begin{proposition}[Differentiation under the integral sign, II]\label{diff-ii}  Let the notation and hypotheses be as in Proposition \ref{diff-i}.  Suppose that $t \mapsto F[t]$ is a continuously differentiable map from $I$ to ${\mathcal L}^\infty(\vec \Omega \to \R)$ with derivative $t \mapsto F'[t]$.  Then one has
\begin{equation}\label{cater-2}
 \partial_t \int_{\vec \Omega^{\vec p}} F[t]\ d\vec \mu_{w[t]}^{\vec p} = \int_{\vec \Omega^{\vec p}} (F[t] \Sigma_{\vec p}(c[t]) + F'[t])\ d\vec \mu_{w[t]}^{\vec p}
\end{equation}
for all $t \in I$.
\end{proposition}

\begin{proof}  Fix $t \in I$, and let $h$ be a sufficiently small real number.  From Proposition \ref{diff-i} (with $F$ replaced by $F[t]$) one already has
$$ \int_{\vec \Omega^{\vec p}} F[t]\ d\vec \mu_{w[t+h]}^{\vec p} = \int_{\vec \Omega^{\vec p}} F[t]\ d\vec \mu_{w[t]}^{\vec p}
+ h \int_{\vec \Omega^{\vec p}} F[t] \Sigma_{\vec p}(c[t])\ d\vec \mu_{w[t]}^{\vec p} + o(|h|)$$
while from Lemma \ref{hold} one has
$$ \int_{\vec \Omega^{\vec p}} hF'[t]\ d\vec \mu_{w[t]}^{\vec p} = \int_{\vec \Omega^{\vec p}} hF'[t]\ d\vec \mu_{w[t+h]}^{\vec p} + o(|h|)$$ 
so by the triangle inequality it suffices to show that
$$ \int_{\vec \Omega^{\vec p}} (F[t+h] - F[t] - hF'[t])\ d\vec \mu_{w[t+h]}^{\vec p} = o(|h|).$$
Using the decomposition \eqref{F-decomp} associated to $F[t]$ and its derivative $F'[t]$, we may assume without loss of generality that $F[t]$ takes the form
$$ F[t] = g(t) \prod_{i \in [n]} \Sigma_{\vec p}( f_i[t] )$$
for some continuously differentiable $g: I \to \R$ and $f_i: I \to L^\infty(\Disjoint \vec \Omega \to \R)$, and that $F'[t]$ similarly takes the form
$$ F'[t] = g'(t) \prod_{i \in [n]} \Sigma_{\vec p}( f_i[t] ) + g(t) \sum_{i_0 \in [n]} \Sigma_{\vec p}(f'_{i_0}[t]) \prod_{i \in [n] \backslash \{i_0\}} \Sigma_{\vec p}(f_i[t]).$$
By Taylor expansion we have $g(t+h) = g(t) + h g'(t) + o(|h|)$, and similarly $f_{i_0}[t+h] = f_{i_0}[t] + h f'_{i_0}[t] + e_{i_0,h}[t]$ where $e_{i_0,h}[t]$ has an $L^\infty$ norm of $o(|h|)$.  Applying these expansions, we eventually conclude that the expression $F[t+h] - F[t] - h F'[t]$ is the finite linear combination (uniformly in $h$) of expressions of the form
$$ g_h(t) \prod_{i \in [n]} \Sigma_{\vec p}(f_{i,h}[t])$$
where all of the quantities $g_h(t), f_{i,h}[t]$ have magnitude bounded uniformly in $h$, and at least one of the quantities of size $o(|h|)$.  Also, from \eqref{tv} we see that $\|\vec \mu_{w[t+h]}\|_{\TV}$ is bounded uniformly in $h$.  The claim now follows from Lemma \ref{hold}.
\end{proof}

\section{Proof of curved multilinear Kakeya estimate}\label{curv-sec}

In this section we prove Theorem \ref{curv-kak}.

We begin with some reductions.  We may normalise
$$ \|\mu_j \|_{\TV} = 1$$
for all $j \in [d]$.  We then have the trivial bound
\begin{equation}\label{dop2}
\int_{\Omega_j} 1_{B^{d-1}(0,t)}( \phi_j( x, \omega_j) )\ d\mu_j(\omega_j) \leq 1
\end{equation}
From this we see that \eqref{dop} trivially holds at $p=\infty$.  Thus by interpolation it suffices to verify \eqref{dop} for bounded choices of $p$, for instance when $\frac{1}{d-1} < p < 2$.  Henceforth we restrict $p$ to this range.

When the maps $\phi_j\colon V \times \Omega_j \to \R^{d-1}$ take the form
\begin{equation}\label{piy} 
\phi_j(x,\omega_j) = x \pi_j - y_j(\omega_j)
\end{equation}
for an arbitrary measurable map $y_j\colon \Omega_j \to \R^{d-1}$, with $\pi_j \in \R^{d \times d-1}$ the $d \times d-1$ matrix formed by deleting the $j^{\mathrm{th}}$ column from the $d \times d$ identity matrix $I_d$, then the required inequality follows easily from the Loomis-Whitney inequality \cite{loomis} (which also holds at the endpoint $p=\frac{1}{d-1}$).  We now reduce to a ``perturbed Loomis-Whitney case'' in which the $\phi_i$ locally behave like the maps \eqref{piy}.

We set a small parameter
\begin{equation}\label{construct}
\eps := C_0^{-1} A^{-C_0} \left(d-1-\frac{1}{p}\right)^{C_0} 
\end{equation}
for some large constant $C_0$ (depending only on $d$) to be chosen later.  It will now suffice to show that
$$ \left\| \prod_{j \in [d]} \int_{\Omega_j} 1_{B^{d-1}(0,t)}( \phi_j( x, \omega_j) )\ d\mu_j(\omega_j) \right\|_{L^p(U_{1/A})} \lesssim A^{O(1)} \eps^{-O(1)} t^{\frac{d}{p}},$$
where the $L^p$ norm is with respect to the $x$ variable.
(One could absorb the $A^{O(1)}$ factor here into the $\eps^{-O(1)}$ factor if desired.)  We may assume that $t \leq \eps^2$ (say), since the claim follows from the trivial bound \eqref{dop} otherwise.

By covering $V_{1/A}$ by $O(\eps^{-O(1)})$ balls of radius $\eps$ (using a maximal $\eps$-separated net of $V_{1/A}$), it suffices to establish the estimate
$$ \left\| \prod_{j \in [d]} \int_{\Omega_j} 1_{B^{d-1}(0,t)}( \phi_j( x, \omega_j) )\ d\mu_j(\omega_j) \right\|_{L^p(B^d(x_0,\eps))} \lesssim A^{O(1)} \eps^{-O(1)} t^{\frac{d}{p}}$$
for any ball $B^d(x_0,\eps)$ with $x_0 \in V$.

For any $\omega_j \in \Omega_j$, the derivative map $\nabla_x^T \phi_j(x_0,\omega_j) \in \R^{d \times d-1}$ has norm $O(A)$.  By partitioning $\Omega_j$ into $O(\eps^{-O(1)})$ regions depending on the value of this map up to errors of size $\eps$ and using the triangle inequality, we may assume without loss of generality that there exist matrices $B^0_j \in \R^{d \times d-1}$ of norm $O(A)$ such that
$$ \nabla_x^T \phi_j(x_0,\omega_j) = B^0_j + O( \eps ) $$
for all $\omega_j \in \Omega_j$, which by \eqref{c2} implies that
\begin{equation}\label{no}
 \nabla_x^T \phi_j(x,\omega_j) = B^0_j + O( A \eps ) 
\end{equation}
for all $\omega_j \in \Omega_j$ and $x \in B(0,\eps)$.  From Theorem \ref{curv-kak}(ii) we conclude that $B^0_j$ is of full rank, with all $d-1$ non-trivial singular values of the form $A^{O(1)}$.

Let $n^0_j \in \R^d$ be a unit vector in the left kernel of $B^0_j$, thus $n^0_j B^0_j = 0$.  From \eqref{no} and the inverse function theorem, we see that for any $\omega_j$, the left kernel of $\nabla_x^T \phi_j(x,\omega_j)$ contains a unit vector in $\R^d$ of the form $n^0_j + O( A^{O(1)} \eps )$.  From  
\eqref{transverse} we then conclude (for $C_0$ large enough) that
$$ \left| \bigwedge_{j \in [d]} n^0_j \right| \gtrsim A^{-O(1)}.$$
In particular, by Cramer's rule, we can find an invertible matrix $T \in \R^{d \times d}$ with singular values $A^{O(1)}$ such that $n_0^j T = e_j$ for all $j \in [d]$.  Applying this transformation (and adjusting $C_0, A, t$ as necessary), replacing $\phi_j$ with the map $(x,\omega_j) \mapsto \phi_j( x T^{-1}, \omega_j)$ and $U$ with $UT$, we may assume without loss of generality that $n^0_j = e_j$ for all $j$.  Thus $B^0_j$ can be factored as $B^0_j = \pi_j \tilde B^0_j$ for some matrix $\tilde B^0_j \in \R^{d-1 \times d-1}$, where $\pi_j \in \R^{d \times d-1}$ is the matrix from \eqref{piy}.  Since $B^0_j$ is of full rank with all non-trivial singular values $A^{O(1)}$, we conclude from \eqref{c2} that $\tilde B^0_j$ is invertible, with all singular values $A^{O(1)}$.  If we then replace each $\phi_j$ with the map $x \mapsto \phi_j(x) (\tilde B^0_j)^{-1}$ (and adjust $C_0$ and $A$ as necessary), we see that we may assume without loss of generality that $\tilde B^0_i$ is the identity, thus by \eqref{no} we now have
\begin{equation}\label{nib}
 \nabla_x^T \phi_j(x,\omega_j) = \pi_j + O( A^{O(1)} \eps )
\end{equation}
(compare with \eqref{piy}).

For any exponent $\alpha \in \R$, let $P(\alpha)$ denote the claim that the estimate
$$ \left\| \prod_{j \in [d]} \int_{\Omega_j} \gamma_t( \phi_j(x,\omega_j) ) \ d\mu_j(\omega_j) \right\|_{L^p(B^d(x_0,\eps))} \lesssim A^{O(1)} \eps^{-O(1)} t^\alpha $$
holds under the above assumptions, where the gaussian weight $\gamma_t$ is defined in \eqref{gammat} or in Section \ref{notation-sec}.  We trivially have $P(\alpha)$ for $\alpha \leq 0$, and thanks to the pointwise bound
$$ 1_{B^{d-1}(0,t)} \lesssim \gamma_t$$
it will suffice to show that $P(d/p)$ holds.  We will show the implication\footnote{There is considerable freedom in the choice of constant $0.1$ here.  This freedom hints that it may be possible to relax the $C^2$ regularity hypothesis in Theorem \ref{curv-kak}(i) to $C^{1,\alpha}$ regularity, although our current arguments do rely on the $C^2$ hypothesis quite heavily when estimating the error terms in Lemma \ref{lowerl}.  It may be possible however to rearrange the argument to reduce the amount of regularity required, perhaps by estimating finite differences of $Q(t)$ rather than first derivatives. We will not pursue this question further here.}
$$ P(\alpha-0.1) \implies P(\alpha)$$
for all $\alpha \leq d/p$, which will give the claim $P(d/p)$ after finitely many iterations of this implication.

Now let $\alpha \leq d/p$ be such that $P(\alpha-0.1)$ holds.  It suffices to show that
\begin{equation}\label{sod}
t^{-\alpha p} \int_{\R^d} \eta_{x_0,\eps}(x) \prod_{j \in [d]} \left(\int_{\Omega_j} \gamma_t(\phi_j(x,\omega_j)) \ d\mu_j(\omega_j) \right)^p
\lesssim \eps^{-O(1)},
\end{equation}
where the cutoff $\eta_{x_0,\eps}$ was defined in Section \ref{notation-sec}.

We now use the virtual integration formalism from Section \ref{fractional}.  Write
$$ \vec \Omega \coloneqq (\Omega_1,\dots,\Omega_d)$$
and
$$\vec \mu \coloneqq (\mu_1,\dots,\mu_d)$$
and
$$ \vec p \coloneqq (p,\dots,p).$$
For each $x \in V$, let $\phi[x]\colon \Disjoint \vec \Omega \to \R^{d-1}$ be the function
$$ \phi[x](j, \omega_j) \coloneqq \phi_j(x, \omega_j)$$
and for any $0 < t \leq 1$, let $w[t,x]\colon \Disjoint \vec \Omega \to \R$ be the non-negative weight
\begin{equation}\label{wit}
 w[t,x] \coloneqq \gamma_t(\phi[x]).
\end{equation}
The left-hand side of \eqref{sod} can now be written using Lemma \ref{basic} (extended to fractional $p$) as
$$ t^{-\alpha p} \int_{\R^d} \eta_{x_0,\eps}(x) \vec \mu_{w[t,x]}(\vec \Omega)^{\vec p}\ dx.$$
To bound this expression, we will perturb it slightly to one that has better monotonicity properties.  For any $x \in V$, the expression
\begin{equation}\label{bx-def}
 B[x] \coloneqq \nabla_x^T \phi[x]
\end{equation}
is an element of $L^\infty(\Disjoint \vec \Omega \to \R^{d \times d-1})$.  Therefore if we define
\begin{equation}\label{M-def}
 M[x] \coloneqq \Sigma_{\vec p}( B[x]^T B[x] )
\end{equation}
then $M[x] \in {\mathcal L}^\infty(\vec \Omega^{\vec p} \to \R^{d \times d})$ is a virtual matrix-valued function, and its determinant
$$ \det M[x] \in {\mathcal L}^\infty(\vec \Omega^{\vec p} \to \R) $$
is a virtual scalar function.  Following \cite{bct}, we now introduce the quantity
$$ Q(t) \coloneqq t^{-\alpha p} \int_{\R^d} \int_{\vec \Omega^{\vec p}} \eta_{x_0,\eps}(x) (\det M[x])\ d\vec \mu_{w[t,x]}^{\vec p}\ dx$$
Applying Theorem \ref{non-neg}(i) (with $B_j^0 = \pi_j$, so that $M^0 = p(d-1) I_d$) we conclude that
\begin{equation}\label{qt}
 Q(t) = ((p(d-1))^d + O(A \eps)) t^{-\alpha p} \int_{\R^d} \eta_{x_0,\eps}(x) \vec \mu_{w[t,x]}(\vec \Omega)^{\vec p}\ dx.
\end{equation}
By the construction \eqref{construct} of $\eps$, we have
$$(p(d-1))^d + O(A \eps) \gtrsim \eps^{O(1)}.$$
Thus, to establish the claim $P(\alpha)$, it will suffice to show that
\begin{equation}\label{qto}
 Q(t) \lesssim A^{O(1)} \eps^{-O(1)}
\end{equation}
for all $0 < t \leq 1$.

From \eqref{qt}, \eqref{dop2}, we already have established \eqref{qto} at the endpoint $t=\eps^2$.  Thus by the fundamental theorem of calculus, it will suffice to establish the monotonicity formula
\begin{equation}\label{mono}
 t \partial_t Q(t) \geq - O( A^{O(1)} \eps^{-O(1)} t^{1-0.1 p} )
\end{equation}
whenever $0 < t < \eps^2$, since we are in the regime $p<2$ which ensures that $t^{-0.1 p}$ is integrable.

We use the abbreviation
$$ \Expect_t( G ) \coloneqq t^{-\alpha p} \int_{\R^d} \int_{\Omega^{\vec p}} G[t,x] \ d\vec \mu_{w[t,x]}^{\vec p} dx $$
whenever $G = G[t,x]$ is a virtual function in ${\mathcal L}^\infty( \Omega^{\vec p} \to \R )$, thus for instance
$$ Q(t) = \Expect_t( \eta_{x_0,\eps} \det M ).$$
We have
$$ t \partial_t w[t,x] = \frac{2}{t^2} \phi[x] \phi[x]^T w[t,x].$$
Applying Proposition \ref{diff-i}, we conclude that
$$ t \partial_t \int_{\Omega^{\vec p}} \eta_{x_0,\eps} \det M \ d\vec \mu_{w[t,x]}^{\vec p} 
= \int_{\Omega^{\vec p}} \frac{2}{t^2} \eta_{x_0,\eps} \det M \Sigma_{\vec p}(\phi[x] \phi[x]^T) \ d\vec \mu_{w[t,x]}^{\vec p} $$
for each $x \in \R^d$.  From Lemma \ref{continuity}, the right-hand side is continuous in both $t$ and $x$.  Applying the fundamental theorem of calculus to convert this differential identity into an integral one, then applying Fubini's theorem (exploiting the compactly supported nature of $\eta_{x_0,\eps}$), then applying the fundamental theorem of calculus once again, one can differentiate under the integral sign and conclude that
\begin{equation}\label{tqt}
 t \partial_t Q(t) = \Expect_t\left( - \alpha p \eta_{x_0,\eps} \det M + \frac{2}{t^2} \eta_{x_0,\eps} \det(M) \Sigma_{\vec p}(\phi \phi^T) \right).
\end{equation}
To manage the $- \alpha p \eta_{x_0} \det M$, term we shall perform a somewhat contrived-looking integration by parts, designed in order to be able to exploit the positivity in Theorem \ref{non-neg}(ii).  We introduce the  virtual column vector-valued function $F = F[x] \in {\mathcal L}^\infty(\Omega^{\vec p} \to \R^{1 \times d})$ defined by the formula
\begin{equation}\label{F-def}
 F \coloneqq \Sigma_{\vec p}( B[x] \phi[x]^T ),
\end{equation}
and consider how the expression
$$
W(t,x) \coloneqq \int_{\Omega^{\vec p}} \eta_{x_0,\eps}(x) \adj(M[x]) F[x]\ d\vec \mu_{w[t,x]}^{\vec p}$$
depends on $x$.  It is not difficult to show that the map $x \mapsto \eta_{x_0,\eps}(x) \adj(M[x]) F[x]$ is continuously differentiable in each coordinate $x_j$ in the sense of Section \ref{diff-sec}; also, from \eqref{wit}, \eqref{gammat}, \eqref{bx-def} we have
$$ \nabla_x^T w[t,x] = -\frac{2}{t^2} B[x] \phi[x]^T w[t,x],$$
and hence by Proposition \ref{diff-ii} and \eqref{F-def} we have
$$ \nabla_x W(t,x)
= \int_{\Omega^{\vec p}} \left(\nabla_x (\eta_{x_0,\eps}(x) \adj(M[x]) F[x]) - \frac{2}{t^2} F[x]^T \eta_{x_0,\eps}(x) \adj(M[x]) F[x]\right) \ d\vec \mu_{w[t,x]}^{\vec p}.$$
From Lemma \ref{continuity}, the right-hand side can be shown to be continuous in $x$.  Thus $W(t,x)$ is continuously differentiable in $x$; as it is also compactly supported, we have
\begin{equation}\label{tash}
 t^{-\alpha p} \int_{\R^d} \nabla_x W(t,x)\ dx = 0
\end{equation}
and thus
$$\Expect_t( \nabla_x (\eta_{x_0,\eps} \adj(M) F) - \frac{2}{t^2} \eta_{x_0,\eps} F^T \adj(M) F ) = 0.$$
We can therefore write \eqref{tqt} as
$$ \Expect_t\left( \nabla_x (\eta_{x_0,\eps} \adj(M) F) - \alpha p \eta_{x_0,\eps} \det M + \frac{2}{t^2} \eta_{x_0,\eps} S_0\right)$$
where $S_0 = S_0[t,x]$ is the virtual function
$$
S_0 \coloneqq \det(M) \Sigma_{\vec p}(\phi \phi^T) - F^T \adj(M) F.
$$
From the Leibniz rule (writing $\nabla_x = \sum_{j \in [d]} e_j \partial_{x_j}$) we have
\begin{align*}
\nabla_x (\eta_{x_0,\eps} \adj(M) F) &= \nabla_x (\eta_{x_0,\eps} \adj(M)) F \\
&\quad + \sum_{j \in [d]} e_j \eta_{x_0,\eps} \adj(M) \Sigma_{\vec p}( (\partial_{x_j} B) \phi^T) \\
&\quad + \mathrm{tr}\left( \eta_{x_0,\eps} \adj(M) \Sigma_{\vec p}\left( B B^T \right)\right).
\end{align*}
From \eqref{M-def} and the identity $\adj(M) M = \det(M) I_d$ we have
$$\mathrm{tr}\left( \eta_{x_0,\eps} \adj(M) \Sigma_{\vec p}\left(BB^T\right) \right) = \mathrm{tr}( \eta_{x_0,\eps} \det(M) I_d) = d \eta_{x_0,\eps} \det(M).$$
We can thus write
$$ t \partial_t Q(t) = \Expect_t\left( \frac{2}{t^2} \eta_{x_0,\eps} S_0 + S_1 + S_2 \right) + (d-\alpha p) Q(t)$$
where $S_i = S_i[t,x]$, $i=1,2$ are the virtual functions
\begin{align*}
S_1 &\coloneqq \nabla_x (\eta_{x_0,\eps} \adj(M)) F  \\
S_2 &\coloneqq \sum_{j \in [d]} e_j \eta_{x_0,\eps} \adj(M) \Sigma_{\vec p}((\partial_{x_j} B) \phi^T).
\end{align*}

We can use the $P(\alpha-0.1)$ hypothesis to control the lower order terms $S_1, S_2$:

\begin{lemma}\label{lowerl}  We have $\Expect_t(S_i)\lesssim \eps^{-O(1)} t^{1-0.1 p}$ for $i=1,2$.
\end{lemma}

\begin{proof} By the product rule and \eqref{F-def}, $S_i$ is a linear combination (with coefficients $O(1)$) of $O(1)$ terms of the form
$$ \Sigma_{\vec p}(f_1) \dots \Sigma_{\vec p}(f_n) \Sigma_{\vec p}(F_{n+1} \phi^{l'})$$
where $n = O(1)$, $l' \in [d-1]$, and $f_1,\dots,f_{n+1}$ are equal to either $\eta_{x_0,\eps}$, $\partial_{x_j} \eta_{x_0,\eps}$, $\partial_{x_j} \phi^l$, or $\partial_{x_j} \partial_{x_{j'}} \phi^l$ for some $j,j' \in [d]$, $l \in [d-1]$, where $\phi^1,\dots,\phi^{d-1}$ are the components of $\phi$.  All of the functions $f_1,\dots,f_{n+1}$ are bounded uniformly by $O(\eps^{-O(1)})$ for $x$ in the support $B^{d-1}(x_0,2\eps)$ of $\eta_{x_0,\eps}$.  By Lemma \ref{holder}, we thus have
$$ \Expect_t(S_i)  \lesssim \eps^{-O(1)} t^{-\alpha p}
\int_{B^d(x_0,2\eps)} \sum_{j \in [d]} \E_{(\mu_j)_{w_j[t,x]}} |\phi_j[x]|\vec \mu_{w[t,x]}(\vec \Omega)^{\vec p}\ dx.$$
From the pointwise estimate
$$ \phi_j[x] w_j[t,x] \lesssim t w_j[t/2,x]^{\min(p,1)} w_j[t,x]^{1-\min(p,1)}$$
and H\"older's inequality, we have
$$ \E_{(\mu_j)_{w_j[t,x]}} |\phi_j[x]| \lesssim t ( \| (\mu_j)_{w_j[t/2,x]} \|_{\TV} / \|(\mu_j)_{w_j[t,x]}\|_{\TV} )^{\min(p,1)};$$
since we also have $w_j[t/2,x] \geq w_j[t,x]$, we conclude that
$$ \E_{(\mu_j)_{w_j[t,x]}} |\phi_j[x]| \lesssim t ( \| (\mu_j)_{w_j[t/2,x]} \|_{\TV} / \|(\mu_j)_{w_j[t,x]}\|_{\TV} )^{p}.$$
We therefore have
$$ \Expect_t(S_i) \lesssim \eps^{-O(1)} t^{1-\alpha p}
\int_{B^d(x_0,2\eps)} \vec \mu_{w[t/2,x]}(\vec \Omega)^{\vec p}\ dx.$$
Using the induction hypothesis $P(\alpha-1)$, we thus have
$$ \Expect_t(S_i) \lesssim \eps^{-O(1)} t^{1-\alpha p} t^{(\alpha-0.1) p}$$
giving the claim.
\end{proof}

Since $(d-\alpha p) Q(t)$ is non-negative, to conclude the proof of \eqref{mono}, it thus suffices to establish the non-negativity
$$ \Expect_t(\eta_{x_0,\eps} S_0)\geq 0.$$
In fact we will establish the stronger pointwise bound
$$
 \int_{\vec \Omega^{\vec p}} S_0[t,x] \ d\vec \mu_{w[t,x]}^{\vec p} \geq 0
$$
for each $0< t \leq \eps^{10}$ and $x \in U$.  But this follows from Theorem \ref{non-neg}(ii) (with $B_j^0 = \pi_j$, so that $M_0 = p(d-1) I_d$), since the positive semi-definite nature of $(1 - CA^C \eps) M_0 - B_j^0 = p(d-1) (1-CA^C \eps) I_d - \pi_j$ follows from the construction of $\eps$.

\section{Proof of multilinear restriction estimate}\label{rest-sec}

We now begin the proof of Theorem \ref{lamrc}, again beginning with some basic reductions.  We let implied constants in our asymptotic notation depend on $d$.  From Plancherel's theorem we have the pointwise bound
\begin{equation}\label{uoi}
 \energy_r[{\mathcal E}_j f_j](x',x_d) \leq \int_{\R^{d-1}} |{\mathcal E}_j f_j(x',x_d)|^2\ dx' = \| f_j \|_{L^2(U_{j,1/A})}^2
\end{equation}
for any $(x',x_d) \in \R^d$, which immediately gives the $p=\infty$ case of Theorem \ref{lamrc}.  Thus by interpolation it suffices to restrict attention to the case of bounded $p$, for instance $\frac{1}{d-1} < p \leq 2$.

As in the preceding section, we set a small parameter
\begin{equation}\label{construct-2}
 \eps := C_0^{-1} A^{-C_0} \left(d-1-\frac{1}{p}\right)^{C_0} 
\end{equation}
for some large constant $C_0$ (depending only on $d$) to be chosen later.   It suffices to show that
\begin{equation}\label{clap}
r^{-d} \left\| \prod_{j \in [d]} \energy_{r}[{\mathcal E}_j f_j] \right\|_{L^{p}(\R^d)}^{p} \lesssim A^{O(1)} \eps^{-O(1)} \prod_{j \in [d]} \| f_j \|_{L^2(U_{j,1/A})}^{2p}.
\end{equation}

By covering $U_{j,1/A}$ by $O(\eps^{-O(1)})$ balls of radius $\eps$, we may assume without loss of generality that each $f_j$ is supported in a ball $B^{d-1}(\xi_j^0, \eps)$ of radius $\eps$, with $\xi_j^0 \in U_{j,1/A}$.  For $1 \leq r \leq R$, let $C(R,r)$ denote the best constant in the estimate
\begin{equation}\label{crat}
r^{- d} \left\| \prod_{j \in [d]} \energy_{r}[{\mathcal E}_j f_j]\right \|_{L^p(B^d(x_0,R))}^{p} \leq C(R,r) \prod_{j \in [d]} \| f_j \|_{L^2(B^{d-1}(\xi_j^0, \eps))}^{2p}
\end{equation}
where $x_0 \in \R$ and $f_j \in L^2(B^{d-1}(\xi_j^0, \eps))$.  This is clearly a finite quantity.  We will introduce a variant $\tilde C(R,r)$ of $C(R,r)$ and verify the following claims:
\begin{itemize}
\item[(i)]  For any $1 \leq r \leq R$, we have $C(R,r) = A^{O(1)} \tilde C(R,r)$.
\item[(ii)]  For any $1 \leq r \leq R$, we have $C(R,r) \lesssim (R/r)^{O(1)}$.
\item[(iii)]  If $\eps^{-2} \leq r_1 \leq r_2 \leq R$ with $r_2 \leq r_1^{1.1}$, we have
\begin{equation}\label{booyah}
 \tilde C(R,r_1) \leq \tilde C(R,r_2) + O\left( A^{O(1)} \delta\sup_{r_1 \leq t \leq r_2} C(R,t)\right)
\end{equation}
where $\delta = \delta_{r_1,r_2}$ is the quantity
\begin{equation}\label{deta}
 \delta \coloneqq ( r_1^{-c} + (r_2/R)^{c} ) \end{equation}
and $c>0$ depends only on $d$.
\end{itemize}
Let us assume (i), (ii), (iii) for now and establish \eqref{clap}.  Combining (iii) with (ii) we have
$$ \tilde C(R,r_1) \leq \tilde C(R,r_2) + O\left( A^{O(1)} \delta \sup_{r_1 \leq t \leq r_2} \tilde C(R,t) \right)$$
whenever $\eps^{-2} \leq r_1 \leq r_2 \leq R$ and $r_2 \leq r_1^{1.1}$.  Replacing $r_1$ by any quantity in $[r_1,r_2]$ (which does not increase $\delta$) and taking suprema, we conclude that
$$ \sup_{r_1 \leq t \leq r_2} \tilde C(R,t) \leq \tilde C(R,r_2) + O\left( A^{O(1)} \delta \sup_{r_1 \leq t \leq r_2} \tilde C(R,t)\right) $$
If in addition we assume that $r_1, R/r_2 \geq \eps^{-C_1}$ for a sufficiently large constant $C_1$ (depending only on $d$), we can rearrange this as
$$ \sup_{r_1 \leq t \leq r_2} \tilde C(R,t)  \leq \left(1 + O( A^{O(1)} \delta)\right ) \tilde C(R,r_2)$$
and in particular
$$ \tilde C(R,r_1)  \leq \exp\left( O( A^{O(1)} \delta) \right) \tilde C(R,r_2)$$
Iterating this starting with some $\eps^{-C_1} \leq r \leq \eps^{C_1} R$ using the sequence $r_1,r_2,\dots$ defined by $r_1 = r$ and $r_{n+1} = \min(r_n^{1.1}, \eps^{C_1} R)$, and summing the geometric series arising from substituting \eqref{deta}, we conclude that
$$ \tilde C(R,r) \lesssim \tilde C(R, \eps^{C_1} R)$$
whenever $\eps^{-C_1} \leq r \leq \eps^{C_1} R$.  Combining this with (i), (ii), we conclude that
$$ C(R,r) \lesssim A^{O(1)} \eps^{-O(C_1)}$$
whenever $\eps^{-C_1} \leq r \leq \eps^{C_1} R$.  Sending $R$ to infinity using monotone convergence, we conclude that \eqref{clap} holds whenever $r \geq \eps^{-C_1}$.  Finally, the remaining cases $1 \leq r \leq \eps^{-C_1}$ follow from the fact (from \eqref{Avg-def}) that $\energy_r[{\mathcal E}_j f_j]$ is monotone non-decreasing in $r$.

It remains to establish claims (i), (ii), (iii).  Claim (ii) is immediate from \eqref{uoi}.  For the other two claims we need to define the quantity $\tilde C(R,r)$.  We first need some matrices adapted to the normal vectors $n_j(\xi)$ associated to the functions ${\mathcal E}_j f_j$ that were defined in \eqref{njxi}:

\begin{lemma}\label{matrix-lem}  There exist matrices $B'_j \in \R^{d-1 \times d-1}$ for $j \in [d]$ with all singular values $A^{O(1)}$ with the following property: if for any $\xi_j \in B( \xi_j^0, 2\eps)$ we define the matrix $B_j(\xi_j) \in \R^{d \times d-1}$ by
\begin{equation}\label{tubx}
B_j(\xi_j) \coloneqq \begin{pmatrix} I_{d-1} & \nabla_\xi h_j(\xi_j)^T \end{pmatrix} B'_j
\end{equation}
(so in particular $n_j(\xi_j)$ is a left null vector of $B_j(\xi_j)$), then the matrix $M_0 \in \R^{d \times d}$ defined by
\begin{equation}\label{modo}
 M_0 \coloneqq \sum_{j \in [d]} p B_j(\xi_j^0) B_j(\xi_j^0)^T
\end{equation}
has all singular values $A^{O(1)}$, and $\frac{1}{p(d-1)} M_0 - B_j(\xi_j^0) B_j(\xi_j^0)^T$ is positive semi-definite for all $j \in [d]$.
\end{lemma}

\begin{proof}  Let $S \in \R^{d \times d}$ be the matrix with rows $n_j(\xi_j^0)$ for $j \in [d]$.  From \eqref{regular}, \eqref{transverse} and Cramer's rule one sees that $S$ is invertible, with all singular values $A^{O(1)}$.  If $\pi_j \in \R^{d \times d-1}$ are the matrices from Section \ref{curv-sec}, then $n_j(\xi_j^0)$ lies in the left null space of $S^{-1} \pi_j$, and hence we can write
$$ S^{-1} \pi_j = \begin{pmatrix} I_{d-1} & \nabla_\xi h_j(\xi_j)^T \end{pmatrix} B'_j$$
where $B'_j \in \R^{d-1 \times d-1}$ is formed from $S^{-1} \pi_j$ by deleting the bottom row.  Since $\pi_j$ has $d-1$ singular values of $1$ and one singular value of $0$, $S$ has all singular values $A^{O(1)}$, and the matrix $\begin{pmatrix} I_{d-1} & \nabla_\xi h_j(\xi_j)^T \end{pmatrix}$ has $d-1$ singular values of $A^{O(1)}$ and one singular value of $0$, we see that $B'_j$ has all singular values $A^{O(1)}$.  By construction we have
$$ B_j(\xi_j^0) B_j(\xi_j^0)^T = S^{-1} \pi_j \pi_j^T (S^{-1})^T$$
and hence
$$ M_0 = p(d-1) S^{-1} (S^{-1})^T$$
so that $\frac{1}{p(d-1)} M_0 -B_j(\xi_j^0) B_j(\xi_j^0)^T$ is positive semi-definite as required.
\end{proof}

Henceforth $B'_1,\dots,B'_d$ are as in the above lemma.

Let $f_j \in L^2(B^{d-1}(\xi_j^0, \eps))$, which we can take to be Schwartz functions not identically zero; the general case can then be handled by a limiting argument.
We introduce a (squared) Gabor-type transform of the ${\mathcal E}_j f_j$.  
For any $r \geq \eps^{-2}$, we define the function $G_{j,r}\colon \R^d \times \R^{d-1} \to \R$ by the formula
\begin{equation}\label{gr-def}
 G_{j,r}((x',x_d), \xi_j) \coloneqq r^{-0.9(d-1)} \left| \int_{\R^{d-1}} {\mathcal E}_j f_j( y', x_d ) e^{-2\pi i \xi_j (y')^T} \varphi_{x',r^{0.9}}(y')\ dy' \right|^2
\end{equation}
where the rescaled cutoffs $\varphi_{x',r^{0.9}}$ are defined in Section \ref{notation-sec}.
Because the Fourier transform of ${\mathcal E}_j f_j( y', x_d )$ is supported on $B^{d-1}(\xi_j^0, \eps)$, and the Fourier transform of
$\varphi_{x',r^{0.9}}$ is supported in $B^{d-1}(0,r^{-0.9})$, we see that $G_{j,r}((x',x_d),\xi_j)$ vanishes unless $\xi \in B^{d-1}(\xi_j^0,2\eps)$; also, as $f_j$ is Schwartz, $G_{j,r}$ will be rapidly decreasing in the $x'$ variable.  Furthermore, from Plancherel's theorem we have
\begin{equation}\label{job}
 \int_{B^{d-1}(\xi_j^0,2\eps)} G_{j,r}((x',x_d), \xi_j)\ d\xi_j = r^{-0.9(d-1)} \int_{\R^{d-1}} |{\mathcal E}_j f_j( y', x_d )|^2 \varphi_{x',r^{0.9}}(y')^2\ dy',
\end{equation}
thus by \eqref{norm} the marginal $\int_{\R^{d-1}} G_{j,r}((x',x_d), \xi_j)\ d\xi_j$ is an averaged version of $|{\mathcal E}_j f_j( x', x_d )|^2$ at scale $r^{0.9}$ in the $x'$ variable.  Informally, $G_{j,r}(x,\xi_j)$ measures the energy density of ${\mathcal E}_j f_j$ at the physical location $x + O(r^{0.9})$ and at horizontal frequencies $\xi_j + O(r^{-0.9})$.

For each $x = (x',x_d)$, let $(\vec \Omega', \vec \mu'[r,x])$ be the $d$-tuple of measure spaces $(\Omega'_j, \mu'_j[r,x])$, where
$$ \Omega'_j \coloneqq \R^{d-1} \times B^{d-1}(\xi_j^0, 2\eps)$$
(parameterised by $(y',\xi_j)$ with $y' \in \R^{d-1}$ and $\xi_j \in B^{d-1}$) 
and
\begin{equation}\label{demure}
 d\mu'_j[r,x] \coloneqq G_{j,r}((y',x_d),\xi_j) \gamma_r( (x'-y') B'_j )\ dy' d\xi,
\end{equation}
with the gaussian cutoff $\gamma_r$ defined by \eqref{gammat} or Section \ref{notation-sec}.
From \eqref{job} we see that $\mu'_j$ is a finite measure that is not identically zero.
We have a matrix function $ B \colon \Disjoint \vec \Omega' \to \R^{d \times d-1}$ defined by
\begin{equation}\label{bud}
  B( j, (y',\xi_j) ) \coloneqq  B_j(\xi_j)
	\end{equation}
where $ B_j$ was defined in \eqref{tubx}.  These matrices will play the role that $B[x]$ did in the preceding section.  We then set
$$ \vec p \coloneqq (p,\dots,p)$$
and define the virtual matrix-valued function $M \in {\mathcal L}^\infty((\vec \Omega')^{\vec p} \to \R^{d \times d})$ by
\begin{equation}\label{mdef}
 M \coloneqq \Sigma_{\vec p}(  B  B^T ).
\end{equation}
In contrast to the previous section, the quantity $M$ is now independent of $x$, which simplifies the situation slightly by eliminating some lower order error terms.  The determinant $\det M \in {\mathcal L}^\infty((\vec \Omega')^{\vec p} \to \R)$ is then a virtual scalar function.  We then define $\tilde C(R,r)$ for $\eps^{-2} \leq r \leq R$ to be the best constant for which one has the inequality
\begin{equation}\label{ineq}
r^{-d} \int_{\R^d} \eta_R(x) \int_{(\vec \Omega')^{\vec p}} \det M\ d\vec \mu'[r,x]^{\vec p}\ dx \leq \tilde C(R,r)
\prod_{j \in [d]} \| f_j \|_{L^2(B^{d-1}(\xi_j^0, \eps))}^{2p}
\end{equation}
for all Schwartz functions $f_j \colon B^{d-1}(\xi_j^0,\eps) \to \C$ that are not identically zero, where $\eta_R$ is defined in Section \ref{notation-sec}.  By a standard limiting argument, the inequality \eqref{ineq} then in fact holds for all $f_j \in L^2(B^{d-1}(\xi_j^0, \eps) \to \C)$.

We now establish Claim (i). We begin with the upper bound $\tilde C(R,r) \lesssim A^{O(1)} C(R,r)$.  Let $f_j \colon B^{d-1}(\xi_j^0,\eps) \to \C$ be Schwartz functions not identically zero.
From \eqref{tubx} and the bounds on $B'_j$ and $\xi_j^0$ we have
$$ B_j(\xi_j) B_j^T(\xi_j) = B_j^0 (B_j^0)^T + O(A^{O(1)} \eps )$$
for all $(j, (y',\xi_j)) \in \Disjoint \vec \Omega$, and hence by Theorem \ref{non-neg}(i), the left-hand side of \eqref{ineq} is equal to
\begin{equation}\label{tang}
 (\det(M_0) + O(A^{O(1)} \eps)) r^{-d} \int_{\R^d} \eta_R(x) \vec \mu'[r,x]( \vec \Omega' )^{\vec p}\ dx.
\end{equation}
By \eqref{demure}, \eqref{job} one has
\begin{equation}\label{joey}
\begin{split}
 \| \mu'_j[r,x] \|_{\TV} &= 
\int_{B^{d-1}(\xi_j^0, 2\eps)} \int_{\R^{d-1}} G_{j,r}((y',x_d),\xi_j) \gamma_r((x'-y') B'_j)\ dy' d\xi \\
&= 
r^{-0.9(d-1)} \int_{\R^{d-1}} \int_{\R^{d-1}} |{\mathcal E}_j f_j( z', x_d )|^2 \varphi_{y',r^{0.9}}( z' )^2 \gamma_r( (x'-y') B'_j )\ dz' dy'.
\end{split}
\end{equation}
Since all the singular values of $B'_j$ are $A^{O(1)}$, and $\varphi$ is rapidly decreasing, one has
\begin{equation}\label{chang}
 r^{-0.9(d-1)} \int_{\R^{d-1}} \varphi_{y',r^{0.9}}(z')^2 \gamma_t( (x'-y') B'_j )\ dy' \lesssim A^{O(1)} \rho_{x',t}(z')
\end{equation}
for all $t \geq r^{0.9}$, where $\rho_{x',t}$ is defined in \eqref{rhoar} or Section \ref{notation-sec}.  Applying this with $t=r$ and using the Fubini-Tonelli theorem to evaluate the $y'$ integral first, one concludes
\begin{align*}
 \| \mu'_j \|_{\TV} &\lesssim A^{O(1)} \int_{\R^{d-1}} |{\mathcal E}_j f_j( z', x_d )|^2 \rho_{x',r}(z') \\
&\lesssim A^{O(1)} \energy_r[{\mathcal E}_j f_j](x)
\end{align*}
thanks to \eqref{Avg-def}.  Thus
$$ \vec \mu'[r,x]( \vec \Omega' )^{\vec p} \lesssim A^{O(1)} \prod_{j \in [d]} \energy_r[{\mathcal E}_j f_j](x)$$
and hence by \eqref{tang}, \eqref{crat} we have
$$
r^{-d} \int_{\R^d} \eta_R(x) \int_{(\vec \Omega')^{\vec p}} \det M\ d\vec \mu'[r,x]^{\vec p}\ dx \lesssim A^{O(1)} C(R,r)
\prod_{j \in [d]} \| f_j \|_{L^2(B^{d-1}(\xi_j^0, \eps))}^{2p}$$
which when compared with \eqref{ineq} gives the upper bound $\tilde C(R,r) \lesssim A^{O(1)} C(R,r)$.

Now we prove the matching lower bound $\tilde C(R,r) \gtrsim A^{O(1)} C(R,r)$.  From Lemma \ref{matrix-lem}, the quantity $\det(M_0) + O(A^{O(1)} \eps)$ appearing in \eqref{tang} is equal to $A^{O(1)}$, thus by \eqref{ineq}
$$
r^{-d} \int_{\R^d} \eta_R(x) \vec \mu'[r,x]( \vec \Omega' )^{\vec p}\ dx \lesssim A^{O(1)}
\tilde C(R,r) \prod_{j \in [d]} \| f_j \|_{L^2(B^{d-1}(\xi_j^0, \eps))}^{2p}$$
for any Schwartz $f_j \colon B^{d-1}(\xi_j^0,\eps) \to \C$ not identically zero.  Similar to \eqref{chang}, we have a lower bound
$$
 r^{-0.9(d-1)} \int_{\R^{d-1}} \varphi_{y',r^{0.9}}(z')^2 \gamma_r((x'-y') B'_j )\ dy' \gtrsim 1$$
whenever $z' \in B^{d-1}(x', A^{-C} r)$ for some large absolute constant $C \geq 1$ depending only on $d$, since in this case we would have $\gamma_r((x'-y') B'_j) \sim 1$ for $y' \in B^{d-1}(z', r^{0.95})$ (say), and then one can use \eqref{norm} and the rapid decay of $\varphi$.  Using \eqref{joey}, we conclude that
$$ \| \mu'_j[r,x] \|_{\TV}  \gtrsim E_{j,r}(x',x_d) $$
where the local energies $E_{j,r}(x',x_d)$ are defined by
$$ E_j(x',x_d) \coloneqq \int_{B^{d-1}(x', A^{-C} r)} |{\mathcal E}_j f_j( z', x_d )|^2\ dz'.$$
We conclude that
$$ r^{-d} \int_{B^d(0,R)} \prod_{j \in [d]} E_{j,r}(x)^p
\lesssim A^{O(1)} \tilde C(R,r) \prod_{j \in [d]} \| f_j \|_{L^2(B^{d-1}(\xi_j^0, \eps))}^{2p}.$$
Modulating $f_j$ by various phases in order to spatially translate ${\mathcal E}_j f_j$, we conclude that
\begin{equation}\label{top}
 r^{-d} \int_{B^d(0,R)} \prod_{j \in [d]} E_{j,r}(x+h_j)^p\ dx
\lesssim A^{O(1)} \tilde C(R,r) \prod_{j \in [d]} \| f_j \|_{L^2(B^{d-1}(\xi_j^0, \eps))}^{2p}
\end{equation}
uniformly for all shifts $h_j \in \R^{d}$.

Next, by a partition of unity, we see that
$$ \energy_r[{\mathcal E}_j f_j](x',x_d) \lesssim A^{O(C)} \sum_{k_j \in \Z^{d-1}} \langle k_j \rangle^{-10d^2} E_{j,r}(z' + A^{-2C} k_j, x_d ).$$
Applying H\"older's inequality when $p \geq 1$, and $(\sum_k c_k)^p \leq \sum_k c_k^p$ for $p<1$, we conclude that
$$ \energy_r[{\mathcal E}_j f_j](x',x_d)^p \lesssim A^{O(C)} \sum_{k_j \in \Z^{d-1}} \langle k_j \rangle^{-10d^2 \min(p,1)} E_{j,r}(z' + A^{-2C} k_j, x_d )^{p}.$$
Since $p > \frac{1}{d-1}$, the series $\sum_{k_j \in \Z^{d-1}} \langle k_j \rangle^{-10d^2 \min(p,1)} $ is absolutely convergent.  By the triangle inequality and \eqref{top}, we conclude that
$$ r^{-d} \int_{B^d(0,R)} \prod_{j \in [d]} \energy_r[{\mathcal E}_j f_j](x)^p\ dx
\lesssim A^{O(1)} \tilde C(R,r) \prod_{j \in [d]} \| f_j \|_{L^2(B^{d-1}(\xi_j^0, \eps))}^{2p}$$
and on comparison with \eqref{crat} one obtains the desired lower bound $\tilde C(R,r) \gtrsim A^{O(1)} C(R,r)$ (using the usual density argument to remove the Schwartz hypothesis on $f_j$).

It remains to establish Claim (iii).  This is accomplished in three stages.  Firstly, by using the dispersion relation of ${\mathcal E}_j f_j$, we replace the ``horizontal'' virtual integral in \eqref{ineq} with a ``spatial'' virtual integral in which different values of $x_d$ interact with each other at a horizontal scale of $r_1$ (and a slightly larger vertical scale, which we will set somewhat arbitrarily to be $r_1^{1.2}$, though in practice any scale that is genuinely between $r_2$ and $r_1^2$ would have worked here).  Then, using a variant of the heat flow monotonicity arguments in the previous section, we bound this integral (up to small errors) by a similar integral in which the horizontal scale has increased to $r_2$.  Finally, we use the propagation properties of ${\mathcal E}_j f_j$ again to revert back to a horizontal virtual integral, which (again up to small errors) can be controlled by $\tilde C(R,r_2)$.

We turn to the details.  Suppose that $\eps^{-2} \leq r_1 \leq r_2 \leq R$ with $r_2 \leq r_1^{1.1}$.  We may normalise $\| f_j \|_{L^2(B^{d-1}(\xi_j^0, \eps))} = 1$ for all $j \in [d]$, and our task is now to show that
\begin{align*}
&r_1^{-d} \int_{\R^d} \eta_R(x) \int_{(\vec \Omega')^{\vec p}} \det M\ d\vec \mu'[r_1,x]^{\vec p}\ dx \\
&\quad \leq \tilde C(R,r_2) + O\left( A^{O(1)} \delta \sup_{r_1 \leq t \leq r_2} C(R,t)\right).
\end{align*}

For any time $t>0$ and any $x = (x',x_d) \in \R^d$, we introduce a new $d$-tuple $(\vec \Omega, \vec \mu[t,x])$ of measure spaces $(\Omega_j, \mu_j[t,x])$ by setting
$$ \Omega_j \coloneqq \R^{d} \times B^{d-1}(\xi_j^0, 2\eps)$$
(parameterised by $((y',y_d),\xi_j)$ with $y' \in \R^{d-1}$, $y_d \in \R^d$ and $\xi_j \in B^{d-1}(\xi_j^0,2\eps)$, and setting $y = (y',y_d)$) 
and
\begin{equation}\label{muir}
 d\mu_j[t,x] \coloneqq r_1^{-1.2} G_{j,r_1}(y,\xi_j) \gamma_t( (x-y) B_j(\xi_j) ) \gamma_{r_1^{1.2}}^{(1)}( x_d - y_d )\ dy d\xi_j
\end{equation}
where the weight $\gamma_{r_1^{1.2}}^{(1)}$ is defined in Section \ref{notation-sec}.  As before, we can define a matrix function $ B \colon \Disjoint \vec \Omega \to \R^{d \times d-1}$ by the formula
$$  B( j, (y,\xi_j) ) \coloneqq  B_j(\xi_j)$$
and then define the virtual matrix-valued function $M \in {\mathcal L}^\infty((\vec \Omega)^{\vec p} \to \R^{d \times d})$ by the formula \eqref{mdef}.  We then define the quantity
$$
Q(t) \coloneqq t^{-d} \int_{\R^d} \eta_R(x) \int_{\vec \Omega^{\vec p}} \det M\ d\vec \mu[t,x]^{\vec p}\ dx $$
for any $t>0$.

The virtual integral defining $Q(t)$ involves all of the spatial domain $\R^d$, and not just the horizontal slice $\R^{d-1} \times \{x_d\}$.  On the other hand, the \emph{dispersion relation} for ${\mathcal E}_j$ asserts, roughly speaking, that components of ${\mathcal E}_j f_j$ at (horizontal) frequency $\xi_j$ should propagate in the direction $n_j(\xi_j)$ defined in \eqref{njxi}.  This suggests that one can approximately express $Q(t)$ as an analogous integral that only requires evaluating ${\mathcal E}_j f_j$ on the horizontal slice $\R^{d-1} \times \{x_d\}$.  This is indeed the case:

\begin{lemma}[Horizontal approximation of $Q(t)$]\label{compar}  For all $r_1 \leq t \leq r_2$, one has
$$
Q(t) = t^{-d} \int_{\R^d} \eta_R(x) \int_{(\vec \Omega')^{\vec p}} \det M\ d\vec \mu'[t,x]^{\vec p}\ dx + O( A^{O(1)} \delta C(R,t) ).$$
\end{lemma}

\begin{proof}
From \eqref{crat} one has
$$
t^{- d} \int_{\R^d} \eta_R(x)\prod_{j \in [d]} \energy_{t}[{\mathcal E}_j f_j](x)^p \lesssim C(R,t)$$
so it will suffice to establish the pointwise bound
$$ \int_{\vec \Omega^{\vec p}} \det M\ d\vec \mu[t,x]^{\vec p} = \int_{(\vec \Omega')^{\vec p}} \det M\ d\vec \mu'[t,x]^{\vec p} + O\left( A^{O(1)} \delta \prod_{j \in [d]} \energy_t[{\mathcal E}_j f_j](x)^p\right)$$
for all $x \in \R^d$.

Fix $x = (x',x_d)$; we now abbreviate $\vec \mu[t,x],\mu'[t,x]$ as $\vec \mu, \vec \mu'$ respectively.  We abbreviate the estimate
$$ X = Y + O\left( A^{O(1)} \delta \prod_{j \in [d]} \energy_t[{\mathcal E}_j f_j](x)^p\right)$$
as
$$ X \approx Y,$$
thus our task is now to show that
$$ \int_{\vec \Omega^{\vec p}} \det M\ d\vec \mu^{\vec p} \approx \int_{(\vec \Omega')^{\vec p}} \det M\ d(\vec \mu')^{\vec p}.$$
It is convenient to introduce the function $F_j\colon \R^{d-1} \to \C$ for $j \in [d]$ by
$$ F_j(y') \coloneqq {\mathcal E}_j f_j(y',x_d)$$
so that by \eqref{Avg-def}
\begin{equation}\label{energy-relation}
 \energy_t[{\mathcal E}_j f_j](x) = \int_{\R^{d-1}} \rho_{x',t}(y') |F_j(y')|^2\ dy'.
\end{equation}

By making the change of variables 
\begin{equation}\label{changer}
y = (y',x_d) + r_1^{1.2} s n_j(\xi_j)
\end{equation}
in $\Omega_j$, with $n_j(\xi_j)$ defined by \eqref{njxi}, we obtain a projection map $\pi_j\colon \Omega_j \to \Omega'_j$ defined by
$$ \pi_j( (y',x_d) + r_1^{1.2} s n_j(\xi_j), \xi_j ) \coloneqq (y', \xi_j ),$$
which pushes forward the measure $\mu_j$ to the measure $\nu'_j$ defined by
$$ d\nu'_j \coloneqq \left( \int_\R G_{j,r_1}(y,\xi_j) \gamma_t((x-y) B_j(\xi_j)) \gamma^{(1)}(s)\ ds\right) dy' d\xi_j$$
with $y$ defined by \eqref{changer}.  Also, the virtual function $M$ on $(\vec \Omega')^{\vec p}$ pulls back by these maps to the virtual function also denoted $M$ on $\vec \Omega^{\vec p}$, because the maps $\pi_j$ do not affect the $\xi_j$ variable.  Setting $\vec \nu'$ to be the $d$-tuple of measures $\nu'_j$, we thus conclude from \eqref{change-multi} that
$$ \int_{\vec \Omega^{\vec p}} \det M\ d\vec \mu^{\vec p} = \int_{(\vec \Omega')^{\vec p}} \det M\ d(\vec \nu')^{\vec p}.$$
Since $n_j(\xi_j)$ is in the null space of $B_j(\xi_j)$, we can rewrite $d\nu'_j$ as
$$ d\nu'_j = \left( \int_\R G_{j,r_1}((y',x_d) + r^{1.2} sn_j(\xi_j),\xi_j) \gamma^{(1)}(s)\ ds\right)  \gamma_t((x'-y') B'_j)\ dy' d\xi_j.$$
If for any $r>0$ we define the tuple $\vec \mu'_r =  \vec \mu'_r[t,x]$ of measures $\mu'_{j,r} = \mu'_{j,r}[t,x]$ on $\Omega'_j$ by
\begin{equation}\label{dmuj}
 d\mu'_{j,r} \coloneqq G_{j,r}((y',x_d),\xi_j) \gamma_t((x'-y') B'_j)\ dy' d\xi_j
\end{equation}
then $\vec \mu' = \vec \mu'_t$ and
\begin{equation}\label{numu}
d\nu'_j - d\mu'_{j,r_1} = \left( \int_\R H( y',x_d,s,\xi_j) \gamma^{(1)}(s)\ ds\right) \gamma_t((x'-y') B'_j)\ dy' d\xi_j.
\end{equation}
where
$$ H(y',x_d,s,\xi_j) \coloneqq G_{j,r_1}((y',x_d) + r_1^{1.2} sn_j(\xi_j),\xi_j) - G_{j,r_1}((y',x_d),\xi_j).$$
It will now suffice to obtain the bounds
\begin{equation}\label{claim-1}
\int_{(\vec \Omega')^{\vec p}} \det M\ d(\vec \nu')^{\vec p} \approx \int_{(\vec \Omega')^{\vec p}} \det M\ d(\vec \mu'_{r_1})^{\vec p}
\end{equation}
and
\begin{equation}\label{claim-2}
\int_{(\vec \Omega')^{\vec p}} \det M\ d(\vec \mu'_{r_1})^{\vec p} \approx
\int_{(\vec \Omega')^{\vec p}} \det M\ d(\vec \mu'_{t})^{\vec p}.
\end{equation}

We first establish \eqref{claim-1}.
Applying Lemma \ref{hold} and the triangle inequality, it suffices to show that
\begin{equation}\label{bunbun}
( \vec \nu'(\vec \Omega')^{\vec p} +  \vec \mu'_{r_1}(\vec \Omega')^{\vec p} ) 
\frac{\|\mu'_{j,r_1}-\nu'_j\|_{\TV}}{\|\mu'_{j,r_1}\|_{\TV} + \|\nu'_{j}\|_{\TV}}  \approx 0
\end{equation}
for any $j \in [d]$.  This will in turn follow from the estimates
\begin{equation}\label{mu-r}
\|\mu'_{j,r_1}\|_{\TV} \lesssim A^{O(1)} \energy_t[{\mathcal E}_j f_j](x) 
\end{equation}
and
\begin{equation}\label{mu-rd}
\| \mu'_{j,r_1} - \nu'_j \|_{\TV} \lesssim A^{O(1)} \delta \energy_t[{\mathcal E}_j f_j](x), 
\end{equation}
since this implies
$$ (\|\nu'_j\|_{\TV}^{p} + \|\mu'_{j,r_1}\|_{\TV}^{p} ) \frac{\|\mu'_{j,r_1}-\nu'_j\|_{\TV}}{\|\mu'_{j,r_1}\|_{\TV} + \|\nu'_j\|_{\TV} } 
\lesssim A^{O(1)} \delta \energy_t[{\mathcal E}_j f_j](x)^{p}$$
(possibly after adjusting the value of $c$).

To verify \eqref{mu-r}, we apply \eqref{job}, \eqref{dmuj} to write
$$ \|\mu'_{j,r_1}\|_{\TV} = r^{-0.9(d-1)} \int_{\R^{d-1}} \int_{\R^{d-1}} |F_j( z' )|^2 \varphi_{y',r_1^{0.9}}(z')^2 \gamma_t( (x'-y') B'_j )\ dy' dz'.$$
The claim \eqref{mu-r} then follows from \eqref{energy-relation}, \eqref{chang}, and the Fubini-Tonelli theorem. 

Now we show \eqref{mu-rd}.  From \eqref{numu}, the left-hand side is bounded by
$$
\int_{\R^{d-1}} \int_{\R^{d-1}} \left( \int_\R |H(y',x_d,s,\xi_j)| \gamma^{(1)}(s)\ ds\right)  \gamma_t((x'-y') B'_j)\ dy' d\xi_j.
$$
It will then suffice to show that
$$\int_{\R^{d-1}} \int_{\R^{d-1}} |H(y',x_d,s,\xi_j)| \gamma_t((x'-y') B'_j)\ dy' d\xi_j \lesssim \langle s \rangle^{O(1)} A^{O(1)} \delta \energy_t[{\mathcal E}_j f_j](x)$$
for any $s \in \R$.  

Fix $s$.  It will suffice to show the bound
\begin{equation}\label{song}
\int_{\R^{d-1}} |H(y',x_d,s,\xi_j) | d\xi_j \lesssim r_1^{-0.9(d-1)} \langle s \rangle^{O(1)} A^{O(1)} \delta \int_{\R^{d-1}} \rho_{y',r_1^{0.9}}(z') |F_j(z')|^2\ dz'
\end{equation}
for each $y' \in \R^{d-1}$, since the claim then follows from \eqref{energy-relation} and an estimate nearly identical to \eqref{chang}.  

Fix $y'$.  
From \eqref{ext-def}, \eqref{njxi} and the Fourier inversion formula we have
$$ {\mathcal E}_j f_j( (w',x_d) + r_1^{1.2} sn_j(\xi_j) ) = \int_{\R^{d-1}} \hat F_j(\xi) e^{2\pi i ( r_1^{1.2} s h_j(\xi) + (w' - r_1^{1.2} s \nabla_\xi h_j(\xi_j)) \xi^T )}\ d\xi$$
for any $w' \in \R^{d-1}$, and hence by \eqref{gr-def}
\begin{align*}
& G_{j,r_1}((y',x_d) + r_1^{1.2} s n_j(\xi_j),\xi_j)\\
&\quad = r_1^{-0.9(d-1)} \left| \int_{\R^{d-1}} \int_{\R^{d-1}} \hat F_j(\xi) e^{2\pi i ( r_1^{1.2} s h_j(\xi) + (w' - r_1^{1.2} s \nabla_\xi h_j(\xi_j)) \xi^T - w' \xi_j^T)} \varphi_{y',r_1^{0.9}}(w')\ d\xi dw' \right|^2.
\end{align*}
We make the change of variables $\xi = \xi_j + \zeta$, and insert the phase $e^{2\pi i r_1^{1.2} s(\nabla_\xi h_j(\xi_j) \xi_j^T - h(\xi_j))}$ outside the integral (which is harmless due to the absolute values) to write this as
$$ r_1^{-0.9(d-1)} \left| \int_{\R^{d-1}} \int_{\R^{d-1}} \hat F_j(\xi_j+\zeta) e^{2\pi i ( r_1^{1.2} s h_{j,\xi_j}(\zeta) + w' \zeta^T)} \varphi_{y',r_1^{0.9}}(w')\ d\zeta dw' \right|^2$$
where
\begin{equation}\label{h-def}
 h_{j,\xi_j}(\zeta) \coloneqq h_j(\xi_j+\zeta) - h_j(\xi_j) - \nabla_\xi h_j(\xi_j) \zeta^T
\end{equation}
is the remainder in the first order Taylor expansion of $h_j$ around $\xi_j$.  Performing the $w'$ integral (and using the fact that $\varphi$ is real and even), this can be rewritten as
$$ r_1^{0.9(d-1)} \left| \int_{\R^{d-1}} \hat F_j(\xi_j+\zeta) e^{2\pi i ( r_1^{1.2} s h_{j,\xi_j}(\zeta) + y' \zeta^T)} \hat \varphi( r^{0.9} \zeta )\ d\zeta \right|^2.$$
Applying this also with $s$ replaced by $0$, and using the inequality 
\begin{equation}\label{zw}
|z|^2 - |w|^2 \lesssim |w| |z-w| + |z-w|^2 \lesssim r^{-c} |z|^2 + r^c |z-w|^2
\end{equation}
for any $c>0$, we thus may bound the left-hand side of \eqref{song} by
$$ \lesssim \int_{\R^{d-1}} (r^{-c} G_{j,r_1}((y',x_d),\xi_j) + r^c X_{j,r_1,s}((y',x_d),\xi_j))\ d\xi$$
for any $c>0$, where
$$ X_{j,r_1,s}((y',x_d),\xi_j) \coloneqq r_1^{0.9(d-1)} \left| \int_{\R^{d-1}} \hat F_j(\xi_j+\zeta) a_{s,j,\xi_j}(\zeta) e^{2\pi i y' \zeta^T} \hat \varphi( r^{0.9} \zeta )\ d\zeta \right|^2$$
and
\begin{equation}\label{adef}
 a_{s,j,\xi_j}(\zeta) \coloneqq e^{2\pi i r_1^{1.2} s h_{j,\xi_j}(\zeta)}-1.
\end{equation}
The point will be that the support of $\varphi$ restricts $\zeta$ to be of size $O(r^{-0.9})$, which by the vanishing of $h_{j,\xi_j}$ to second order at the origin yields that $a_{s,j,\xi_j}$ is only of size $O( r_1^{1.2} s (r^{-0.9})^2 ) = O( r^{-0.6} |s| )$, with the $r^{-0.6}$ factor being the ultimate source of the $\delta$ gain in the estimates.

From \eqref{job} we have
$$ \int_{\R^{d-1}} G_{j,r_1}((y',x_d),\xi_j)\ d\xi_j \lesssim
r_1^{-0.9(d-1)}  \int_{\R^{d-1}} \rho_{y',r_1^{0.9}}(z') |F_j(z')|^2\ dz'$$
so it will suffice (after adjusting $c$ as necessary) to show that
\begin{equation}\label{troy}
 \int_{\R^{d-1}} X_{j,r_1,s}((y',x_d),\xi_j)\ d\xi_j \lesssim
r_1^{-0.9(d-1)} \langle s \rangle^{O(1)} A^{O(1)} \delta \int_{\R^{d-1}} \rho_{y',r_1^{0.9}}(z') |F_j(z')|^2\ dz'.
\end{equation}

Since $\hat F$ is supported in $B^{d-1}(0,A)$, and $\varphi$ is supported in $B(0,1)$, we may restrict $\xi_j$ to $B(0,2A)$.  We can then expand the left-hand side of \eqref{troy} as
$$r_1^{0.9(d-1)} \int_{\R^{d-1}} \eta'_{2A}(\xi_j) \left| \int_{\R^{d-1}} \int_{\R^{d-1}} F_j(z') a_{s,j,\xi_j}(\zeta) e^{2\pi i (-z' \xi_j^T + (y'-z') \zeta^T)} \hat \varphi( r^{0.9} \zeta )\ d z' d\zeta \right|^2$$
where $\eta'_{2A}$ is defined in Section \ref{notation-sec}.  This can be rearranged as
\begin{equation}\label{panic}
 r_1^{-0.9(d-1)} \int_{\R^{d-1}} \int_{\R^{d-1}} F_j(z'_1) \overline{F_j(z'_2)} K_{j,y',s}( z'_1, z'_2 )\ dy'_1 dy'_2 
\end{equation}
where the kernel $K_{j,y',s}$ is given by
\begin{align*}
& K_{j,y',s}(z'_1, z'_2) \coloneqq r_1^{1.8(d-1)} \int_{\R^{d-1}} \int_{\R^{d-1}} \int_{\R^{d-1}} \\
&\quad \eta'_{2A}( \xi_j  ) a_{s,j,\xi_j}(\zeta_1) \overline{a_{s,j,\xi_j}(\zeta_2)}  \hat \varphi(r_1^{0.9} \zeta_1) \hat \varphi(r_1^{0.9} \zeta_2) e^{2\pi i (-(z'_1-z'_2) \xi_j^T + (y'-z'_1) \zeta_1^T - (y'-z'_2) \zeta_2^T} \ d\zeta_1 d\zeta_2 d\xi_j.
\end{align*}
Rescaling $\zeta_1,\zeta_2$ by $r_1^{0.9}$, we can rewrite this as
$$ K_{j,y',s}(z'_1, z'_2) = \int_{\R^{d-1}} \int_{\R^{d-1}} \int_{\R^{d-1}} \eta'_{2A}( \xi_j ) a_{s,j,\xi_j}(r_1^{-0.9} \zeta_1) \overline{a_{s,j,\xi_j}(r_1^{-0.9} \zeta_2)} \hat \varphi(\zeta_1) \hat \varphi(\zeta_2) e^{2\pi i (-(z'_1-z'_2) \xi_j^T + (y'-z'_1) \zeta_1^T / r_1^{0.9} - (y'-z'_2) \zeta_2^T / r_1^{0.9})} \ d\zeta_1 d\zeta_2 d\xi_j.$$
Note that the support of $\varphi$ allows us to restrict $\zeta_1, \zeta_2$ to $B^{d-1}(0,1)$.  From \eqref{h-def} and two applications of the fundamental theorem of calculus, we may write
$$ h_{j,\xi_j}(\zeta) = \int_0^1 \int_0^1 ((\zeta \nabla^T_\xi) (\zeta \nabla^T_\xi) h_j)(\xi_j + uv \zeta)\ du dv$$
(where $\zeta \nabla^T_\xi$ is the directional derivative in the $\zeta$ direction), which when combined with \eqref{regular}, \eqref{adef}, and the chain rule yields the bounds
$$ a_{s,j,\xi_j}( r_1^{-0.9} \zeta_1) \overline{a_{s,j,\xi_j}(r_1^{-0.9} \zeta_2)} \lesssim A^{O(1)} r_1^{1.2} |s| r_1^{-1.8} \lesssim r_1^{-0.6} A^{O(1)} \langle s \rangle.$$
Crucially, the exponent of $r_1$ here is negative.  More generally one has the derivative bounds
$$ \nabla_{\zeta_1}^{\otimes m_1} \otimes \nabla_{\zeta_2}^{\otimes m_2} \otimes \nabla^{\otimes m}_{\xi_j} \otimes (a_{s,j,\xi_j}( r_1^{-0.9} \zeta_1) \overline{a_{s,j,\xi_j}(r_1^{-0.9} \zeta_2)}) \lesssim_{m_1,m_2,m} r_1^{-0.6} (A \langle s \rangle)^{O_{m_1,m_2,m}(1)}$$
for any $m_1,m_2,m$, provided that the regularity $M$ in \eqref{regular} is sufficiently large depending on $m_1,m_2,m$.  By integration by parts in the $\zeta_1, \zeta_2, \xi_j$ variables, we thus conclude the kernel bounds
$$ K_{j,y',s}(z'_1, z'_2)  \ll r_1^{-0.6} A^{O(1)} \langle s \rangle^{O(1)} \rho(z'_1-z'_2) \rho_{y',r_1^{0.9}}(z'_1) \rho_{y',r_1^{0.9}}(z'_2)$$
and hence by\eqref{panic} and Young's inequality (or Schur's test) we see that
$$ \int_{\R^{d-1}} X_{j,r_1,s}((y',x_d),\xi_j)\ d\xi_j \lesssim r_1^{-0.6} A^{O(1)} \langle s \rangle^{O(1)} r_1^{-0.9(d-1)} \int_{\R^{d-1}} |F(z')|^2 \rho_{y',r_1^{0.9}}(z')$$
which gives \eqref{troy} as required.

Finally we show \eqref{claim-2}.  By transitivity and symmetry of the $\approx$ relation, it will suffice to show that for any $r_1 \leq r \leq r_2$, we have
$$ \int_{(\vec \Omega')^{\vec p}} \det M\ d(\vec \mu'_{r})^{\vec p} \approx Z$$
for some quantity $Z$ independent of $r$.

The measure $\mu'_{j,r}$ can be expanded using \eqref{dmuj}, \eqref{gr-def} as
$$ d\mu'_{j,r} = 
r^{-0.9(d-1)} \left| \int_{\R^{d-1}} \gamma_t((x'-y') B'_j) F_j( z' ) e^{-2\pi i \xi_j (z')^T} \varphi_{y',r^{0.9}}(z')\ dz' \right|^2\ dy' d\xi_j$$
It will be convenient to compare this measure to
$$ d\tilde \mu'_{j,r} = 
r^{-0.9(d-1)} \left| \int_{\R^{d-1}} \gamma_t((x'-z') B'_j) F_j( z' ) e^{-2\pi i \xi_j (z')^T} \varphi_{y',r^{0.9}}(z')\ dz' \right|^2\ dy' d\xi_j.$$
Indeed, if we set $\tilde \mu'_r \coloneqq (\tilde \mu'_{1,r},\dots,\tilde \mu'_{d,r})$, we claim that
\begin{equation}\label{coco}
 \int_{(\vec \Omega')^{\vec p}} \det M\ d(\vec \mu'_{r})^{\vec p} \approx \int_{(\vec \Omega')^{\vec p}} \det M\ d(\vec \tilde \mu'_{r})^{\vec p}.
\end{equation}
By Lemma \ref{hold} and the triangle inequality, it suffices to show that
$$
( \vec \mu'_{r}(\vec \Omega')^{\vec p} +  \vec \tilde \mu'_{r}(\vec \Omega')^{\vec p} ) 
\frac{\|\mu'_{j,r}-\tilde \mu'_{j,r}\|_{\TV}}{\|\mu'_{j,r}\|_{\TV} + \|\tilde \mu'_{j,r}\|_{\TV}} \approx 0
$$
for any $j \in [d]$.  This will in turn follow from \eqref{mu-r} and the estimate
\begin{equation}\label{mu-rd-a}
\| \mu'_{j,r} - \tilde \mu'_{j,r} \|_{\TV} \lesssim A^{O(1)} t^{-c} \energy_t[{\mathcal E}_j f_j](x). 
\end{equation}
By \eqref{zw} (and \eqref{mu-r}), it suffices to show that
$$
\int_{\R^{d-1}} \int_{\R^{d-1}} r^{-0.9(d-1)} \left| \int_{\R^{d-1}} g_{j,t,x'}(z',y') F_j( z' ) e^{-2\pi i \xi_j (z')^T} \varphi_{y',r^{0.9}}(z')\ dz' \right|^2\ dy' d\xi_j
\lesssim A^{O(1)} \energy_t[{\mathcal E}_j f_j](x)$$
where
$$ g_{j,t,x'}(z',y') \coloneqq \gamma_t((x'-z') B'_j)-\gamma_t((x'-y') B'_j).$$
By Plancherel's theorem, the left-hand side may be expressed as
\begin{equation}\label{stif}
\int_{\R^{d-1}} \int_{\R^{d-1}} r^{-0.9(d-1)} |g_{j,t,x'}(z',y') F_j( z' ) \varphi_{y',r^{0.9}}(z')|^2\ dz' dy'.
\end{equation}
One can use the mean value theorem to bound
$$ g_{j,t,x'}(z',y') \lesssim A^{O(1)} \frac{r_1^{0.9}}{t} \left\langle \frac{z'-y'}{r_1^{0.9}} \right\rangle^{O(1)} \rho_{x',t}(y');$$
since $\frac{r_1^{0.9}}{t} \leq r_1^{-0.1}$, we can thus bound \eqref{stif} (after evaluating the $y'$ integral) by
$$
A^{O(1)} r_1^{-0.2} \int_{\R^{d-1}} |F_j( z' )|^2 \rho_{x',t}(z')\ dz',$$
giving the claim \eqref{coco}.

To finish the proof of \eqref{claim-2}, it suffices to show that
\begin{equation}\label{lefty}
 \int_{(\vec \Omega')^{\vec p}} \det M\ d(\vec{\tilde \mu}'_{r})^{\vec p} \approx Z
\end{equation}
for some quantity $Z$ independent of $r$.  Let $\Omega_{*,j} \coloneqq \R^{d-1}$ be parameterised by $\xi_j$, and let $\pi_j\colon \Omega'_j \to \Omega_{*,j}$ be the projection map $\pi_j \colon (y',\xi_j) \mapsto \xi_j$.  This map pushes forward $\tilde \mu'_{j,r}$ to the measure $\mu_{*,j,r}$ on $\Omega_{*,j}$ defined by
$$ d\mu_{*,j,r} \coloneqq
r^{-0.9(d-1)} \left(\int_{\R^{d-1}} \left| \int_{\R^{d-1}} g(z') F_j( z' ) e^{-2\pi i \xi_j (z')^T} \varphi_{y',r^{0.9}}(z')\ dz' \right|^2\ dy'\right) d\xi_j.$$
Let $(\vec \Omega_*,\vec \mu_*)$ be the $d$-tuple of measure spaces $(\Omega_{*,j}, \mu_{*,j})$.  The virtual function $M$ on $(\vec \Omega')^{\vec p}$ can be interpreted as the pullback of a virtual function on $\vec \Omega_*^{\vec p}$ which by abuse of notation we shall also call $M$.  By \eqref{change-multi}, we can thus write the left-hand side of \eqref{lefty} as
$$ \int_{\vec \Omega_*^{\vec p}} \det M\ d\vec \mu_{*,r}^{\vec p}.$$
The Fourier transform of
$$ y' \mapsto \int_{\R^{d-1}} g(z') F_j( z' ) e^{-2\pi i \xi_j (z')^T} \varphi_{y',r^{0.9}}(z')\ dz'$$
is equal to
$$ \xi \mapsto r^{0.9 (d-1)} \widehat{g F_j}(\xi_j + \xi) \hat \varphi( r^{0.9} \xi)$$
and hence by Plancherel's theorem, we may write
$$ d\mu_{*,j,r} =
r^{0.9(d-1)} \left(\int_{\R^{d-1}} |\widehat{g F_j}(\xi_j+\xi)|^2 |\hat \varphi(r^{0.9} \xi)|^2\ d\xi\right) d\xi_j$$
which on applying the rescaling $\xi = r^{-0.9} \zeta$ and using the support of $\hat \varphi$ becomes
$$ d\mu_{*,j,r} =
\left(\int_{B^{d-1}(0,1)} |\widehat{g F_j}(\xi_j+r^{-0.9} \zeta)|^2 |\hat \varphi(\zeta)|^2\ d\zeta\right) d\xi_j.$$
Note now that the $r$ parameter only affects the shift of $\xi_j$ in the argument of $g F_j$.  To exploit this, we rewrite the virtual integral once again, introducing the $d$-tuple $\vec \Omega_{**}$ of spaces $\Omega_{**,j} \coloneqq B^{d-1}(0,1) \times B^{d-1}(\xi_j^0, 2\eps)$ (with each $\Omega_{**,j}$ parameterised by $\zeta, \xi_j$) and $\vec \mu_{**,r}$ is the $d$-tuple of measures $\mu_{**,j,r}$ on $\Omega_{**,j}$ defined by
$$ d\mu_{**,j,r} \coloneqq |\widehat{g F_j}(\xi_j+r^{-0.9} \zeta)|^2 |\hat \varphi(\zeta)|^2\ d\zeta d\xi_j.$$
Observe that $\mu_{*,j,r}$ is the pushforward of $\mu_{**,j,r}$ by the map $(\zeta,\xi_j) \mapsto \xi_j$, and the virtual function $M$ on $\vec \Omega_*^{\vec p}$ pulls back to a virtual function on $\vec \Omega_{**}^{\vec p}$, which by abuse of notation we will continue to call $M$.  Then by \eqref{change-multi} one has
$$ \int_{\vec \Omega_*^{\vec p}} \det M\ d\vec \mu_{*,r}^{\vec p} = \int_{\vec \Omega_{**}^{\vec p}} \det M\ d\vec \mu_{**,r}^{\vec p}.$$
On the right-hand side, $M$ is now interpreted as the virtual function
$$ M \coloneqq \Sigma_{\vec p}(BB^T) = \Sigma_{\vec p}( (j, (\zeta,\xi_j)) \mapsto B_j(\xi_j) B_j(\xi_j)^T ).$$
We will compare $M$ with the variant
$$ M' \coloneqq \Sigma_{\vec p}( (j, (\zeta,\xi_j)) \mapsto B_j(\xi_j+r^{-0.9} \zeta) B_j(\xi_j+r^{-0.9} \zeta)^T ).$$
From the construction of $B_j$ (and the fact that $B'_j = O(A^{O(1)})$) we have
$$ B_j(\xi_j+r^{-0.9} \zeta) B_j(\xi_j+r^{-0.9} \zeta)^T  = B_j(\xi_j) B_j(\xi_j)^T + O( A^{O(1)} r^{-0.9} )$$
and hence by Lemma \ref{holder} 
$$ \int_{\vec \Omega_{**}^{\vec p}} \det M\ d\vec \mu_{**,r}^{\vec p} = \int_{\vec \Omega_{**}^{\vec p}} \det M'\ d\vec \mu_{**,r}^{\vec p}
+ O( A^{O(1)} r^{-0.9} \vec \mu_{**,r}(\vec \Omega_{**})^{\vec p} ).$$
From the Fubini-Tonelli theorem, Plancherel's theorem and \eqref{energy-relation}, we have
\begin{align*}
\mu_{**,j,r}(\Omega_{**,j}) &= \int_{\R^{d-1}} \int_{B^{d-1}(0,1)} |\widehat{g F_j}(\xi_j+r^{-0.9} \zeta)|^2 |\hat \varphi(\zeta)|^2\ d\zeta d\xi_j \\
&= \int_{\R^{d-1}} |\widehat{g F_j}(\xi)|^2\ d\xi \\
&= \int_{\R^{d-1}} |g F_j(y')|^2\ dy' \\
&\lesssim \energy_{t}[{\mathcal E}_j f_j](x).
\end{align*}
We conclude that
$$ \int_{\vec \Omega_{**}^{\vec p}} \det M\ d\vec \mu_{**,r}^{\vec p} \approx \int_{\vec \Omega_{**}^{\vec p}} \det M'\ d\vec \mu_{**,r}^{\vec p}$$
and so it now suffices to show that
$$ \int_{\vec \Omega_{**}^{\vec p}} \det M'\ d\vec \mu_{**,r}^{\vec p} \approx Z$$
for some $Z$ independent of $r$.

We now make the change of variables $(\zeta,\xi'_j) \coloneqq (\zeta,\xi_j + r^{-0.9} \zeta)$.  By \eqref{change-multi}, this lets us write
\begin{equation}\label{virt}
 \int_{\vec \Omega_{**}^{\vec p}} \det M'\ d\vec \mu_{**,r}^{\vec p} = \int_{\vec \Omega_{***}^{\vec p}} \det M\ d\vec \mu_{***}^{\vec p}
\end{equation}
where $\vec \Omega_{***}$ is the $d$-tuple of spaces $\Omega_{***,j} \coloneqq B^{d-1}(0,1) \times B^{d-1}(\xi_j^0, 2\eps)$ (with each $\Omega_{***,j}$ parameterised by $\zeta, \xi'_j$), $\vec \mu_{***,r}$ is the $d$-tuple of measures $\mu_{***,j}$ on $\Omega_{***,j}$ defined by
$$ d\mu_{***,j} \coloneqq |\widehat{g F_j}(\xi_j)|^2 |\hat \varphi(\zeta)|^2\ d\zeta d\xi_j$$
(which in particular is just supported in the subset $B^{d-1}(0,1) \times B^{d-1}(\xi_j^0, \eps)$ of $\Omega_{***,j}$) and $M$ is now interpreted as the virtual function
$$ M \coloneqq \Sigma_{\vec p}( (j, (\zeta,\xi'_j)) \mapsto B_j(\xi'_j) B_j(\xi'_j)^T ).$$
But the right-hand side of \eqref{virt} is now independent of $r$, concluding the proof of \eqref{claim-2}.
\end{proof}

We now show the key monotonicity formula:

\begin{lemma}[Monotonicity]  For $r_1 \leq t \leq r_2$, one has
$$ t \partial_t Q(t) \geq - O( A^{O(1)} \delta C(R,t) ).$$
\end{lemma}

\begin{proof}
We adapt the arguments of the previous section. We can write
$$ \vec \mu[t,x] = \vec \mu_{w[t,x]}$$
where $\vec \mu = (\mu_1,\dots,\mu_d)$ is the reference measure
\begin{equation}\label{dmuj-def}
 d\mu_j \coloneqq  G_{j,r_1}(y,\xi_j) \ dy d\xi
\end{equation}
and $w\colon \Disjoint \vec \Omega \to \R$ is the weight
\begin{equation}\label{wit-def}
 w[t,x] \coloneqq r_1^{-1.2} \gamma_t(\phi[x]) \gamma^{(1)}_{r_1^{1.2}}(\varphi[x]) 
\end{equation}
with $\phi[x]\colon \Disjoint \vec \Omega \to \R^{d-1}$ the function
\begin{equation}\label{philin}
 \phi[x](j,(y,\xi_j)) \coloneqq (x-y) B_j(\xi_j)
\end{equation}
and $\varphi[x]\colon \Disjoint \vec \Omega \to \R$ the function
$$ \varphi[x](j,y,\xi) \coloneqq x_d-y_d.$$
If we define (as in the previous section)
\begin{equation}\label{expect-def}
 \Expect_t( G ) \coloneqq t^{-d} \int_{\R^d} \int_{\Omega^{\vec p}} G[t,x] \ d\vec \mu_{w[t,x]}^{\vec p} dx 
\end{equation}
then
$$ Q(t) = \Expect_t( \eta_R \det M ).$$
Since
\begin{equation}\label{doto}
 t\partial_t w[t,x] = \frac{2}{t^2} \phi[x] \phi[x]^T w[t,x]
\end{equation}
and $\eta_R \det M$ is independent of $t$, Proposition \ref{diff-i} formally implies that
\begin{equation}\label{trop}
 t \partial_t Q(t) = \Expect_t( - d \eta_R \det M + \frac{2}{t^2} \eta_R \det(M) \Sigma_{\vec p}(\phi \phi^T) ).
\end{equation}
There is however a slight technical issue in justifying this identity, namely that the function $\phi[x]$ appearing in \eqref{doto} is unbounded due to the unboundedness of the coordinate $y$ in $\Omega_j = \R^{d} \times B^{d-1}(\xi_j^0, 2\eps)$.  However, the measures $\vec \mu_{w[t,x]}$ are rapidly decreasing in the $y$ variable, and one can rigorously justify \eqref{trop} by a truncation argument which we briefly summarise here.  Firstly, by the fundamental theorem of calculus, we can rewrite \eqref{trop} in an equivalent integral form
\begin{equation}\label{trop-2}
 Q(t+h) - Q(t) = \int_t^{t+h} (t')^{-1} \Expect_t( - d \eta_R \det M + \frac{2}{(t')^2} \eta_R \det(M) \Sigma_{\vec p}(\phi \phi^T) )\ dt'.
\end{equation}
Next, we weight each measure $\mu_j$ in \eqref{dmuj-def} by a cutoff $\chi_N$ for some extremely large scale $N$; by monotone convergence, these truncated measures will not vanish identically for $N$ large enough.  By repeating the arguments used to justify \eqref{tqt}, one can then establish a truncated version of \eqref{trop} in which the measures $\mu_j$ that appear implicitly on both sides of \eqref{trop-2} are weighted by $\chi_N$.  One can then send $N \to \infty$ using Lemma \ref{hold} and the rapid decrease of $G_{j,r_1}$ and $w[t,x]$ to remove the cutoff, and then \eqref{trop} follows from the fundamental theorem of calculus.  We leave the details to the interested reader.

In analogy with \eqref{tash}, we observe (formally, at least) that
\begin{equation}\label{tash-2}
 t^{-d} \int_{\R^d} \nabla_x \left(\int_{\Omega^{\vec p}} \eta_R \adj(M) \Phi[x]\ d\vec \mu_{w[t,x]}^{\vec p}\right)\ dx = 0
\end{equation}
where for each $x$, $\Phi = \Phi[x] \in {\mathcal L}^\infty(\Omega^{\vec p} \to \R^{1 \times d})$ is the virtual column vector-valued function
$$ \Phi \coloneqq \Sigma_{\vec p}( (\nabla_x^T \phi[x]) \phi[x]^T ).$$
We have
$$ \nabla_x^T w[t,x] = -\frac{2}{t^2} (\nabla_x^T \phi[x]) \phi[x]^T w[t,x] - 2 \pi r_1^{-2.4} \varphi[x] e_d^T$$
and thus by Proposition \ref{diff-ii} we can write \eqref{tash-2} as
$$
\Expect_t( \nabla_x (\eta_R \adj(M) \Phi) - \frac{2}{t^2} \eta_R \Phi^T \adj(M) \Phi - 2 \pi \eta_R r_1^{-2.4} \Sigma_{\vec p}(\varphi) e_d^T \adj(M) \Phi ) = 0.
$$
Again, we run into the technical issue that $\phi$ (and $\varphi$) are unbounded when justifying this identity, however this can be resolved by the same truncation argument presented previously; again, we leave the details to the interested reader.  We can therefore write $t \partial_t Q(t)$ as
$$ \Expect_t\left( \nabla_x (\eta_R \adj(M) \Phi) - d \eta_R \det M + \frac{2}{t^2} \eta_R S_0 - 2 \pi r_1^{-2.4} \eta_R \Sigma_{\vec p}(\varphi) e_d^T \adj(M) \Phi \right)$$
where
$$S_0 \coloneqq \det(M) \Sigma_{\vec p}(\phi \phi^T) - \Phi^T \adj(M) \Phi.$$
From the Leibniz rule as before (taking advantage of the fact that $M$ is now independent of $x$, and $\phi$ is now linear in $x$ thanks to \eqref{philin}) we have
$$\nabla_x (\eta_R \adj(M) \Phi) - d \eta_R \det M  = R^{-1} (\nabla_x \eta)_R \adj(M) \Phi.$$
We can thus write
$$ t \partial_t Q(t) = \Expect_t\left( \frac{2}{t^2} \eta_R S_0 + R^{-1} S_1 - 2 \pi r_1^{-2.4} S_2 \right)$$
where $S_1 = S_1[t,x]$, $S_2 = S_2[t,x]$ are the virtual functions
\begin{align*}
S_1 &\coloneqq (\nabla_x \eta)_R \adj(M) \Phi  \\
S_2 &\coloneqq \eta_R \Sigma_{\vec p}(\varphi) e_d^T \adj(M) \Phi.
\end{align*}

From Theorem \ref{non-neg}(ii) (with $B_j^0 = B_j( \xi_j^0)$), Lemma \ref{matrix-lem}, and the choice \eqref{construct-2} of $\eps$, we have
$$
 \int_{\vec \Omega^{\vec p}} S_0[t,x]\ d\vec \mu_{w[t,x]}^{\vec p} \geq 0 
$$
for each $x \in \R^d$, and thus
\begin{equation}\label{doke}
 \Expect_t\left( \frac{2}{t^2} \eta_R S_0 \right ) \geq 0.
\end{equation}
Now we control $S_1$, adapting the arguments used to prove Lemma \ref{lowerl}.  At each point $x \in \R^d$, the expression $S_1$ can be expanded as the sum of $O(1)$ virtual functions of the form $\Sigma_{\vec p}(F_1) \dots \Sigma_{\vec p}(F_n) \Sigma_{\vec p}(\phi^l)$, where $n=O(1)$, $F_1,\dots,F_n$ are bounded in magnitude by $A^{O(1)}$, and $\phi^l$ is one of the components of $\phi$.  By \eqref{expect-def}, Lemma \ref{holder}, we thus have
\begin{equation}\label{sodd}
 \Expect_t(S_1) 
\lesssim A^{O(1)} t^{-d} \int_{B^d(0,2R)} \vec \mu_{w[t,x]}(\vec \Omega)^{\vec p} \sum_{j \in [d]} \E_{(\mu_j)_{w_j[t,x]}} |\phi[x]|\ dx.
\end{equation}
From \eqref{wit-def} we have the pointwise estimate
$$ t^{-1} |\phi[x]| w_j[t,x] \lesssim w_j[t/2,x]^{\min(p,1)} w_j[t,x]^{1-\min(p,1)}$$
and hence by H\"older's inequality
$$ \E_{(\mu_j)_{w_j[t,x]}} |\phi[x]| \lesssim t (\| (\mu_j)_{w_j[t/2,x]} \|_{\TV} / \| (\mu_j)_{w_j[t,x]} \|_{\TV})^{\min(p,1)}$$
and thus (since $w_j[t,x] \leq w_j[t/2,x]$)
$$ \E_{(\mu_j)_{w_j[t,x]}} |\phi[x]| \lesssim t (\| (\mu_j)_{w_j[t/2,x]} \|_{\TV} / \| (\mu_j)_{w_j[t,x]} \|_{\TV})^{p}.$$
We conclude that
$$
 \Expect_t(S_1) 
\lesssim A^{O(1)} t^{1-d} \int_{\R^d} \eta(x/R) \vec \mu_{w[t/2,x]}(\vec \Omega)^{\vec p}\ dx.
$$
The expression $S_2$ behaves similarly to $S_1$, except there is an additional factor of $\Sigma_{\vec p}(\varphi)$, which causes \eqref{sodd} to be modified to
\begin{align*}
&\Expect_t(S_2) 
\lesssim A^{O(1)} t^{-d} \int_{B^d(0,2R)} \\
&\quad \vec \mu_{w[t,x]}(\vec \Omega)^{\vec p} \sum_{j,j' \in [d]} (\E_{(\mu_j)_{w_j[t,x]}} |\varphi[x]|^2)^{1/2} (\E_{(\mu_j)_{w_j[t,x]}} |\phi[x]|^2)^{1/2}\ dx.
\end{align*}
By a modification of the preceding arguments we have
$$ (\E_{(\mu_j)_{w_j[t,x]}} |\phi[x]|^2)^{1/2} \lesssim t (\| (\mu_j)_{\tilde w_j[t/2,x]} \|_{\TV} / \| (\mu_j)_{w_j[t,x]} \|_{\TV})^{p}$$
and
$$ (\E_{(\mu_j)_{w_j[t,x]}} |\varphi[x]|^2)^{1/2} \lesssim r_1^{1.2} (\| (\mu_j)_{\tilde w_j[t/2,x]} \|_{\TV} / \| (\mu_j)_{w_j[t,x]} \|_{\TV})^{p}$$
where $\tilde w_j$ is the modification of $w_j$ defined by
$$ \tilde w[t,x] \coloneqq r_1^{-1.2} \gamma_t(\phi[x]) \gamma^{(1)}_{2r_1^{1.2}}(\varphi[x]).$$
We thus have
$$
 \Expect_t(S_2) 
\lesssim A^{O(1)} t^{1-d} r_1^{1.2} \int_{\R^d} \eta(x/R) \vec \mu_{\tilde w[t/2,x]}(\vec \Omega)^{\vec p}\ dx.
$$
From this and \eqref{doke}, we conclude that
$$ t \partial_t Q(t) \geq -A^{O(1)} (tR^{-1} + tr_1^{-1.2}) t^{-d} \int_{B^d(0,2R)}
\vec \mu_{\tilde w[t/2,x]}(\vec \Omega)^{\vec p} \ dx.$$
The expression
$$ t^{-d} \int_{B^d(0,2R)}
\vec \mu_{\tilde w[t/2,x]}(\vec \Omega)^{\vec p}\ dx$$
is basically $Q(t/2)$ (but with the cutoffs $\eta$, $\gamma^{(1)}$ replaced by a slightly larger cutoffs, and with the $\det M$ weight missing). Adapting Lemma \ref{compar} to control this quantity, we find that
$$ t^{-d} \int_{B^d(0,2R)}
\vec \mu_{\tilde w[t/2,x]}(\vec \Omega)^{\vec p}\ dx \lesssim A^{O(1)} \left( t^{-d} \int_{B^d(0,2R)} \vec \mu'[t/2,x](\vec \Omega')^{\vec p}\ dx +  C(R,t/2) \right)$$
while from a routine modification of the proof of Claim (i) one has
$$ t^{-d} \int_{B^d(0,2R)} \vec \mu'[t/2,x](\vec \Omega')^{\vec p}\ dx  \lesssim A^{O(1)} C(R,t/2);$$
also, since $\energy_{t/2}(f) \sim  \energy_t(f)$ for any $f$, one has
$$ C(R,t/2) \sim C(R,t).$$
We thus conclude that
$$ t \partial_t Q(t) \geq -A^{O(1)} (tR^{-1} + tr_1^{-1.2}) C(R,t)$$
and the lemma follows.
\end{proof}

Integrating the above lemma using the fundamental theorem of calculus, one obtains
$$ Q(r_1) \leq Q(r_2) + O(A^{O(1)} \delta \sup_{r_1 \leq t \leq r_2} C(R,t) )$$
and hence by Lemma \ref{compar}
\begin{align*}
& r_1^{-d} \int_{\R^d} \eta_R(x) \int_{(\vec \Omega')^{\vec p}} \det M\ d\vec \mu'[r_1,x]^{\vec p}\ dx\\
&\quad \leq r_2^{-d} \int_{\R^d} \eta_R(x) \int_{(\vec \Omega')^{\vec p}} \det M\ d\vec \mu'[r_2,x]^{\vec p}\ dx
+ O( A^{O(1)} \delta \sup_{r_1 \leq t \leq r_2} C(R,t)).
\end{align*}
From \eqref{ineq} we conclude Claim (iii) as required.  The proof of Theorem \ref{lamrc} is now complete.

\section{Proof of multilinear oscillatory integral estimate}\label{osc-sec}

We now prove Theorem \ref{lmrc-osc}.  Our main tools will be Theorem \ref{curv-kak}, Theorem \ref{lmrc}, and an ``epsilon loss-free'' version of the wave packet decomposition arguments appearing in \cite[Proposition 6.9]{bct}.  As usual, we begin with some basic reductions.  We let implied constants in our asymptotic notation depend on $d$.  The case $q > 2$ follows easily from the $q=2$ case by H\"older's inequality; the cases $q<2$ similarly follow from the $q=2$ case after interpolation with the trivial bound
\begin{equation}\label{chunk}
\| \prod_{j \in [d]} S_\lambda^{(j)} f_j \|_{L^{\infty}(V)} 
\lesssim A^{O(1)} \prod_{j \in [d]} \|f_j \|_{L^1(U_j)}.
\end{equation}
Thus we may assume without loss of generality that $q=2$.  The trivial bound \eqref{chunk} also establishes the $p=\infty$ case of the theorem, so by interpolation we may assume that $p$ is bounded, for instance $p \leq 2$.

As in preceding sections, we set a small parameter
$$ \eps := C_0^{-1} A^{-C_0} \left(d-1-\frac{1}{p}\right)^{C_0} $$
for some large constant $C_0$ (depending only on $d$) to be chosen later.   It suffices to show that
$$
\| \prod_{j \in [d]} S_\lambda^{(j)} f_j \|_{L^{2p}(V_{1/A})} 
\lesssim A^{O(1)} \eps^{-O(1)} \lambda^{-d/2p} \prod_{j \in [d]} \|f_j \|_{L^2(U_{j,1/A})}$$
whenever $f_j \in L^2(U_{j,1/A})$.

By covering $U_{j,1/A}$ and $V_{1/A}$ by $O(\eps^{-O(1)})$ balls of radius $\eps$ and using the triangle inequality, it suffices to establish the estimate
$$
\| \prod_{j \in [d]} S_\lambda^{(j)} f_j \|_{L^{2p}(B^d(x_0,\eps))} 
\lesssim A^{O(1)} \eps^{-O(1)} \lambda^{-d/2p} \prod_{j \in [d]} \|f_j \|_{L^2(B^{d-1}(\xi_j^0,\eps))}$$
whenever $f_j \in L^2(B^{d-1}(\xi_j^0,\eps))$, $\xi_j^0 \in U_{j,1/A}$, and $x_0 \in V_{1/A}$.  We may assume that 
\begin{equation}\label{lambda-large}
\lambda \geq \eps^{-10}
\end{equation}
(say), since the claim follows from \eqref{chunk} otherwise.

As the operators $S_\lambda^{(j)} f_j$ depend linearly on the amplitude function $\psi_j(x,\xi)$, we may reduce to the case when $\psi_j(x_0,\xi_j^0)$ is bounded away from zero, and more specifically that
$$ |\psi_j(x_0,\xi_j^0)| \gtrsim A^{-O(1)}$$
for each $j \in [d]$.  By the regularity of $\psi_j$, we then have
\begin{equation}\label{jxi}
 |\psi_j(x,\xi)| \gtrsim A^{-O(1)}
\end{equation}
for $x \in B^d(x_0,3\eps)$ and $\xi \in B^{d-1}(\xi_j^0, 3\eps)$.

By the submersion hypothesis, each of the matrices $\nabla_x^T \nabla_\xi \Phi_j(x_0,\xi^j_0)$ is of full rank, and has a unit left null vector $n_j^0 \in \R^d$; the transversality hypothesis yields the lower bound
$$ \left| \bigwedge_{j \in [d]} n^j_0 \right| \geq 1/A.$$
By applying a rotation of $\R^d$, we may assume without loss of generality that the vertical components $n_j^0 e_d^T$ of each of the $n_j^0$ is bounded away from zero:
$$ |n_j^0 e_d^T| \gtrsim 1.$$
Indeed, a random rotation will achieve this goal with positive probability.  The left null space of $\nabla_x^T \nabla_\xi \Phi_j(x_0,\xi^j_0)$ now makes an angle of $\gtrsim 1$ with the horizontal space $\R^{d-1}$, and hence if we set $B_j^0 \in \R^{d-1 \times d-1}$ to be the matrix $\nabla_x^T \nabla_\xi \Phi_j(x_0,\xi^j_0)$ with the bottom row removed, then all singular values of $B_j^0$ are $A^{O(1)}$.

We normalise
\begin{equation}\label{norma}
\|f_j \|_{L^2(B^{d-1}(\xi_j^0,\eps))} = 1
\end{equation}
for all $j$.  We introduce the rescaled functions
\begin{equation}\label{fspl}
\begin{split}
F_{j}(x) &\coloneqq S_\lambda^{(j)} f_{j}(\frac{x}{\lambda_j}) \\
&= \int_{B^{d-1}(\xi_j,\eps)} e^{2\pi i \lambda \Phi_j(\frac{x}{\lambda},\xi)} \psi_j(\frac{x}{\lambda},\xi) f_j(\xi) d\xi
\end{split}
\end{equation}
so after rescaling $x$ by $\lambda$, our task is now to show that
\begin{equation}\label{bash}
\| \prod_{j \in [d]} F_{j} \|_{L^{2p}(B^d(\lambda x_0,\eps \lambda))} 
\lesssim A^{O(1)} \eps^{-O(1)}.
\end{equation}

We begin with a local estimate:

\begin{proposition}[Local estimate]\label{local} There is a constant $C_1$ depending only on $d$ such that, for any $x = (x',x_d) \in B^d(\lambda x_0, \eps \lambda)$, one has
\begin{equation}\label{bod}
 \| \prod_{j \in [d]} F_{j} \|_{L^{2p}(B^d(x, A^{-C_1} \sqrt{\lambda}))}
\lesssim A^{O(1)} \eps^{-O(1)} \prod_{j \in [d]} \| F_j(\cdot,x_d) \|_{L^2(B^{d-1}(x', \sqrt{\lambda}))} + A^{O(1)} \eps^{-O(1)} \lambda^{-10d^2}.
\end{equation}
\end{proposition}

\begin{proof}  Let $C_1$ be a sufficiently large constant depending on $d$ to be chosen later. Our main tool will be the multilinear restriction estimate in Theorem \ref{lmrc}, but first it will be convenient to normalise the phase functions $\Phi_j$ and the amplitude functions $\psi_j$ in various ways.

By translating the $x$ variable (and adjusting $A$ slightly if necessary) we may normalise $x = (x',x_d)=0$ (so that $x_0 = O(\eps)$).  By subtracting $\Phi_j(0,\xi)$ from $\Phi_j(x,\xi)$, and multiplying each $f_j(\xi)$ by $e^{i\lambda \Phi_j(0,\xi)}$ to compensate, we may also normalise $\Phi_j(0,\xi)=0$ for all $\xi$.  For $\xi \in B^{d-1}(\xi^0_j, \eps)$, the map $\xi \mapsto \nabla_{x'} \Phi_j(0,\xi)$ has derivative
$\nabla_\xi^T \nabla_{x'} \Phi_j(0,\xi) \in \R^{d-1 \times d-1}$ equal to $(B_j^0)^T + O( A^{O(1)} \eps )$, so in particular it has singular values $A^{O(1)}$; here of course $\nabla_{x'} = (\partial_{x_1},\dots,\partial_{x_{d-1}})$ is the initial segment of $\nabla_x = (\partial_{x_1},\dots,\partial_{x_d})$.  By the inverse function theorem and the chain rule, this map $\xi \mapsto \nabla_{x'}^T \Phi_j(0,\xi)$ is then a diffeomorphism on $B^{d-1}(\xi^0_j, A^C \eps)$ for any fixed constant $C$, with both this map and its inverse having all derivatives up to $m^{\mathrm{th}}$ order bounded by $A^{O_m(1)}$ if $M$ is sufficiently large depending on $m$.  Applying this change of variables in the $\xi$ variable (and adjusting the amplitude $\psi_j$ and the base frequencies $\xi_j^0$ accordingly, as well as adjusting $A$), we may then assume without loss of generality that 
$$ \nabla_{x'} \Phi_j(0,\xi) = \xi$$
for all $\xi \in B^{d-1}(\xi^0_j,3\eps)$.

From \eqref{jxi}, we may divide $\psi_j(x,\xi)$ by $\psi_j(0,\xi)$, and multiply $f_j(\xi)$ by $\psi_j(0,\xi)$ to compensate, to arrive at the additional normalisation
\begin{equation}\label{psi-norm}
 \psi_j(0,\xi) = 1
\end{equation}
for all $\xi \in B^{d-1}(\xi^0_j,3\eps)$, without affecting any of the previous normalisations (after adjusting $A$ accordingly).

By Taylor expansion (or two applications of the fundamental theorem of calculus), we then have
\begin{align*}
 \lambda \Phi_j(\frac{y}{\lambda},\xi) &= \Phi_j(0,\xi) + y \nabla_x^T \Phi_j(0,\xi) + \Psi_j(y,\xi) \\
&= y' \xi^T + y_d h_j(\xi) + \Psi_j( y, \xi )
\end{align*}
for any $y = (y',y_d) \in B^d(0, \sqrt{\lambda})$, where 
$$ h_j(\xi) \coloneqq \partial_{x_d} \Phi_j(0,\xi)$$
and
$$ \Psi_j(y,\xi) \coloneqq \lambda^{-1} \int_0^1 \int_0^1  ((y \nabla_x^T) (y \nabla_x^T) \Phi_j)( \frac{uvy}{\lambda}, \xi)\ du dv $$
and $y\nabla_x^T$ is the directional derivative in the $y$ direction.  Therefore we can write
\begin{equation}\label{fjt}
 F_{j}(y',y_d) = \int_{\R^{d-1}} e^{2\pi i (y' \xi^T + y_d h_j(\xi))} \tilde \psi_j(y,\xi) f_j(\xi) d\xi
\end{equation}
where the modified amplitude function $\tilde \psi_j$ is defined by
$$ \tilde \psi_j(y,\xi) \coloneqq e^{2\pi i \Psi_j(y,\xi)} \psi_j(\frac{y}{\lambda},\xi)\ d\xi.$$
From the derivative bounds on $\Phi_j$, we have
\begin{equation}\label{hand}
 \nabla_\xi^{\otimes m} h_j(\xi) \lesssim_m A^{O_m(1)}
\end{equation}
for any $m$ and any $\xi \in B^d(\xi_j^0,3\eps)$, if $M$ is sufficiently large depending on $m$.  By construction, the vector $n_j(\xi) \coloneqq (-\nabla h_j(\xi), 1)$ is in the left null space of $\nabla_x^T \nabla_\xi \Phi_j(0, \xi)$, and so we also have the transversality property
$$ \left| \bigwedge_{j \in [d]} n_j(\xi_j) \right| \gtrsim A^{-O(1)}$$
whenever $\xi_j \in B^d(\xi_j^0, 3\eps)$.

For $y \in B^d(0,\sqrt{\lambda})$, the derivative bounds on $\Phi_j$ imply that
$$ \Psi_j(y,\xi) \lesssim A^{O(1)} \lambda^{-1} |y|^2 \lesssim A^{O(1)} $$
and more generally
$$ \nabla_y^{\otimes m} \otimes \nabla_\xi^{\otimes m'} \Psi_j(y,\xi) \lesssim_{m,m'} \lambda^{-m/2} A^{O_{m,m'}(1)}$$
for any $m,m'$, if $M$ is sufficiently large depending on $m,m'$.  By the chain rule and product rule, and the derivative bounds on $\psi_j$, we conclude that
\begin{equation}\label{inter}
 \nabla_y^{\otimes m} \otimes \nabla_\xi^{\otimes m'} \tilde \psi_j(y,\xi) \lesssim_{m,m'} \lambda^{-m/2} A^{O_{m,m'}(1)}
\end{equation}
for any $m,m'$, if $M$ is sufficiently large depending on $m,m'$.  
Thus the amplitude function $\tilde \psi_j$ does not vary too wildly for $y \in B^d(0,\sqrt{\lambda})$ and $\xi \in B^{d-1}(\xi_j^0, \eps)$, and the representation \eqref{fjt} is approximately representing $F_{j}$ as an extension operator.  Note from \eqref{psi-norm} that we have the normalisation
\begin{equation}\label{tpsi-norm}
\tilde \psi_j(0,\xi) = 1
\end{equation}
for $\xi\in B^{d-1}(\xi^0_j, 2\eps)$.

It is tempting to now apply Theorem \ref{lmrc} (after using some sort of decomposition of the amplitude $\tilde \psi_j(y,\xi)$, but this will not achieve the required localisation of the $L^2$ norms of $F_j$ in the right-hand side of \eqref{local}.  To achieve this, we need an appropriate decomposition of $f_j$, at a scale that we will select by a pigeonholing argument.  For every scale $r$ between (say) $A^{-C_1/2} \sqrt{\lambda}$ and $2 A^{-C_1/2} \sqrt{\lambda}$, define the cutoff $\chi_r\colon \R^{d-1} \to \R$ by the convolution
\begin{equation}\label{chir}
 \chi_r(y') \coloneqq \int_{\R^{d-1}} 1_{B^{d-1}(0, r)}( y' + z' / \eps ) \varphi(z')^2\ dz',
\end{equation}
where $\varphi$ was defined in Section \ref{notation-sec}; this non-negative function takes values between $0$ and $1$, is monotone non-decreasing in $r$, and is a smoothed out version of $1_{B^{d-1}(0, r)}$ that has Fourier transform supported in $B^{d-1}(0,\eps)$.  For a given $j \in [d]$, we then consider the smoothed out local energies
\begin{equation}\label{ej-def}
 E_j(r) \coloneqq \int_{\R^{d-1}} |\check f_j(y')|^2 \chi_r(y')^2\ dy'.
\end{equation}
The $E_j(r)$ are monotone non-decreasing in $r$, and by \eqref{norma} and Plancherel's theorem, they lie in $[0,1]$.  We can therefore find a scale $r_j \in [A^{-C_1/2}\sqrt{\lambda}, 2A^{-C_1/2}\sqrt{\lambda}]$ with the property that
\begin{equation}\label{raj}
 E_j(r_j + \lambda^{0.49}) \leq 2E_j(r_j) + \lambda^{-20d^2},
\end{equation}
since if this inequality fails for all $r_j$ in this range, an iteration involving $O(\log \lambda)$ applications of the failure of this inequality starting at $r_j = A^{-C_1/2}\sqrt{\lambda}$ and advancing in steps of $\lambda^{0.49}$ will force $E_j$ to exceed $1$ at some point, giving a contradiction.

Henceforth we fix $r_j$ for which \eqref{raj} holds; the usefulness of this bound will become apparent at the end of the argument.  We then decompose $f_j = f'_j + f''_j$, where
\begin{equation}\label{fjr}
 f'_j(\xi) \coloneqq f_j * \hat \chi_{r_j}(\xi) = \int_{B^{d-1}(0,\eps)} f_j(\xi-\eta) \hat \chi_{r_j}(\eta)\ d\eta
\end{equation}
and $f''_j \coloneqq f_j - f'_j$.  Both $f'_j$ and $f''_j$ are supported on $B^{d-1}(\xi_j^0,2\eps)$.  Then we have $F_j = F'_j + F''_j$, where
$$
F'_{j}(y',y_d) \coloneqq \int_{\R^{d-1}} e^{2\pi i (y' \xi^T + y_d h_j(\xi))} \tilde \psi_j(y,\xi) f'_j(\xi) d\xi
$$ 
and
$$
F''_{j}(y',y_d) \coloneqq \int_{\R^{d-1}} e^{2\pi i (y' \xi^T + y_d h_j(\xi))} \tilde \psi_j(y,\xi) f''_j(\xi) d\xi.
$$ 
We claim the pointwise bound
\begin{equation}\label{hong}
 F''_j(y',y_d) \lesssim A^{O(C_1)} \eps^{-O(1)} \lambda^{-20d^2}
\end{equation}
for all $(y',y_d) \in B^d(0, A^{-C_1} \sqrt{\lambda})$.  Indeed, we may expand the left-hand side as
$$
\int_{\R^{d-1}} \int_{\R^{d-1}} e^{2\pi i ((y'-z') \xi^T + y_d h_j(\xi))} \tilde \psi_j(y,\xi) \eta'_{2\eps}(\xi) d\xi \hat f_j(z') (1-\chi_{r_j}(z'))\ dz'$$
where the cutoffs $\eta'_{2\eps}$ were defined in Section \ref{notation-sec}.
In the region $|z'| \leq r_j/2$, one sees from \eqref{chir} and the properties of $\varphi$ that $1-\chi_{r_j}(z') \lesssim \lambda^{-40d^2}$, so the contribution of this region is acceptable.  In the region $|z'| > r_j/2$, we have $|y'-z'| \gtrsim |z'| \gtrsim A^{-C_1/2} \sqrt{\lambda}$, so by \eqref{hand} we have the lower bound
$$ |\nabla_\xi ( (y'-z') \xi^T + y_d h_j(\xi) )| \gtrsim |z'|$$
for $(y',y_d) \in B^d(0,A^{-C_1} \sqrt{\lambda})$. Applying repeated integration by parts in $\xi$ using \eqref{inter}, \eqref{hand}, we then see that 
$$\int_{\R^{d-1}} e^{2\pi i ((y'-z') \xi^T + y_d h_j(\xi))} \tilde \psi_j(y,\xi) \eta'_{2\eps}(\xi) d\xi \lesssim A^{O(C_1)} \eps^{-O(1)} \langle z' \rangle^{-60 d^2}$$
and from this and \eqref{norma} we see that this contribution is also acceptable.

From the triangle inequality (and crude bounds on $F_j$ using \eqref{norma}) one now has
$$\prod_{j \in [d]} F_{j}  = \prod_{j \in [d]} F'_{j} + O( A^{O(C_1)} \eps^{-O(1)} \lambda^{-150d^2} )$$
(say) on $B^d(0,A^{-C_1} \sqrt{\lambda})$, so we may replace $F_j$ by $F'_j$ in the left-hand side of \eqref{bod}.

Next, by performing a Fourier series expansion in $y$ and using \eqref{inter}, we may decompose
$$ \tilde \psi_j(y,\xi) = \sum_{k \in \Z^d} e^{2\pi i yk^T / \sqrt{\lambda}} \tilde \psi_{j,k}(\xi) $$
for $y \in B^d(0,A^{-C_1} \sqrt{\lambda})$ and $\xi \in B^{d-1}(\xi_j,3\eps)$, where the Fourier coefficients $\tilde \psi_{j,k}(\xi)$ obey the pointwise bounds
\begin{equation}\label{mong}
 \tilde \psi_{j,k}(\xi) \lesssim A^{O(1)} \langle k \rangle^{-10d^3}.
\end{equation}
In particular there is no difficulty justifying convergence of the series in $k$.  As a consequence, we may decompose $F'_j$ into extension operators:
$$ F'_j(y) = \sum_{k \in \Z^d} e^{2\pi i yk^T / \sqrt{\lambda}} {\mathcal E}_j( \tilde \psi_{j,k} f'_j )(y)$$
and hence we have the pointwise bound
$$ |\prod_{j \in [d]} F'_j| \leq \sum_{k_1,\dots,k_d \in \Z^d} \prod_{j \in [d]} |{\mathcal E}_j( \tilde \psi_{j,k_j} f'_j )|.$$
We can take $L^{2p}$ norms using \eqref{quasi} to conclude that
$$ \|\prod_{j \in [d]} F'_j \|_{L^{2p}(B^d(x, A^{-C_1} \sqrt{\lambda}))}
\leq (\sum_{k_1,\dots,k_d \in \Z^d} \| \prod_{j \in [d]} {\mathcal E}_j( \tilde \psi_{j,k_j} f'_j ) \|_{L^{2p}(B^d(x, A^{-C_1} \sqrt{\lambda}))}^{\min(2p,1)})^{1/\min(2p,1)}.$$
Applying Theorem \ref{lmrc} and \eqref{mong}, we conclude that
$$ \|\prod_{j \in [d]} F'_j \|_{L^{2p}(B^d(x, A^{-C_1} \sqrt{\lambda}))} \lesssim A^{O(C_1)} \eps^{-O(1)} \prod_{j \in [d]} \|f'_j \|_{L^2(\R^{d-1})}.$$
To conclude the proof of the proposition, it suffices by \eqref{norma} to show that
\begin{equation}\label{ham}
 \|f'_j \|_{L^2(\R^{d-1})} \lesssim  \| F_j(\cdot,0) \|_{L^2(B^{d-1}(0, \sqrt{\lambda}))} + A^{O(1)} \eps^{-O(1)} \lambda^{-20d^2}
\end{equation}
(say).  From \eqref{fjt} one has
$$ F_j(y,0) = \check f_j(\xi) + \int_{\R^{d-1}} e^{2\pi i y' \xi^T} (\tilde \psi_j((y',0),\xi)-1) f_j(\xi)\ d\xi.$$
Meanwhile, from \eqref{fjr}, \eqref{ej-def} and Plancherel's theorem one has
\begin{equation}\label{fjp}
\|f'_j \|_{L^2(\R^{d-1})} = \| \check f_j \chi_{r_j} \|_{L^2(\R^{d-1})} = E_j(r_j)^{1/2}.
\end{equation}
From the triangle inequality we conclude that
$$ \|f'_j \|_{L^2(\R^{d-1})} \leq \| F_j(\cdot,0) \chi_r \|_{L^2(\R^{d-1})}
+ \| \chi_{r_j}(y') \int_{\R^{d-1}} e^{2\pi i y' \xi^T} (\tilde \psi_j((y',0),\xi)-1) f_j(\xi)\ d\xi \|_{L^2(\R^{d-1})}$$
where the final $L^2$ norm is with respect to the $y'$ variable.  From the rapid decay of $\chi_r$ outside of $B^d(0,2r)$ and \eqref{norma} we have
$$ \| F_j(\cdot,0) \chi_r \|_{L^2(\R^{d-1})}\lesssim  \| F_j(\cdot,0) \|_{L^2(B^{d-1}(0, \sqrt{\lambda}))} + A^{O(1)} \eps^{-O(1)} \lambda^{-20d^2}
$$ 
so to establish \eqref{ham}, it suffices to show that
\begin{align*}
&\left\| \chi_{r_j}(y') \int_{\R^{d-1}} e^{2\pi i y' \xi^T} (\tilde \psi_j((y',0),\xi)-1) f_j(\xi)\ d\xi \right\|_{L^2(\R^{d-1})}\\
&\quad \lesssim A^{-C_1/2} E_j(r_j + \lambda^{0.49})^{1/2}  + A^{O(1)} \eps^{-O(1)} \lambda^{-20d^2}.
\end{align*}
From \eqref{tpsi-norm} and the fundamental theorem of calculus we have
$$ \tilde \psi_j((y',0),\xi)-1 = \int_0^1 ((y'\nabla^T_{y'}) \tilde \psi_j)((uy',0),\xi)\ du$$
where $y' \nabla^T_{y'}$ is the directional derivative in the $y$ direction; thus by Minkowski's inequality it suffices to show that
$$ 
\| \chi_{r_j}(y') \int_{\R^{d-1}} e^{2\pi i y' \xi^T} \tilde \psi_{j,u}(y',\xi) f_j(\xi)\ d\xi \|_{L^2(\R^{d-1})}
\lesssim A^{-C_1/2 + O(1)} E_j(\tilde r_j)^{1/2}  + A^{O(1)} \eps^{-O(1)} \lambda^{-20d^2}$$
for all $u\in [0,1]$, where $\tilde r_j \coloneqq r_j + \lambda^{0.49}$ and
$$ \tilde \psi_{j,u}(y',\xi) \coloneqq ((y' \nabla_{y'}^T) \tilde \psi_j)((uy',0),\xi).$$
We can form a decomposition $f_j = \tilde f'_j + \tilde f''_j$ by replacing $r_j$ by $\tilde r_j$ in \eqref{fjr}, thus
$$ \tilde f'_j(\xi) \coloneqq f_j * \hat \chi_{\tilde r_j}(\xi) = \int_{B^{d-1}(0,\eps)} f_j(\xi-\eta) \hat \chi_{\tilde r_j}(\eta)\ d\eta
$$
and $\tilde f''_j \coloneqq f_j - \tilde f'_j$.  The same integration by parts argument that showed \eqref{hong} also shows (with minor modifications) that 
$$ 
\| \chi_{r_j}(y') \int_{\R^{d-1}} e^{2\pi i y' \xi^T} \tilde \psi_{j,u}(y',\xi) \tilde f''_j(\xi)\ d\xi \|_{L^2(\R^{d-1})}
\lesssim A^{O(1)} \eps^{-O(1)} \lambda^{-20d^2},$$
the key point being that up to negligible errors, $\chi_{r_j}$ and $1 - \chi_{\tilde r_j}$ are separated from each other by a distance $\gtrsim \lambda^{0.49}$.   Thus it will suffice to show that
$$ 
\| \chi_{r_j}(y') \int_{\R^{d-1}} e^{2\pi i y' \xi^T} \tilde \psi_{j,u}(y',\xi) \tilde f'_j(\xi)\ d\xi \|_{L^2(\R^{d-1})}
\lesssim A^{-C_1/2 + O(1)} E_j(\tilde r_j)^{1/2}.$$
By repeating the proof of \eqref{fjp} one has
$$ E_j(\tilde r_j)^{1/2} = \| \tilde f'_j \|_{L^2(\R^{d-1})}.$$
The function $\chi_{r_j}$ is extremely small (e.g., of size $O(\lambda^{-10d})$ outside of the ball $B^{d-1}(0, 2r_j)$, and bounded otherwise, so we may dominate it by $\eta'_{2r_j}$ up to negligible error.  By squaring, it then suffices to show that
\begin{equation}\label{schur} 
\int_{\R^{d-1}} \int_{\R^{d-1}} K(\xi_1,\xi_2) \tilde f'_j(\xi_1) \overline{\tilde f'_j(\xi_2)}\ d\xi_1 d\xi_2
\lesssim A^{-C_1+O(1)} \|\tilde f'_j \|_{L^2(\R^{d-1})}^2
\end{equation}
where the kernel $K(\xi_1,\xi_2)$ is given by
$$
K(\xi_1,\xi_2) \coloneqq \int_{\R^{d-1}} \eta'_{2r_j}(y')^2 e^{2\pi i y' (\xi_1-\xi_2)^T} \tilde \psi_{j,u}(y',\xi_1) \overline{\tilde \psi_{j,u}(y',\xi_2)}\ dy'.$$
For $\xi \in B^{d-1}(\xi_j^0,3\eps)$ and $y' \in B^{d-1}(0,3r_j)$, we have from \eqref{inter}
$$
 \nabla_{y'}^{\otimes m} \tilde \psi_{j,u}(y',\xi) \lesssim_m A^{-C_1/2 + O_m(1)} r_j^{-m} $$
for any $m$, if $M$ is sufficiently large depending on $m$.  From this and repeated integration by parts, we obtain the bounds
$$ K(\xi_1,\xi_2) \lesssim A^{-C_1+O(1)} r_j^{d-1} \rho_{1/r_j}(\xi_1-\xi_2) + A^{O(1)} \lambda^{-10d} $$
(say), and the claim \eqref{schur} now follows from Schur's test (or Young's inequality).
\end{proof}

Let $C_1$ be as in the above proposition; henceforth implied constants are allowed to depend on $C_1$.  From Fubini's theorem we have
$$
\| \prod_{j \in [d]} S_\lambda^{(j)} f_j \|_{L^{2p}(B^d(\lambda x_0,\eps \lambda))} 
\lesssim A^{O(1)} \lambda^{-\frac{d}{4p}}
\| \| \prod_{j \in [d]} S_\lambda^{(j)} f_j \|_{L^{2p}(B^d(x,A^{-C_1} \sqrt{\lambda}))} \|_{L^{2p}(B^d(\lambda x_0,2\eps \lambda))} $$
(where the outer $L^{2p}$ norm on the right-hand side is with respect to the $x$ variable), so to show \eqref{bash}, it suffices by Proposition \ref{local} to show that
$$
\| \prod_{j \in [d]} e_j \|_{L^{p}(B^d(\lambda x_0,2\eps \lambda))} \lesssim A^{O(1)} \eps^{-O(1)} \lambda^{\frac{d}{2p}}$$
where $e_j$ is the local energy density
\begin{equation}\label{ejo}
 e_j(x',x_d) \coloneqq \int_{\R^{d-1}} \eta'_{x',2\sqrt{\lambda}}( y') |F_j(y',x_d)|^2\ dy'.
\end{equation}
To estimate these densities, we perform a Gabor-type decomposition of $f_j$. Let $\varphi$ be the function defined in Section \ref{notation-sec}, then
$$ f_j(\xi) = \lambda^{\frac{d-1}{2}} \int_{\R^{d-1}} \varphi_{\zeta,\sqrt{\lambda}}(\xi)^2 f_j(\xi)\ d\zeta$$
and hence by Fourier inversion
$$ f_j(\xi) = \lambda^{\frac{d-1}{4}} \int_{\R^{d-1}} \int_{\R^{d-1}} e^{-2\pi i z \xi^T} \varphi_{\zeta,\sqrt{\lambda}}(\xi) g_j(\zeta,z)\ dz d\zeta$$
where
$$ g_j(\zeta,z) \coloneqq \lambda^{\frac{d-1}{4}} \int_{\R^{d-1}} e^{2\pi i z \xi^T} \varphi_{\zeta,\sqrt{\lambda}}(\xi) f_j(\xi)\ d\xi$$
is a Gabor-type transform of $f_j$.  Note that $g_j(\zeta,z)$ vanishes unless $\zeta \in B(\xi_j^0, 2\eps)$, and from Plancherel's theorem and \eqref{norma} one has
\begin{equation}\label{gabor}
\begin{split}
\int_{\R^{d-1}} \int_{\R^{d-1}} |g_j(\zeta,z)|^2\ dz d\zeta
&= 
\lambda^{\frac{d-1}{2}} \int_{\R^{d-1}} \int_{\R^{d-1}} \varphi_{\zeta,\sqrt{\lambda}}(\xi)^2 |f_j(\zeta)|^2\ d\xi d\zeta \\
&= \int_{\R^{d-1}} |f_j(\xi)|^2\ d\xi \\
&= 1.
\end{split}
\end{equation}
For any $x \in B^d(\lambda x_0, 3\eps \lambda)$, one has from \eqref{fspl} that
$$ F_j(x) = 
\lambda^{\frac{d-1}{4}} \int_{\R^{d-1}} \int_{\R^{d-1}} \int_{\R^{d-1}} e^{2\pi i (\lambda \Phi_j(\frac{x}{\lambda},\xi) - z \xi^T)} \psi_j(\frac{x}{\lambda},\xi)  \varphi_{\zeta,\sqrt{\lambda}}(\xi) g_j(\zeta,z)\ d \xi dz d\zeta$$
and hence by \eqref{ejo}
$$ e_j(x) = \int_{\R^{d-1}} \int_{\R^{d-1}} \int_{\R^{d-1}} \int_{\R^{d-1}} K_{j,x}( \zeta_1,z_1,\zeta_2,z_2) g_j(\zeta_1,z_1) \overline{g_j(\zeta_2,z_2)}  \ dz_1 d\zeta_1 dz_2 d\zeta_2$$
where the kernel $K_{j,x}$ is given by
\begin{align*}& K_{j,x}(\zeta_1,z_1,\zeta_2,z_2) \coloneqq \lambda^{\frac{d-1}{2}} \int_{\R^{d-1}} \int_{\R^{d-1}}
\int_{\R^{d-1}}\\
&\quad  e^{2\pi i \Sigma} \psi_j\left(\frac{(y',x_d)}{\lambda},\xi_1\right)  \varphi_{\zeta_1,\sqrt{\lambda}}(\xi_1) \overline{\psi_j}(\frac{(y',x_d)}{\lambda},\xi_2)  \varphi_{\zeta_2,\sqrt{\lambda}}(\xi_2)\ d\xi_1 d\xi_2 dy' 
\end{align*}
and $\Sigma$ is the phase
$$ \Sigma \coloneqq \lambda \Phi_j\left(\frac{(y',x_d)}{\lambda},\xi_1\right) - z_1 \xi_1^T - \lambda \Phi_j\left(\frac{(y',x_d)}{\lambda},\xi_2\right) + z_2 \xi_2^T.$$
Observe that on the support of the integrand, one has the derivative estimates
$$ \nabla_{\xi_1} \Sigma = \lambda \nabla_\xi \Phi_j( \frac{x}{\lambda}, \zeta_1 ) - z_1 + O( A^{O(1)} \lambda^{1/2} )$$
and
$$ \nabla_{\xi_2} \Sigma = -\lambda \nabla_\xi \Phi_j( \frac{x}{\lambda}, \zeta_2 ) + z_2 + O( A^{O(1)} \lambda^{1/2} )$$
and (by the submersion property)
$$ |\nabla_{y'} \Sigma| \gtrsim A^{-O(1)} |\zeta_1 - \zeta_2| - O( A^{O(1)} \lambda^{-1/2} )$$
After many integrations by parts using all the derivative bounds on $\Phi, \eta, \psi_j, \varphi$, we conclude the kernel bounds
$$ K_{j,x}(\zeta_1,z_1,\zeta_2,z_2) \lesssim A^{O(1)} \prod_{i=1,2} \rho_{\lambda^{-1/2}}( \nabla_\xi \Phi_j( \frac{x}{\lambda}, \zeta_i ) - \lambda^{-1} z_i )^2 \rho_{\lambda^{-1/2}}( \zeta_1 - \zeta_2 )$$
(say) and thus by Schur's test
$$ e_j(x) \lesssim A^{O(1)} \int_{\R^{d-1}} \int_{\R^{d-1}} \rho_{\lambda^{-1/2}}( \nabla_\xi \Phi_j( \frac{x}{\lambda}, \zeta ) - \lambda^{-1} z )^2 |g_j(\zeta,z)|^2\ d\zeta dz.$$

If we set $\Omega_j \coloneqq \R^{d-1} \times \R^{d-1}$ (parameterised by $(\zeta,z)$) with measure
$$ d\mu_j \coloneqq |g_j(\zeta,z)|^2\ d\zeta dz$$
and for each $x \in B^{d-1}(x_0, 3\eps)$ we let $\phi_j[x]\colon \Omega_j \to \R^{d-1}$ denote the map
$$ \phi_j[x](\zeta,z) \coloneqq \nabla_\xi \Phi_j( x, \zeta ) - \lambda^{-1} z$$
then after rescaling $x$ by $\lambda$, we conclude that
\begin{equation}\label{sp}
\| \prod_{j \in [d]} e_j \|_{L^{p}(B^d(\lambda x_0,2\eps \lambda))} \lesssim A^{O(1)} \lambda^{\frac{d}{p}}
\| \prod_{j \in [d]} \int_{\Omega_j} \rho_{\lambda^{-1/2}}( \phi_j[x](\omega_j)  )\ d\mu(\omega_j) \|_{L^p(B^{d-1}(x_0,2\eps))}.
\end{equation}
Set $t := \lambda^{-1/2}$.  We can bound
$$ \rho_{\lambda^{-1/2}}( \phi_j[x](\omega_j) ) \lesssim \sum_{k \in \Z^d} \langle k \rangle^{-10d^2} 
1_{B^{d-1}(0,t)}( \phi_j[x](\omega_j) - C^{-1} k \lambda^{-1/2} ) $$
for some constant $C$ depending only on $d$, and so by the quasi-triangle inequality \eqref{quasi} we can bound the right-hand side of \eqref{sp} by
$$
(\sum_{k_1,\dots,k_d \in \Z^d} (\langle k \rangle^{-10d^2}
\| \prod_{j \in [d]} \int_{\Omega_j} 1_{B^{d-1}(0,t)}( \phi_j[x](\omega_j) - C^{-1} k_j \lambda^{-1/2} ) \|_{L^p(B^{d-1}(x_0,2\eps))})^{\min(p,1)})^{1/\min(p,1)}.$$
By the hypotheses of Theorem \ref{lmrc-osc}, the maps $(x,\omega_j) \mapsto \phi_j[x](\omega_j) - C^{-1} k_j \lambda^{-1/2}$ obey the hypotheses of Theorem \ref{curv-kak} uniformly in $k_j$.  Applying that theorem and using the definition of $\eps$, we conclude that
$$
\| \prod_{j \in [d]} e_j \|_{L^{p}(B^d(\lambda x_0,2\eps \lambda))} \lesssim A^{O(1)} \eps^{-O(1)} \lambda^{\frac{d}{p}} t^{\frac{d}{p}} \prod_{j \in [d]} \mu_j(\Omega_j)$$
and the claim follows from \eqref{gabor} and the definition of $t$.

\appendix

\section{Erratum to a previous paper}\label{erratum}

In this appendix we disclose a small gap in the arguments in \cite{bct}, and specifically in the proof of \cite[Theorem 1.16]{bct} in the case $q/d < 1$ (which only occurs in four and higher dimensions $d \geq 4$).  We thank Jon Bennett for discussions regarding this issue, which was also independently discovered by Ciprian Demeter.

The issue lies with the proof of \cite[Lemma 2.2]{bct}.  In that paper, it is asserted that this lemma is proven in exactly the same fashion as \cite[Proposition 4.3]{tvv}.  While one of the two implications in this lemma is not problematic, in the other implication, an application of H\"older's inequality is used to control an integral by its $L^{q/d}$ norm, and this is only justified when $q/d \geq 1$.  This is not an issue in \cite{tvv}, or in \cite{bct} in dimensions up to three, but creates a gap in four and higher dimensions.

There are several resolutions to this problem.  If one is willing to impose sufficient amounts of regularity on the hypersurface beyond $C^2$, one can use the more complicated arguments in \cite[\S 6]{bct} (which were explicitly designed to avoid the use of \cite[Lemma 2.2]{bct}), or the alternate proof of multilinear restriction in \cite{bejenaru}; one can also use Theorem \ref{lmrc} from the current paper.  If one insists on only assuming $C^2$ regularity, another fix was given in \cite[\S 4]{aspects}, in which the induction on scales was performed with the restriction constant ${\mathcal C}_{\operatorname{Rest}}(R)$ replaced by a related constant ${\mathcal C}_{\operatorname{Rest}}(R)$, and valid implication in \cite[Lemma 2.2]{bct} can be used to conclude.  An alternative fix (which is basically equivalent to the previous one) is to generalise \cite[Theorem 1.16]{bct} to a vector-valued setting, in which the functions $g_j$ are now assumed to take values in an arbitrary Hilbert space $H_j$, and the expression $\prod_{j=1}^d {\mathcal E}_j g_j$ in the left-hand side is now replaced by $\prod_{j=1}^d |{\mathcal E}_j g_j|$; similar modifications are made to the definition of ${\mathcal R}^*(2 \times \dots \times 2 \to q; \alpha)$ and to \cite[Lemma 2.2]{bct}.  One can verify that all of the arguments in \cite[\S 2]{bct} extend to this vector-valued setting.  Furthermore, one can now rectify the proof of \cite[Lemma 2.2]{bct} by using H\"older's inequality to control an integral by its $L^2$ norm rather than by its $L^{q/d}$ norm.  More specifically, in the notation of that lemma, suppose that each $f_j$ is supported on $A_j^R$, and takes values in some Hilbert space $H_j$.  The measure $\mu_j$ on $\R^d$ defined by
$$ \int_{\R^d} F(x)\ d\mu_j(x) \coloneqq R^{-(d-1)} \int_{B^d(0,C/R)} \int_U F(\Sigma_j(x) + t)\ dx dt $$
obeys the pointwise bound $d\mu_j \gtrsim dx$ on $A_j^R$ if the constant $C$ is large enough, and hence one can find a bounded weight $w_j \in L^\infty(\R^d \to \R)$ such that
$$ \int_{A_j^R} F(x)\ dx \coloneqq R^{d-1} \int_{B^d(0,C/R)} \int_U F(\Sigma_j(x) + t) w(\Sigma_j(x)+t) \ dx dt.$$
As a consequence, we have the identity
$$ \hat f_j(\xi) = R^{d-1} \int_{B^d(0,C/R)} e^{i \xi \cdot t} {\mathcal E}_j f_{j,t}(\xi)\ dt$$
where $f_{j,t}\colon U \to H_j$ is the function
$$ f_{j,t}(x) \coloneqq f_j(\Sigma(x)+t) w(\Sigma(x)+t).$$
By Cauchy-Schwarz, we conclude that
$$ \hat f_j(\xi) \lesssim R^{-1} \left( R^{-d} \int_{B^d(0,C/R)} |{\mathcal E}_j f_{j,t}(\xi)|^2\ dt\right)^{1/2}.$$
We can write this as
$$ \hat f_j(\xi) \lesssim R^{-1} |{\mathcal E}_j F_j(\xi)|$$
where $F_j$ takes values in the Hilbert space $H_j \otimes L^2(B^d(0,C/R), R^{-d}\ dt)$ and is defined as
$$ F_j(\xi) \coloneqq (f_{j,t}(\xi))_{t \in B^d(0,C/R)}.$$
Applying the hypothesis ${\mathcal R}^*(2 \times \dots \times 2 \to q, \alpha)$, we conclude that
$$ \| \prod_{j =1}^d |\hat f_j| \|_{L^{q/d}(B(0,R))} \lesssim R^{\alpha-d} \prod_{j=1}^d \| F_j \|_2.$$
On the other hand, from the Fubini-Tonelli theorem one can verify that
$$  \| F_j \|_2 \lesssim R^{1/2} \| f_j \|_2$$
and the claim follows.

One can in fact be able to directly establish the equivalence of the scalar and vector-valued versions of the multilinear restriction theorem by a standard Khinchin inequality argument of Marcinkiewicz and Zygmund \cite{mz} (after first using a limiting argument to reduce to the case of finite-dimensional Hilbert spaces $H_j$), at least in the regime $q/d \leq 2$ which is the case of most interest in applications; we leave the details to the interested reader; in particular, this can be used to recover the second implication in \cite[Lemma 2.2]{bct}.  It is also not difficult to verify that the proof of Theorem \ref{lmrc} extends without difficulty to the vector-valued setting after making the obvious changes.


\begin{thebibliography}{10}

\bibitem{bejenaru}
I. Bejenaru, \emph{The multilinear restriction estimate: a short proof and a refinement}, Math. Res. Lett. \textbf{24} (2017), no. 6, 1585--1603.

\bibitem{aspects}
J. Bennett, \emph{Aspects of multilinear harmonic analysis related to transversality}, Harmonic analysis and partial differential equations, 1--28, Contemp. Math., 612, Amer. Math. Soc., Providence, RI, 2014. 

\bibitem{bct}
J. Bennett, A. Carbery, T. Tao, \emph{On the multilinear restriction and Kakeya conjectures}, Acta Math. \textbf{196} (2006), no. 2, 261--302. 

\bibitem{borg}
J. Bourgain, \emph{Estimates for cone multipliers}, Operator Theory: Advances and Applications, \textbf{77} (1995), 41--60.

\bibitem{bd}
J. Bourgain, C. Demeter, \emph{The proof of the $\ell^2$ decoupling conjecture}, Ann. of Math. (2) \textbf{182} (2015), no. 1, 351--389.

\bibitem{bd-study}
J. Bourgain, C. Demeter, \emph{A study guide for the $\ell^2$ decoupling theorem}, Chin. Ann. Math. Ser. B \textbf{38} (2017), no. 1, 173--200.

\bibitem{bdg}
J.  Bourgain, C. Demeter, L. Guth, \emph{Proof of the main conjecture in Vinogradov's mean value theorem for degrees higher than three}, Ann. of Math. \textbf{184} (2016), no. 2, 633--682.

\bibitem{bg}
J. Bourgain, L. Guth, \emph{Bounds on Oscillatory Integral Operators Based on Multilinear Estimates
}, Geom. Func. Anal. \textbf{21} (2011), 1239--1295.

\bibitem{cv}
A. Carbery, S. Valdimarsson, \emph{The endpoint multilinear Kakeya theorem via the Borsuk-Ulam theorem}, J. Funct. Anal. \textbf{264} (2013), no. 7, 1643--1663. 

\bibitem{guth}
L. Guth, \emph{The endpoint case of the Bennett-Carbery-Tao multilinear Kakeya conjecture}, Acta Math. \textbf{205} (2010), no. 2, 263--286.

\bibitem{guth-easy}
L. Guth, \emph{A short proof of the multilinear Kakeya inequality}, Math. Proc. Cambridge Philos. Soc. \textbf{158} (2015), no. 1, 147--153


\bibitem{loomis}
L. H. Loomis, H. Whitney, \emph{An inequality related to the isoperimetric inequality}, Bull. Amer. Math. Soc \textbf{55}, (1949). 961--962. 

\bibitem{mz}
J. Marcinkiewicz, A. Zygmund, \emph{Quelquels Inequalites pour les Operations Lineaires}, Fund.
Math. \textbf{32} (1939), 115-121; Reprinted in: J. Marcinkiewicz, Collected Papers, PWN, Warszawa

\bibitem{stein:small}
E. M. Stein, Singular integrals and differentiability properties of functions, Princeton University Press, 1970.

\bibitem{tao-br}
T. Tao, \emph{The Bochner-Riesz conjecture implies the restriction conjecture}, Duke Math. J. \textbf{96}, 363--376 (1999).


\bibitem{tao}
T. Tao, \emph{Endpoint bilinear restriction theorems for the cone, and some sharp null form estimates}, Math. Z. \textbf{238} (2001), no. 2, 215--268.

\bibitem{TV}
T. Tao, A. Vargas, \emph{A bilinear approach to cone multipliers I. Restriction theorems}, Geom. Func. Anal. \textbf{10} (2000), 185--215.

\bibitem{tvv}
T. Tao, A. Vargas, L. Vega, \emph{A bilinear approach to the restriction and Kakeya conjectures}, J. Amer. Math. Soc. \textbf{11} (1998), 967--1000.
\end{thebibliography}
\end{document}